\documentclass[11pt]{amsart}
\usepackage{amssymb}
\usepackage{amsthm}

\reversemarginpar
\newcommand{\margin}[1]{%
   \marginpar{\small\sffamily\raggedright #1}}

\newcommand{\B}[1]{\ensuremath{\mathsf{#1}}} 
\newcommand{\C}[1]{\ensuremath{\mathcal{#1}}}
\newcommand{\ch}{\operatorname{char}}

\newcommand{\Gal}{\operatorname{Gal}}

\newcommand{\iso}{\cong}
\newcommand{\notdivides}{\!\not\hspace*{2.6pt}\bigm|}

\newcommand{\thmref}[1]{Theorem~\ref{#1}}
\newcommand{\lemref}[1]{Lemma~\ref{#1}}
\newcommand{\propref}[1]{Proposition~\ref{#1}}
\newcommand{\secref}[1]{Section~\ref{#1}}

\newcommand{\tabref}[1]{Table~\ref{#1}}

\newtheorem{thm}{Theorem}[section] 
\newtheorem{prop}[thm]{Proposition}

\newtheorem{lem}[thm]{Lemma}

\theoremstyle{remark}
\newtheorem{rem}[thm]{Remark}

\numberwithin{equation}{section}  

\allowdisplaybreaks

\begin{document}

\title{Semisimple Group (and Loop) Algebras over Finite Fields$^*$}
\author{Raul A. Ferraz \and Edgar G. Goodaire \and C\'esar Polcino Milies}
\address{Instituto de Matem\'atica e Estat\'{\i}stica \\
Universidade de S\~ao Paulo, Caixa Postal 66.281 \\
CEP 05314-970, S\~ao Paulo SP \\
Brasil}
\email{raul@ime.usp.br}
\address{Memorial University of Newfoundland\\
St. John's, Newfoundland\\
Canada A1C 5S7}
\email{edgar@mun.ca}
\thanks{The second author wishes to thank FAPESP
of Brasil and the Instituto de Matem\'atica e Estat\'{\i}stica of
the Universidade de S\~ao Paulo for their support and hospitality.}
\address{Instituto de Matem\'atica e Estat\'{\i}stica \\
Universidade de S\~ao Paulo, Caixa Postal 66.281 \\
CEP 05314-970, S\~ao Paulo SP \\
Brasil}
\email{polcino@ime.usp.br}
\thanks{This research was supported by a Discovery Grant
from the Natural Sciences and Engineering Research
Council of Canada, by FAPESP,
Proc. 2004/15319-3 and 2008/57553-3, and by CNPq., Proc. 300243/79-0(RN)
of Brasil.
\thanks{$^*$This is a considerably more detailed version of a paper entitled ``Some Classes
of Semisimple Group
(and Loop) Algebras over Finite Fields'' that has been provisionally
accepted for publication in the \emph{Journal of Algebra}.}
\\ \today}


\begin{abstract}
We determine the structure of the semisimple group algebra of certain groups
over the rationals and over those
finite fields where the Wedderburn decompositions have the least number of simple components.
We apply our work to obtain similar information about the loop
algebras of indecomposable RA loops and to produce
negative answers to the isomorphism problem over various fields.
\end{abstract}

\maketitle

\section{Introduction}\label{sec:intro}
This work was initially motivated by connections with coding theory.
Historically, the first class of codes that became prominent
were \emph{linear} in the sense that the code words form a subspace of a finite
dimensional vector space over a finite field.   An important subclass
of linear codes are \emph{cyclic}, that is, whenever $\alpha=(\alpha_0,\alpha_1,\ldots,\alpha_{n-1})$
is a code word, so is the cyclic shift $\alpha'=(\alpha_{n-1},\alpha_0,\alpha_1\ldots, \alpha_{n-2})$.
If one identifies $\alpha$ with the element $\sum_{i=0}^{n-1}\alpha_i a^i$
of the group algebra $KG$ of the cyclic group $G=\langle a\rangle$ of order $n$
over the field $K$, then a cyclic code is precisely an ideal of $KG$, the
cyclic shift $\alpha\mapsto\alpha'$ corresponding to multiplication
by $a$ in the group algebra.  Thus, cyclic codes can be studied entirely
in the language of group algebras, an observation that gave birth to the idea of a
\emph{group code}, which is simply an ideal in a group algebra.
Since any ideal of a semisimple finite dimensional
group algebra is generated by an idempotent, that is, it has the form $(KG)e$
for some $e=e^2$, and since all idempotents are sums of idempotents that
generate the simple components of $KG$ in the Wedderburn
decomposition, it becomes important to find these simple components.

In several cases that have been studied to date,
the minimal number of simple
components in the Wedderburn decomposition of $KG$, with $K$ a finite field,
was equal to the number of simple components in the decomposition of the rational
group algebra $\B{Q}G$ so that the principal idempotents can be computed
easily, from the subgroup structure of $G$.
Since the groups that appear in this paper
are $2$-groups,
the number of simple components in the decomposition of $KG$ is almost never the same
as the corresponding number for $\B{Q}G$
(the former number is usually greater) \cite[Corollary 2]{Ferraz:09},
so the obvious question becomes just what the minimal number is.  One would
expect this to depend on the field, and it does.

In a final section of this paper, we apply our work
to alternative loop algebras.
Some details of the
structure of the rational loop algebras of indecomposable
RA loops are rediscovered and extended.
Specifically, we consider the semisimple structure of $KL$ with $K$ a finite field
that produces the minimal number of simple components.
We also exhibit counterexamples
to the isomorphism problem over the rationals and finite fields.
\smallskip

The groups $G$ of interest in this paper are such that
$G/\C{Z}(G)\iso C_2\times C_2$,
where $\C{Z}(G)$ denotes the centre of $G$.  These groups first arose in work of
the second author \cite{EGG:83} who was studying Moufang loops whose loop rings
are alternative.  Since then, they have appeared in many other contexts related
to involutions.  They appear in work concerned with commutativity and anticommutativity
of symmetric elements in group rings  \cite{EGG:09c,Cristo:07,Jespers:05,Jespers:06}
and also regarding Lie
properties of these elements \cite{Giambruno:09,Sehgal:09a,Sehgal:09b}.

It can be shown that there are just five classes of such groups
that are \emph{indecomposable}.  Moreover, any group
with $G/\C{Z}(G)\iso C_2\times C_2$ is the direct product of an
indecomposable group with an abelian group.  Here is a description of these classes.
\begin{align*}
\C{D}_1 &: \langle  x,y,t_1 \mid x^2 = y^2 = t_1^{2^m} = 1,m\ge1  \rangle  \\
\C{D}_2 &: \langle  x,y,t_1 \mid x^2 = y^2 = t_1, t_1^{2^m} = 1,  m\ge1 \rangle  \\
\C{D}_3 &: \langle  x,y,t_1,t_2 \mid x^2 = t_1^{2^{m_1}} = t_2^{2^{m_2}} =  1,
y^2 = t_2, m_1,m_2\ge1\rangle  \\
\C{D}_4 &: \langle  x,y,t_1,t_2 \mid x^2 = t_1, y^2 = t_2,
t_1^{2^{m_1}} = t_2^{2^{m_2}} = 1, m_1,m_2\ge1 \rangle  \\
\C{D}_5 &: \langle  x,y,t_1,t_2,t_3 \mid \\
&  \qquad \qquad x^2 =t_2, y^2 = t_3,
t_1^{2^{m_1}} = t_2^{2^{m_2}} = t_3^{2^{m_3}} = 1, m_1,m_2,m_2\ge1  \rangle.
\end{align*}
Each group in any of these classes is a $2$-group
and, in each case, $(x,y)=s=t_1^{2^{m_1-1}}$ is a unique
nontrivial commutator.  The centre of a group in any class is the direct product
of the cyclic subgroups generated by those of $t_1,t_2,t_3$ which
appear.  We refer the reader to Chapter~V of \cite{EGG:96} for details.

\section{The Number of Simple Components of $KG$}
Let $G$ be a finite group of exponent $e$ and
let $K$ be a field with $\ch K\notdivides |G|$.  Let $\zeta$
denote a primitive $e$th root of $1$.  For each $\sigma\in\Gal(K(\zeta),F)$,
we have $\sigma(\zeta)=\zeta^r$ for some positive integer $r$, so maps
of the type $g\mapsto g^r$ corresponding to such $\sigma$ define an action
on $G$.  Two conjugacy classes of $G$ are said
to be \emph{$K$-conjugate} if they correspond under this action.
This notion of $K$-conjugacy is an equivalence relation on the conjugacy classes of $G$
and the corresponding equivalence classes are called \emph{$K$-classes}.  For $x\in G$,
we use
the notation $\C{C}_K(x)$ for the $K$-conjugate class of the conjugacy
class of $x$.  These classes are
important because
a theorem of Witt-Berman \cite[Theorems 21.5 and 21.25]{Curtis:88},
later proven by R.~Ferraz entirely in group ring theoretical terms
with no appeal to character theory \cite{Ferraz:04},
says that the number of simple components of the semisimple group algebra
$KG$ is the number of $K$-classes of $G$.

The groups $G$ that appear in this paper, from the classes
$\C{D}_1,\ldots,\C{D}_5$ described in the previous section,
all are generated by elements of orders a power of $2$ and a central commutator
subgroup $G'=\{1,s\}$ of order two.  It follows that conjugacy classes
are singletons $\{a\}$ if $a\in\C{Z}(G)$ and of the form $\{w,sw\}$ if $w\notin\C{Z}(G)$.
The Galois group $\Gal(\B{Q}(\zeta),\B{Q})$ is the group $\C{U}=\C{U}(\B{Z}_e)$
of units of the integers mod~$e$, so the \B{Q}-class corresponding
to a central element $a$ is the set $\{a,a^3,a^5,\ldots, \}$
of odd powers of $a$.  Notice that this set
is precisely the set of all generators of the cyclic subgroup $\langle a\rangle$
generated by $a$.
Thus the number of \B{Q}-classes of the centre
of $G$ is the number of cyclic subgroups of $\C{Z}(G)$.
On the other hand, the \B{Q}-class containing a noncentral element $w\in G$, that is,
the \B{Q}-class of the conjugacy class $\{w,sw\}$, is
\begin{equation*}
\C{C}_\B{Q}(w)=\{w,w^3,w^5,\ldots, \} \cup\{sw, sw^3, sw^5,\ldots\}.
\end{equation*}

Let $q$ be a power of an odd prime and let $K$ be the field of order $q$.
The Galois group $\Gal(K(\zeta),K)$ is generated by the Frobenius automorphism
$\zeta\mapsto \zeta^q$, so the $K$-class corresponding to a central element
$a\in G$ is $\{a, a^q,a^{q^2}, a^{q^3},\ldots\}$ and the $K$-class
containing a noncentral element $w$ is
\begin{equation*}
\C{C}_K(w)=\{w,w^q,w^{q^2},w^{q^3},\ldots\} \cup \{sw, sw^q,sw^{q^2},sw^{q^3},\ldots\}.
\end{equation*}
Much of this paper is devoted to determining the number of \B{Q}- and $K$-classes
of various groups.  In this task, the following lemmas
will be a big help.

\begin{lem}\label{lem:critical} For any integer $r\ge3$ and any odd integer $q$,
the maximal order of $q\pmod{2^r}$ is $2^{r-2}$.  Any $q\equiv3\pmod{8}$ has
this maximal order.
\end{lem}
\begin{proof}
It is known that $\C{U}_{2^r}$,
the unit group of $\B{Z}_{2^r}$, is $C_2\times C_{2^{r-2}}$ and that any integer
$q\equiv5\pmod{2^r}$ has maximal order $2^{r-2}$ \cite[Proposition 4.2.2]{Ireland:82},
so the import of this Lemma is that any $q\equiv3\pmod8$ has this same order.
First, we argue that $q$ need only be $5\pmod8$ to have maximal order in $\C{U}_{2^r}$.
To see this, let $A=\{5^{2i+1}\mid 1\le 5^{2i+1}< 2^r\}$
be the set of generators of $\langle \overline{5}\rangle$ in $\C{U}_{2^r}$
and let $B=\{q\in\B{Z}\mid 1\le q<2^r,\; q\equiv5\pmod{8}\}$.
Since $5^{2i+1}=25^i(5)\equiv5\pmod{8}$, we have $A\subseteq B$.
Now $|A|=\tfrac12(2^{r-2})=2^{r-3}$ because $5$ has order $2^{r-2}\pmod{2^r}$.
On the other hand, there is a unique integer congruent to $5\pmod8$
in each of the intervals $[1,8)$, $[8,16)$, $[16,24)$, \ldots, $[2^3(2^{r-3}-1),2^r)$
so $|B|=2^{r-3}$ too.  Thus $A=B$ showing, indeed, that any $q\equiv5\pmod{8}$
has order $2^{r-2}\pmod{2^r}$.  In particular, any  $q\equiv-3\pmod{8}$ has this
order modulo $2^r$, and hence so does any $q\equiv3\pmod8$ because
because $q^{2^\ell}=(-q)^{2^\ell}$ for $\ell\ge1$.
\end{proof}

\begin{lem}\label{lem1} Let $G$ be an abelian group,
let $K$ be a field of odd order $q$ and let $a\in G$ are generated by elements of order $2^r$, $r\ge0$.
\begin{enumerate}
   \item[(i)] If $r=0$ or $1$, then $\C{C}_K(a)=\C{C}_\B{Q}(a)$.
   \item[(ii)] If $r=2$ and $q\equiv 1\pmod4$, then $\C{C}_\B{Q}(a)$ splits
into two $K$-classes.  If $r=2$ and $q\equiv3\pmod4$, then $\C{C}_\B{Q}(a)=\C{C}_K(a)$.
   \item[(iii)] If $r\ge 3$, then $\C{C}_\B{Q}(a)$ splits into at least two $K$-classes,
and exactly two when $q\equiv3\pmod8$.
\end{enumerate}
\end{lem}
\begin{proof} (i) If $r=0$, then $a=1$ and the statement is trivial.
If $r=1$, then $a^k=a$ for all odd integers $k$, so
$\C{C}_K(a)=\C{C}_\B{Q}(a)=\{a\}$.
\smallskip\par (ii)  Now assume that $r=2$.  Then $\C{C}_\B{Q}(a)=\{a,a^3\}$.
If $q\equiv1\pmod4$, then $a^q=a$ and $(a^3)^q=a^3$ because $a$ has order $4$,
so $\C{C}_K(a)=\{a\}$ and $\C{C}_K(a^3)=\{a^3\}$.  Clearly $\C{C}_\B{Q}(a)$
splits into two $K$-classes.  On the other hand, if $q\equiv3\pmod4$,
then $a^q=a^3$ and $a^{q^2}=a$, so $\C{C}_\B{Q}(a)=\C{C}_K(a)$.
\smallskip\par (iii) Assume $r\ge3$.
Then $\C{C}_\B{Q}(a)=\{a,a^3,\ldots,a^{2^r-1}\}$ and $\C{C}_K(a)=\{a,a^q,\ldots,a^{q^{k-1}}\}$,
where $k$ is the smallest positive integer such that $a^{q^k}=a$.  Thus
$q^k\equiv1\pmod{2^r}$ with $k$ minimal; that is, $k=|\langle\overline{q}\rangle|$
in $\C{U}_{2^r}$.
This implies $k\le2^{r-2}$ by \lemref{lem:critical}, so $k<2^{r-1}$ and there exists
an element $a^v\in \C{C}_\B{Q}(a)\setminus \C{C}_K(a)$.
As before, we compute $\C{C}_K(a^v)=\{a^v,a^{vq},\ldots,a^{vq^{\ell -1}}\}$,
where $\ell$ is the smallest positive integer such that $a^{vq^\ell}=a^v$.  Thus
$vq^\ell\equiv v\pmod{2^r}$ and, since $v$ is odd, $q^\ell\equiv1\pmod{2^r}$.
So $\ell=|\langle\overline{q}\rangle|=k$. It follows that the $K$-class of $a$ has
size the order of $q\mod{2^r}$ and that $\C{C}_\B{Q}(a)$ is the union of
$[\C{U}_{2^r}\colon\langle\overline{q}\rangle ]$ $K$-classes.

When $q\equiv3\pmod8$,
$q$ has maximal order $2^{r-2}$ by \lemref{lem:critical}.  Since $\C{U}_{2^r}$ has
order $2^{r-1}$, $\C{C}_\B{Q}(a)$ splits into precisely two classes
when $q\equiv3\pmod8$ and otherwise, into more than two classes.
\end{proof}

\begin{rem}
Throughout this paper, we use $F$ consistently to denote a finite field of
an order $q\equiv3\pmod8$ and use \lemref{lem1} to justify claims
that for any finite field $K$, the number of $K$-classes of a certain group---that is, the number of
simple components in the Wedderburn decomposition of $KG$---is minimal when $K=F$.
\end{rem}

Since cyclic groups are very prominent in this work, it is useful to state
as a proposition a fact that will be often needed in the sequel.  This lemma
allows us to give a short proof of a result that also appears in \cite{Pruthi:01}.

\begin{prop}\label{prop2} Let $A=\langle t\rangle$ be a cyclic group of order $2^m$
and let $K$ be any field of characteristic different from $2$.
If $m=1$, there are two $K$-classes in $A$.
If $m>1$, the number of $K$-classes of $A$ is $m+1$ if $K=\B{Q}$ and at least
$2m-1$ if $K$ is finite of odd order $q$.  This minimal number is achieved if
$q\equiv3\pmod8$.
\end{prop}
\begin{proof}
Clearly $\C{C}_K(1)=\{1\}$ for any $K$.
If $m=1$, then $A=\{1,t\}$ and $\C{C}_K(t)=\{t\}$ for any $K$ because
$t^v=t$ for any odd $v$.  Assume $m>1$.
If $i$ is odd, then $t^i$ belongs to $\C{C}_{\B{Q}}(t)$, the \B{Q}-conjugate
class of $t=t^{2^0}$.  If $i$ is even, write $i=2^{i_0}i_1$ with $i_1$ odd.  Then for any odd $j$,
$(t^i)^j=t^{2^{i_0}i_1j}$ with $i_1j$ odd.  Thus $t^i\in\C{C}_{\B{Q}}(t^{2^{i_0}})$.
It follows that the \B{Q}-classes of $A$ consist of the odd powers of $t^{2^i}$,
$i=0,1,2\ldots,m$, and here they are.
\begin{align}
\C{C}_0 &=\C{C}_{\B{Q}}(t)= \{t, t^3, t^5, \ldots, t^{2^m-1}\} \notag \\
\C{C}_1 &=\C{C}_{\B{Q}}(t^2)= \{t^2, t^6, t^{10}, \ldots, t^{2(2^{m-1}-1)}\} \notag \\
\C{C}_2 &=\C{C}_{\B{Q}}(t^4)= \{t^4, t^{12}, t^{20},\ldots, t^{4(2^{m-2}-1)})\} \notag \\
\vdots \notag \\
\C{C}_i &=\C{C}_{\B{Q}}(t^{2^i})= \{t^{2^i},t^{2^i(3)},t^{2^i(5)},\ldots,t^{2^i(2^{m-i}-1)}\} \label{eq2} \\
\vdots \notag \\
\C{C}_{m-3} &= \C{C}_{\B{Q}}(t^{2^{m-3}})=\{t^{2^{m-3}},t^{2^{m-2}(3)}, t^{2^{m-2}(5)}, t^{2^{m-2}(7)}\} \notag \\
\C{C}_{m-2} &= \C{C}_{\B{Q}}(t^{2^{m-2}}) =\{t^{2^{m-2}},t^{2^{m-2}(3)}\}=\{t^{2^{m-2}},st^{2^{m-2}}\}  \notag \\
\C{C}_{m-1} &= \C{C}_{\B{Q}}(t^{2^{m-1}})=\C{C}(s)=\{s\}, \quad\text{where we have set $s=t^{2^{m-1}}$,} \notag \\
\C{C}_m &= \C{C}_{\B{Q}}(t^{2^m})=\{1\}. \notag
\end{align}
The last two \B{Q}-classes,
$\C{C}_{m-1}=\{s\}$ and $\C{C}_m=\{1\}$, are clearly $K$-classes for any field $K$
and, if $q\equiv3\pmod{8}$, the same holds for $\C{C}_{m-2}$ (by \lemref{lem1})
because $t^{2^{m-2}}$ has order $2^2$.
On the other hand, \lemref{lem1} also says that if $q\equiv3\pmod8$,
all the other \B{Q}-classes split
into precisely two $K$-classes, giving a total of $2(m-2)+3=2m-1$ $K$-classes in all.
\end{proof}

The next result will be useful in counting $K$-classes of non-central elements.

\begin{lem}\label{lem:xclasses} Let $x$ be an element of $G$ with $x\notin\C{Z}(G)$ and
$H$ be the subgroup of $G$ generated by $\C{Z}(G)$ and $x$.
Let $K$ be any field with $\ch K\notdivides |G|$.
Then the number of $K$-classes corresponding to conjugacy classes of
the type $\{ax,sax\}$, $a\in \C{Z}(G)$, is the number of $K$-classes
of $H/\langle s\rangle$ less the number of $K$-classes of $\C{Z}(G)/\langle s\rangle$.
\end{lem}
\begin{proof}
Let $\pi$ denote the canonical projection $H\to H/\langle s\rangle$.
The correspondence $\pi\colon \{ax,sax\}\mapsto \{\overline{ax}\}$ is one-to-one,
so the number of $K$-classes of $G$ corresponding to conjugacy classes of
the type $\{ax,sax\}$ is the number of $K$-classes of $H/\langle s\rangle$
corresponding to elements of the type $\overline{ax}$, and this is the
total number of $K$-classes in $H/\langle s\rangle$ less the number
in which $x$ does not appear, which is precisely the number of $K$-classes of $\C{Z}(G)/\langle s\rangle$.
\end{proof}

\section{Groups of type $\C{D}_1$} \label{secd1}
Let
\begin{equation*}
G=\langle t,x,y\mid t^{2^m}=1, x^2=y^2=1, t \text{ central}, (x,y)=t^{2^{m-1}}=s\rangle
\end{equation*}
be a group of type $\C{D}_1$.
As with all groups in this paper, it is important to
note that $G$ has a two-element commutator subgroup $G'$
generated by a central element $s$ of order $2$ and so conjugacy classes
of noncentral elements are of the form $\{w,sw\}$.
The elements of $G$ can be written uniquely in the form
$t^ix^\delta y^\epsilon$, $i=0,1,2,\ldots, 2^m-1$, $\delta,\epsilon=0,1$.
Thus $G$ has order $4(2^m)$.

When $m=1$, $G=D_4$ and $KD_4=K\oplus K\oplus K\oplus K\oplus M_2(K)$
for any field $K$ (of characteristic different from $2$).

When $m=2$, $G$ is $16\Gamma_2b$ in the Hall and Senior
notation \cite{Hall:64},
$\C{Z}(G)=\langle t\rangle = \{1,t,t^2,t^3\}$ and $s=t^2$.
As in \propref{prop2},
there are three \B{Q}-classes,
\begin{equation*}
\C{C}_\B{Q}(1) = \{1\}, \quad \C{C}_\B{Q}(t) = \{t,t^3\} \quad\text{ and }\quad \C{C}_\B{Q}(t^2) = \{t^2\}.
\end{equation*}
For any finite field $K$ of odd order $q$, the first and third classes are also $K$-classes,
while $\C{C}_\B{Q}(t)$ splits into two $K$-classes if $q\equiv1\pmod4$ (but not otherwise).
The conjugacy classes of the form $\{ax,sax\}$, $a\in\C{Z}(G)$ are
\begin{equation*}
\{x,sx\}=\{x,t^2x\} \quad\text{ and } \quad \{tx,stx\}=\{tx,t^3x\}.
\end{equation*}
Each of these is a $K$-class for any (odd) $q$.  The situation is identical
for conjugacy classes of the form $\{ay,say\}$ and $\{axy,saxy\}$, $a\in\C{Z}(G)$.
In all, we have $3+3(2)=9$ \B{Q}-classes and nine $K$-classes if $q\equiv3\pmod8$.
For other $q$, it matters to us only
that the number of $K$-classes is more than $9$.  Thus $\B{Q}G$ is
the direct sum of nine simple algebras while $KG$ is the sum
of at least nine simple algebras, exactly nine when
$q\equiv3\pmod8$.

In this paragraph, let $K$ denote any field of characteristic different from $2$ and let $G$
be a group from any of the classes $\C{D}_1,\ldots, \C{D}_5$.  It is
basic to an understanding of this paper that $K$
is an algebra direct sum, $KG\iso K[G/G']\oplus \Delta(G,G')$,
with $K[G/G']$ a direct sum of fields (we call this the ``commutative part''
of $KG$) and $\Delta(G,G')$ an ideal which is the direct sum of quaternion algebras
(we call $\Delta(G,G')$ the ``noncommutative part'' of $KG$) some of which might
be rings of $2\times2$ matrices.
(See, for example, \cite[Proposition 3.6.7]{Milies:02}.)
With $G$ a group in $\C{D}_1$ and $m=2$, $G/G'=\langle \overline{t},\overline{x},\overline{y}\rangle
\iso C_2\times C_2\times C_2$, an abelian group with eight cyclic subgroups and hence eight
\B{Q}-classes.  Each, being a singleton, is also a $K$-class.
Thus $K[G/G']$ is the direct sum of eight copies of $K$,
the commutative part of $KG$ is the direct sum of eight copies of $K$, and
$KG$ is the direct sum of eight copies of $K$ and at least one quaternion algebra,
precisely one when $K=\B{Q}$ or when $K$ is a finite field of order $q\equiv3\pmod8$.
\medskip

Now assume that $m\ge3$ (and $G$ is a group in $\C{D}_1$) and remember that
we always use $F$ to denote a finite field of order $q\equiv3\pmod8$.
The centre of $G$ is $\C{Z}(G)=\langle t\rangle\iso C_{2^m}$,
so \propref{prop2} says that
$\C{Z}(G)$ contains $N_1=m+1$ \B{Q}-classes and $M_1=2m-1$ $F$-classes.

In this and the remaining sections of this paper, we will consistently use
\begin{center}
\begin{tabular}{lp{288pt}}
$N_1,M_1$ & to denote the number of \B{Q}- and $F$-classes, respectively, of the center of a group  $G$,\\[6pt]
$N_2,M_2$ & for the number of \B{Q}- and $F$-classes, respectively, that correspond to conjugacy classes of the sort $\{ax,sax\}$, $a\in \C{Z}(G)$, \\[6pt]
$N_3,M_3$ & for the number of \B{Q}- and $F$-classes, respectively, that correspond to conjugacy classes of the sort $\{ay,say\}$, $a\in \C{Z}(G)$, \\[6pt]
$N_4,M_4$ & for the number of \B{Q}- and $F$-classes, respectively, that correspond to conjugacy classes of the sort $\{axy,saxy\}$, $a\in \C{Z}(G)$, and\\[6pt]
$N_0,M_0$ & for the number of \B{Q}-classes and $F$-classes, respectively, of $G/G'$.
\end{tabular}
\end{center}
\medskip

We also implicitly use the fact that the number of cyclic
subgroups in a direct product of the sort $C_2\times H$ is
twice the number of cyclic subgroups in~$H$.

Let $H=\langle \C{Z}(G),x\rangle=\langle t,x\rangle$.
Then $H/\langle s\rangle=\langle\overline{t},\overline{x}\rangle\iso C_{2^{m-1}}\times C_2$
contains $2m$ cyclic subgroups, so $2m$ \B{Q}-classes.  Six of these correspond to cyclic
subgroups generated by elements of order at most $4$, so, using \lemref{lem1}, we see
that six \B{Q}-classes are also $F$-classes while the remaining $2m-6$ \B{Q}-classes
split into two $F$-classes.  In all, there are $2(2m-6)+6=4m-6$ $F$-classes in $H/\langle s\rangle$.
Now $\C{Z}(G)/\langle s\rangle\iso C_{2^{m-1}}$
contains $m$ \B{Q}-classes, three of which correspond to subgroups generated by elements
of order at most $4$.  This implies $2(m-3)+3=2m-3$ $F$-classes.
By \lemref{lem:xclasses}, there are $N_2=2m-m=m$ \B{Q}-classes
corresponding to conjugacy classes of the form $\{ax,sax\}$---for brevity, we refer
to these as  conjugacy classes ``involving $x$''---and $M_2=(4m-6)-(2m-3)=2m-3$ $F$-classes
involving $x$.  The calculations and numbers for conjugacy classes involving $y$
and involving $xy$ are identical.   In all, we obtain $N=N_1+3N_2=(m+1)+3m=4m+1$ \B{Q}-classes
and $M=M_1+3M_2=(2m-1)+3(2m-3)=8m-10$ $F$-classes.

Since $G/G'\iso C_{2^{m-1}}\times C_2\times C_2$, for any field $K$ of characteristic
different from $2$,
$K[G/G']$ can be regarded as the group algebra of a cyclic group of order
$2^{m-1}$ with coefficients in $K[C_2\times C_2]=K\oplus K\oplus K\oplus K$.
Applying this observation in the case that $K=\B{Q}$, \propref{prop2} tells us that
$\B{Q}C_{2^{m-1}}$ has $(m-1)+1=m$ simple components, so $\B{Q}[G/G']$
has $4m$ simple components.  Thus the commutative part of $\B{Q}G$, which is
$\B{Q}[G/G']$, is the direct sum of $4m$ fields.   Since there are just $4m+1$ simple
components in all, $\B{Q}G$ has just one (necessarily quaternion) noncommutative component
\cite[Corollary VI.4.8]{EGG:96}.

Similar reasoning shows that
$F[G/G']$ is the direct sum of $4(2m-3)=8m-12$ fields, so
$\Delta(G,G')$ is the direct sum of $(8m-10)-(8m-12)=2$
quaternion algebras.

\begin{thm}\label{thm:d1} Let $G$ be a group of type $\C{D}_1$, $m\ge2$.
Then $\B{Q}G$ is the direct sum of $4m$ fields and one quaternion algebra.
Let $K$ be a finite field of odd order $q$.  Then
the Wedderburn decomposition of $KG$ contains
at least $8m-10$ simple components, the minimal number being achieved
if $q\equiv3\pmod8$ in which case the decomposition of $KG$
consists of $8m-12$ fields and two quaternion algebras,
each necessarily a ring of $2\times2$ matrices.
\end{thm}

\section{Groups of type $\C{D}_2$}
Let
\begin{equation*}
G=\langle x,y,t\mid t^{2^{m}}=1, x^2=y^2=t, t\text{ central}, (x,y)=s=t^{2^{m-1}}\rangle
\end{equation*}
be a group of type $\C{D}_2$.
Then $G$ has order $4(2^m)$.  If $m=1$,
then $G=Q_8$ and, for any field $K$ of characteristic not $2$,
$KQ_8\iso K\oplus K\oplus K\oplus K\oplus \B{H}$ where
$\B{H}$ is a quaternion  algebra over $K$.

Let $m=2$.  ($G$ is $16\Gamma_2d$ in the Hall and Senior notation.)
Since $\C{Z}(G)=\langle t\rangle = \{1,t,t^2,t^3\}$ and $s=t^2$
as in the previous section,
there are three \B{Q}-classes which are also $K$-classes when $K$ has order $q\equiv3\pmod8$;
otherwise there are four $K$-classes.
Since $x^3=tx$, $x^5=t^2x=sx$ and $x^7=stx$, there is just one \B{Q}-class involving $x$,
$\{x,tx, sx, stx\}$,
and this is also a $K$-class if $K$ has order $q\equiv3\pmod8$.
Since $y^2=t$ too, the situation is identical for $y$.  Since $(xy)^2=xyxy=sxxyy=st^2=1$,
$\{xy,sxy\}$ and $\{txy,stxy\}$ are $K$-classes for any field $K$.
In all, we have  $3+1+1+2=7$ $K$-classes with $K=\B{Q}$
or a finite field of order $q\equiv3\pmod8$.
Here, $G/G'\iso C_2\times C_4$ has six cyclic subgroups and so six \B{Q}-classes.
Since each corresponds to an element of order at most $4$, each \B{Q}-class
is a $K$-class when $K$ has finite order $q\equiv3\pmod8$---see \lemref{lem1}.
It follows that $KG$ is the direct sum of six fields and one quaternion
algebra whether $K=\B{Q}$ or $K$ has order $q\equiv3\pmod8$.  For other finite
fields, there are more than seven simple components in the Wedderburn
decomposition of $KG$.

We now assume that $m\ge3$ and denote by $F$ a finite field of order $q\equiv3\pmod8$.
The centre of $G$ is cyclic of order $2^m$, so there are $N_1=m+1$ \B{Q}-classes
and $M_1=2m-1$ $F$-classes, by \propref{prop2}.

Let $H=\langle \C{Z}(G),x\rangle=\langle x\rangle$, so $H/\langle s\rangle \iso C_{2^m}$
has $m+1$ \B{Q}-classes and $2m-1$ $F$-classes while $\C{Z}(G)/\langle s\rangle\iso C_{2^{m-1}}$
has $m$ \B{Q}-classes and $2(m-1)-1=2m-3$ $F$-classes.  By \lemref{lem1}, there
is just $1=m+1-m$ \B{Q}-class involving $x$---we consistently call this number $N_2$---and
there are $M_2=(2m-1)-(2m-3)=2$ such $F$-classes.
The numbers $N_3,M_3$, respectively, are the same for~$y$.

Let $H=\langle \C{Z}(G),xy\rangle$ and notice that $(xy)^2=sx^2y^2=st^2$
so that $\langle xy\rangle\cap\C{Z}(G)\ne\{1\}$.  Nevertheless,
it is not hard to see that $H/\langle s\rangle\iso C_{2^{m-1}}\times C_2$, so this group
has $2m$ \B{Q}-classes.  Of these, six correspond to subgroups
generated by elements of order
at most $4$ and hence which do not split, while the other $2m-6$ \B{Q}-classes split
into two $F$-classes.  With the help of \lemref{lem1}, we see
that there are $2(2m-6)+6=4m-6$ $F$-classes in $H/\langle s\rangle$.  As shown before,
$\C{Z}(G)/\langle s\rangle$ has $m$ \B{Q}-classes and $2m-3$ $F$-classes so, by \lemref{lem1}, there
are $N_4=2m-m=m$ \B{Q}-classes involving $xy$ and $M_4=(4m-6)-(2m-3)=2m-3$ $F$-classes
involving $xy$.
In all, we have $N=N_1+N_2+N_3+N_4=(m+1)+1+1+m=2m+3$ \B{Q}-classes and
$M=M_1+M_2+M_3+M_4=(2m-1)+2+2+(2m-3)=4m$ $F$-classes.

For groups of type $\C{D}_2$, $G/G'=\langle \overline{x},\overline{y}\rangle$.
As previously, we must be careful, noting that $G/G'\ne\langle\overline{x}\rangle\times\langle\overline{y}\rangle$
(because $\langle\overline{x}\rangle\cap\langle\overline{y}\rangle$).
Nonetheless, it is not hard to see that
$G/G'\iso C_{2^m}\times C_2$ (for instance,
$G/G'=\langle \overline{x}\rangle\times\langle \overline{y^{-1}x}\rangle$).
Thus, for any field $K$ (of characteristic not $2$),
the group algebra $K[G/G']$ is the group algebra of $C_{2^m}$ with
coefficients in $KC_2=K\oplus K$.  By \propref{prop2}, $\B{Q}C_{2^m}$
is the direct sum of $m+1$ fields, so $\B{Q}[G/G']$ is the direct sum
of $N_0=2(m+1)=2m+2$ fields.  Similarly, $F[G/G']$ is the direct sum of $M_0=2(2m-1)=4m-2$ fields.
Since $\B{Q}G$ has just $2m+3$ simple components in all, while $FG$ has $4m+2$,
the following theorem summarizes the work of this section.

\begin{thm}\label{thm:d2} Let $G$ be a group of type $\C{D}_2$ with $m\ge3$.
Then $\B{Q}G$ is the direct sum of $2m+2$ fields and one quaternion algebra.
Let $K$ be a finite field of order $q$.  Then $KG$ has at least $4m+2$
components in its Wedderburn decomposition.  If
$q\equiv3\pmod{8}$, this minimal number is achieved:  In this case,
$KG$ is the direct sum of $4m-2$ fields and two quaternion algebras
(each necessarily a ring of $2\times2$ matrices).
\end{thm}

\section{Groups of type $\C{D}_3$}
The following easy lemma is useful in determining the number of cyclic subgroups of
the centre of groups of type $\C{D}_3$
which, remember, is precisely the number of \B{Q}-classes of the
centre.

\begin{lem}\label{lem:2cyclics} Let $A=C_{2^a}\times C_{2^b}$
be the direct product of cyclic groups of orders $2^a$ and $2^b$
and assume $a\ge b\ge1$. Then the number of cyclic subgroups of $A$ of order $2^k$ is
\begin{itemize}
\item $3(2^{k-1})$  if $1\le k\le b$ and
\item $2^{b}$ if $b< k\le a$.
\end{itemize}
In all, $A$ has $N=2^b(3+a-b)-2$ cyclic subgroups (and hence $N$ \B{Q}-classes).
\end{lem}
\begin{proof}
We begin with the observation
that an element $(u,v)\in A$ satisfies $(u,v)^{2^k}=1$ if and only if $u^{2^k}=v^{2^k}=1$
so, for $1\le k\le b$, such an element lies in a subgroup of $A$ isomorphic
to $C_{2^k}\times C_{2^k}$.  In number, there are $n_k=|C_{2^k}\times C_{2^k}|=2^{2k}$
such elements.  Since also $n_{k-1}= 2^{2(k-1)}$,
the number of elements of order $2^k$ in $A$ is $2^{2k}-2^{2(k-1)}=3(2^{2(k-1)})$.  Since
$\phi (2^k)=2^{k-1}$, the number of cyclic subgroups of order $2^k$ is
$3(2^{2(k-1)})/\phi(2^k)=3(2^{k-1})$ as asserted.

If $b<k\le a$, an element $(u,v)\in A$ has order $2^k$
if and only if $u$ has order $2^k$. Since there are $2^{k-1}$ elements of order $2^k$ in
$C_{2^a}$, the number of elements of this order in $A$ is $2^{k-1}2^b$
and the number of cyclic subgroups of this order is $2^{k-1}2^{b}/\phi (2^k)=2^b$.

It follows that the total number of cyclic subgroups of  $A$ is
\begin{equation*}
1 + 3(1+2+ \cdots +2^{b-1}) + 2^{b}(a-b)
     =  2^b(3+a-b)-2. \hspace{12pt}\qed
\end{equation*}
\renewcommand{\qed}{}
\end{proof}
\medskip

In order to expedite many of the calculations
in this paper, we record in \tabref{tab1} some specific instances of this lemma.
In checking this
table, note that \lemref{lem:2cyclics} provides the number of \B{Q}-classes
while \lemref{lem1}, which tells us which of these split into (two) $F$-classes and
which do not, allows us to determine the number of $F$-classes.  Specifically,
if there are $\alpha$ \B{Q}-classes, $\beta$ of which correspond to cyclic groups
generated by an element of order at most $4$, then there are
$\beta+2(\alpha-\beta)=2\alpha-\beta$ $F$-classes.

Now let
\begin{equation*}
G=\langle x,y,t_1,t_2\mid
t^{2^{m_1}}=t^{2^{m_2}}=1, x^2=1, y^2=t_2, (x,y)=s=t_1^{2^{m_1-1}}\rangle
\end{equation*}
be a group from the class $\C{D}_3$.

We first consider four particular small cases.

\smallskip
\noindent (i) If $m_1=m_2=1$, then $|G|=16$, $\C{Z}(G)=C_2\times C_2$ and $s=t_1$.  The centre
has one element of order $1$ and three elements of order $2$, each defining
a conjugacy class of size $1$  and hence, for any field $K$ (of characteristic
different from $2$), a $K$-class: $N_1=4$.
Set $H=\langle \C{Z}(G),x \rangle=\langle t_1,t_2,x\rangle$.
Then $H/\langle s\rangle\iso C_2\times C_2$ has four $K$-classes and
$\langle \C{Z}(G)\rangle/\langle s\rangle=\langle t_2\rangle\iso C_2$
has two $K$-classes.  \lemref{lem:xclasses} says there are $N_2=4-2=2$ classes
involving $x$.  Let $H=\langle\C{Z}(G),y\rangle=\langle t_1,y\rangle$.
The group $H/\langle s\rangle=\langle\overline{y}\rangle\iso C_4$ has
three cyclic subgroups, each generated by an element of order at most $4$.
So there are three \B{Q}-classes and, with a field of order $q\equiv3\pmod8$,
these are also $F$-classes, by \lemref{lem1}.  By \lemref{lem:xclasses},
there is $N_3=3-2=1$ \B{Q}-class involving $y$ and one $F$-class:  $M_3=1$ too.
Since $(xy)^2=sy^2=st_2$, the situation with $xy$ is as with $y$:
$N_4=M_4=1$.  In all, there are $N=N_1+N_2+N_3+N_4=4+2+1+1=8$ \B{Q}-classes
and eight $F$-classes in $G$.  The Wedderburn decomposition of $KG$ has
eight simple components, whether $K=\B{Q}$ or $K=F$.

Now
$G/G'=\langle\overline{t_2},\overline{x},\overline{y}\rangle=\langle\overline{x},\overline{y}\rangle$
since $y^2=t_2$, so $G/G'=\langle \overline{x}\rangle\times \langle\overline{y}\rangle
\iso C_2\times C_4$,
and we can view $K[G/G']$ as the group algebra of $C_4$ with coefficients in $KC_2\iso K\oplus K$,
with $K$ any field (of characteristic not $2$).
With $m=2$ in \propref{prop2},  we see that $KC_4$ is the direct sum of $m+1=2m-1=3$ fields, so
$K[G/G']$ is the direct sum of six fields.
This being the commutative part of $KG$, it now follows
that whether $K=\B{Q}$ or $F$, $KG$ is the direct sum of six fields and two quaternion algebras.

\medskip
\noindent (ii.a)  Suppose $m_1=2$ and $m_2=1$.  Then $|G|=32$, $\C{Z}(G)\iso C_4\times C_2$
and $s=t_1^2$.
Let $K=\B{Q}$ or a finite field of order $q\equiv3\pmod8$.
Since $\C{Z}(G)$ contains six cyclic subgroups, each generated by an element of order at most $4$,
the centre of $G$ contains $N_1=6$ \B{Q}-classes, each of which
is also a $K$-class.

With $H=\langle\C{Z}(G),x\rangle=\langle t_1,t_2,x\rangle$, we have
$H/\langle s\rangle=\langle\overline{t_1},\overline{t_2},\overline{x}\rangle
\iso C_2\times C_2\times C_2$ and
$\C{Z}(G)/\langle s\rangle=\langle\overline{t_1},\overline{t_2}\rangle\iso C_2\times C_2$.
The groups $H/\langle s\rangle$ and $\C{Z}(G)/\langle s\rangle$ have eight and four $K$-classes
respectively so, from \lemref{lem:xclasses}, we see that $G$ has $N_2=8-4=4$
$K$-classes involving $x$.
If we take $H=\langle\C{Z}(G),y\rangle=\langle t_1,y\rangle$, then
$H/\langle s\rangle\iso C_2\times C_4$ has six cyclic subgroups
while $\C{Z}(G)/\langle s\rangle \iso C_2\times C_4$ has f
our.  All subgroups
are generated by elements of order at most $4$ and so correspond to \B{Q}-classes
which do not split when $K$ has order $q\equiv3\pmod8$.
By \lemref{lem:xclasses}, the number of $K$-classes containing $y$ is $N_3=6-4=2$.
Similarly, there are $N_4=2$ $K$-classes containing $xy$, so the total number of $K$-classes is
$N=N_1+N_2+N_3+N_4=6+2+2+2=14$.
Thus, $KG$ is the direct sum of $14$ simple algebras.

Since $G/G'=\langle \overline{t_1},\overline{x},\overline{y}\rangle
\iso C_2\times C_2\times C_4$, we have
$K[G/G']\iso (K\oplus K\oplus K\oplus K)C_4$.  Since $KC_4$ is the
direct sum of three fields, we see that $K[G/G']$ (the commutative part of $KG$) is the direct sum of
$N_0=12$ fields.  We conclude that for $K=\B{Q}$ or $K=F$,
$KG$ is the direct sum of $N_0=12$ fields and $N-N_0=2$ quaternion algebras.

\smallskip\noindent (ii.b)  Suppose  $m_1\ge3$ and $m_2=1$.
Then $\C{Z}(G)=C_{2^{m_1}}\times C_2$ contains $2(m_1+1)$ cyclic subgroups, six
generated by elements of order at most $4$,
so the centre of $G$ contains $N_1=2m_1+2$ \B{Q}-classes.  Those generated by elements
of order at most $4$
are also $F$-classes while the remaining $\B{Q}$-classes each split into two $F$-classes
giving in all, $M_1=6+2(2m_1+2-6)=4m_1-2$ $F$-classes.

Set $H=\langle\C{Z}(G),x\rangle = \langle t_1,t_2,x\rangle$.
Then $H/\langle s \rangle=\langle\overline{t_1},\overline{t_2},\overline{x}\rangle
\iso C_{2^{m_1-1}}\times C_2\times C_2$ contains $4m_1$ cyclic
subgroups twelve generated by elements of order at most $4$.
These are also $F$-classes
while the remaining $4m_1-12$ each split into two $F$-classes.
So we have $4m_1$ \B{Q}-classes
and $12+2(4m_1-12)=8m_1-12$ $F$-classes.  The group $\C{Z}(G)/\langle s\rangle
\iso C_{2^{m_1-1}}\times C_2$ has $2m_1$ cyclic subgroups, six generated by elements of order at
most $4$, so these numbers imply $2m_1$ \B{Q}-classes and $6+2(2m_1-6)=4m_1-6$
$F$-classes.
By \lemref{lem:xclasses}, there are $N_2=4m_1-2m_1=2m_1$ \B{Q}-classes involving $x$
and $M_2=(8m_1-12)-(4m_1-6)=4m_1-6$ $F$-classes.

With $H=\langle \C{Z}(G),y \rangle = \langle t_1,y\rangle$, we have
$H/\langle s\rangle\iso C_{2^{m_1-1}}\times C_4$ with $4[3+(m_1-1)-2]=4m_1-2$
cyclic subgroups (\lemref{lem:2cyclics} or \tabref{tab1}), ten
of which are generated by elements of order at most $4$.
This implies $4m_1-2$ \B{Q}-classes and $10+2(4m_1-2-10)=8m_1-14$
$F$-classes.  The group $\langle \C{Z}(G) \rangle/\langle s\rangle$ was considered above.
We obtain $N_3=(4m_1-2)-2m_1=2m_1-2$ \B{Q}-classes involving $y$
and $M_3=(8m_1-14)-(4m_1-6)=4m_1-8$ $F$-classes.  Since $(xy)^2=st_2$, the number of \B{Q}-classes
and $F$-classes involving $xy$ are the same as for $y$.
In all, we have $N=N_1+N_2+2N_3=(2m_1+2)+2m_1+2(2m_1-2)=8m_1-2$ \B{Q}-classes
and hence $8m_1-2$ simple components
in the Wedderburn decomposition of $\B{Q}G$
and $M=M_1+M_2+2M_3=(4m_1-2)+(4m_1-6)+2(4m_1-8)=16m_1-24$ $F$-classes, hence $16m_1-24$
simple components in the Wedderburn decomposition of $FG$.

Now $G/G'\iso C_{2^{m_1-1}}\times C_2\times C_4$ has $2(4m_1-2)=8m_1-4$ cyclic subgroups,
$20$ generated by elements of order at most $4$,
so the commutative part of $\B{Q}G$ is the direct sum of
$N_0=8m_1-4$ fields, while the commutative part of $FG$ is the direct sum
of $M_0=20+2(8m_1-4-20)=16m_1-28$ fields.
It follows that
$\B{Q}G$ is the direct sum of $N_0=8m_1-4$ fields and $N-N_0=(8m_1-2)-(8m_1-4)=2$ quaternion algebras,
while $FG$ is the direct sum of $M_0=16m_1-28$ fields
and $M-M_0=(16m_1-24)-(16m_1-28)=4$ quaternion algebras.

\medskip
\noindent (iii.a) Suppose $m_1=1$ and $m_2=2$.  Then $|G|=32$, $\C{Z}(G)\iso C_2\times C_4$ and $s=t_1$.
As in case (ii.a), there are $N_1=6$ \B{Q}-classes in the centre and these are also $F$-classes.
($F$ always denotes a finite field of order $q\equiv3\pmod8$.)
Set $H=\langle \C{Z}(G),x \rangle=\langle t_1,t_2,x\rangle$.
Then $H/\langle s\rangle=\langle\overline{t_2},\overline{x}\rangle\iso C_4\times C_2$,
so there are six
\B{Q}-classes which are also $F$-classes.
Since $\langle\C{Z}(G)\rangle/\langle s\rangle=\langle\overline{t_2}\rangle\iso C_4$
contains three, \lemref{lem:xclasses} gives $N_2=M_2=6-3=3$ $K$-classes involving $x$,
whether $K=\B{Q}$ or $K=F$.

With  $H=\langle\C{Z}(G),y\rangle=\langle t_1,y\rangle$,
$H/\langle s\rangle=\langle\overline{y}\rangle\iso C_8$ contains four \B{Q}-classes,
three of which are also $F$-classes, while one splits into two $F$-classes, giving
$2+3=5$ $F$-classes.  Since
$\C{Z}(G)/\langle s\rangle\iso C_4$ contains three \B{Q}-classes which are also $F$-classes,
there is just $N_3=4-3=1$ \B{Q}-class involving $y$
but there are $M_3=5-3=2$ $F$-classes involving $y$.
The calculations are the same for $xy$,
so the total number of \B{Q}-classes of $G$ is $N=N_1+N_2+2N_3=6+3+2(1)=11$ and
the number of $F$-classes is $M=M_1+M_2+2M_3=6+3+2(2)=13$.

Since $G/G'=\langle \overline{x}\rangle\times\langle \overline{y}\rangle\iso C_2\times C_8$,
$K[G/G']$ (for any field $K$ of characteristic different from $2$)
can be viewed as the group algebra of $C_8$ with
coefficients in $KC_2\iso\B{Q}\oplus\B{Q}$.  Since $\B{Q}C_8$ is the
direct sum of four fields (\propref{prop2}), the commutative part
of $\B{Q}G$ is the direct sum of $N_0=8$ fields, so
$\B{Q}G$ is the direct sum of eight fields and $N-N_0=3$
quaternion algebras.  Since
$FC_8$ is the direct sum of five fields (\propref{prop2}),
$F[G/G']$ is the direct sum of $M_0=10$ fields.
It follows that $FG$ is the
direct sum of ten fields and $M-M_0=3$ quaternion algebras,
each necessarily a ring of $2\times2$ matrices over a field.

\smallskip\noindent (iii.b)  Suppose $m_1=1$ and $m_2\ge3$. Then
$\C{Z}(G)=C_2\times C_{2^{m_2}}$ contains $2(m_2+1)=2m_2+2$
cyclic subgroups, six generated by elements of order at most
$4$, so the centre of $G$ contains $N_0=2m_2+2$ \B{Q}-classes
and $M_0=6+2(2m_2+2-6)=4m_2-2$ $F$-classes.

Set $H=\langle \C{Z}(G),x\rangle = \langle t_1,t_2,x\rangle$.
Then $H/\langle s \rangle=\langle\overline{t_2},\overline{x}\rangle \iso C_{2^{m_2}}\times C_2$
contains $2m_2+2$ cyclic
subgroups six generated by elements of order at most $4$.  These numbers imply $2m_2+2$ \B{Q}-classes
and $6+2(2m_2+2-6)=4m_2-2$ $F$-classes.  The group $\C{Z}(G)/\langle s\rangle
\iso C_{2^{m_2}}$ has $m_2+1$ cyclic subgroups, three of which are generated by
elements of order at most $4$, so this implies $m_2+1$ \B{Q}-classes and $3+2(m_2+1-3)=2m_2-1$
$F$-classes.
By \lemref{lem:xclasses}, there are $N_2=(2m_2+2)-(m_2+1)=m_2+1$ \B{Q}-classes involving $x$
and $M_2=(4m_2-2)-(2m_2-1)=2m_2-1$ $F$-classes.

If we take $H=\langle \C{Z}(G),y\rangle=\langle t_1,y\rangle$, then
$H/\langle s\rangle=\langle\overline{y}\rangle\iso C_{2^{m_2+1}}$
has $m_2+2$ cyclic subgroups, three
of which are generated by elements of order at most $4$.
This implies $m_2+2$ \B{Q}-classes and $3+2(m_2+2-3)=2m_2+1$
$F$-classes.  The group $\langle \C{Z}(G)\rangle/\langle s\rangle$ was considered above.
We obtain $N_3=(m_2+2)-(m_2+1)=1$ \B{Q}-class involving $y$
and $M_3=(2m_2+1)-(2m_2-1)=2$ $F$-classes.  Since $(xy)^2=st_2$, the number of \B{Q}-classes
and $F$-classes involving $xy$ are the same as for $y$.
In all, we have $N=N_1+N_2+2N_3=(2m_2+2)+(m_2+1)+2(1)=3m_2+5$ \B{Q}-classes
and hence $3m_2+5$ simple components
in the Wedderburn decomposition of $\B{Q}G$,
and $M=M_1+M_2+2M_3=(4m_2-2)+(2m_2-1)+2(2)=6m_2+1$ $F$-classes, that is, $6m_2+1$
simple components in the Wedderburn decomposition of $FG$.

Now $G/G'=\langle\overline{x},\overline{y}\rangle\iso C_2\times C_{2^{m_2+1}}$
has $2(m_2+2)=2m_2+4$ cyclic subgroups,
six generated by elements of order at most $4$,
so the commutative part of $\B{Q}G$ is the direct sum of
$N_0=2m_2+4$ fields, while the commutative part of $FG$ is the direct sum
of $M_0=6+2(2m_2+4-6)=4m_2+2$ fields.
It follows that
$\B{Q}G$ is the direct sum of $2m_2+4$ fields and $N-N_0=(3m_2+5)-(2m_2+4)=m_2+1$
quaternion algebras,
while $FG$ is the direct sum of $4m_2+2$ fields
and $M-M_0=(6m_2+1)-(4m_2+2)=2m_2-1$ quaternion algebras.

\medskip
\noindent (iv.a) Suppose $m_1=m_2=2$.  Then $|G|=64$ and $\C{Z}(G)\iso C_4\times C_4$
contains $10$ \B{Q}-classes that are also $F$-classes:  $N_1=M_1=10$.

With $H=\langle\C{Z}(G),x\rangle=\langle t_1,t_2,x\rangle$,
the group $H/\langle s\rangle\iso C_2\times C_4\times C_2$ contains $12$ \B{Q}-classes
(which are also $F$-classes)
while $ \C{Z}(G)/\langle s\rangle \iso C_2\times C_4$ contains six \B{Q}-classes
(which are also $F$-classes).  So there are $N_2=12-6=6$ \B{Q}-classes
and $M_2=6$ $F$-classes involving~$x$.

Taking $H=\langle \C{Z}(G),y \rangle=\langle t_1,y\rangle$,
we see that $H/\langle s\rangle\iso C_2\times C_8$ contains eight $\B{Q}$-classes
two of which split into two $F$-classes each, giving $6+2(8-6)=10$ $F$-classes.
There are
six \B{Q}-classes and six $F$-classes in $\C{C}(Z)(G)/\langle s\rangle$
so, by \lemref{lem:xclasses},
the number of $\B{Q}$-classes involving $y$ is $N_3=8-6=2$ and the number of $F$-classes is
$M_3=10-6=4$. The same holds for classes containing $xy$, so the total number of \B{Q}-classes in $G$
is $N=N_1+N_2+2N_3=10+6+2(2)=20$ and the
total number of $F$-classes is $M=M_1+M_2+2M_3=10+6+2(4)=24$.

In the present case,
$G/G'=\langle \overline{x}\rangle\times\langle\overline{y}\rangle\iso C_2\times C_8\times C_2$, so
$K[G/G']$ can be viewed as the group algebra of $C_8$ with
coefficients in $K[C_2\times C_2]\iso K\oplus K\oplus K\oplus K$
with $K$ any field of characteristic different from $2$.
With $m=3$, \propref{prop2} says that $\B{Q}C_8$ is the direct sum of four fields and $FC_8$
the direct sum of $2m-1=5$ fields.  It follows that the commutative part of $\B{Q}G$
is the direct sum of $N_0=16$ fields while the commutative part of $FG$
is the direct sum of $M_0=20$ fields.  We conclude that $\B{Q}G$ is the direct sum of
$16$ fields and $N-N_0=4$ quaternion algebras,  while $FG$ is the
direct sum of $20$ fields and $M-M_0=4$ quaternion algebras.

\smallskip
\noindent (iv.b) Suppose $m_1\ge3$ and $m_2=2$.  Then
$\C{Z}(G)\iso C_{2^{m_1}}\times C_4$ contains $N_1=4m_1+2$ \B{Q}-classes
and $M_1=8m_1-6$ $F$-classes.  (It is helpful to start referring to \tabref{tab1}.)

Let $H=\langle \C{Z}(G),x\rangle=\langle t_1,t_2,x\rangle$.  Then $H/\langle s\rangle
\iso C_{2^{m_1-1}}\times C_4\times C_2$ has $8(m_1-1)+4=8m_1-4$ \B{Q}-classes
and $16(m_1-1)-12=16m_1-28$ $F$-classes,
while $\C{Z}(G)/\langle s\rangle\iso C_{2^{m_1-1}}\times C_4$ has $4(m_1-1)+2=4m_1-2$
and $8(m_1-1)-6=8m_1-14$, respectively.  \lemref{lem:xclasses} gives
$N_2=(8m_1-4)-(4m_1-2)=4m_1-2$ \B{Q}-classes involving $x$ and
$M_2=(16m_1-28)-(8m_1-14)=8m_1-14$ $F$-classes.

Let $H=\langle \C{Z}(G),y\rangle=\langle t_1,y\rangle$.
Then $H/\langle s\rangle\iso C_{2^{m_1-1}}\times C_8$
has $8(m_1-1)-2=8m_1-10$ \B{Q}-classes and $16(m_1-1)-14=16m_1-30$ $F$-classes, giving
$N_3=(8m_1-10)-(4m_1-2)=4m_1-8$ \B{Q}-classes involving $y$ and
$M_3=(16m_1-30)-(8m_1-14)=8m_1-16$
$F$-classes.  Since $(xy)^2=st_2$, $\langle \C{Z}(G),xy\rangle=\langle t_1,xy\rangle$,
so the number of \B{Q}- and $F$-classes involving $xy$ are the same as those
involving $y$.  In all, we have $N=N_1+N_2+2N_3=16m_1-16$ \B{Q}-classes in $G$
and $M=M_1+M_2+2M_3=32m_1-52$ $F$-classes.

Now $G/G'=\langle \overline{t_1},\overline{x},\overline{y}\rangle
\iso C_{2^{m_1}}\times C_2\times C_8$
has $N_0=16(m_1-1)-4=16m_1-20$ \B{Q}-classes and
$M_0=32(m_1-1)-28=32m_1-60$ $F$-classes.  We conclude
that $\B{Q}G$ is the direct sum of $N_0=16m_1-20$ fields and $N-N_0=4$
quaternion algebras, while $FG$ is the direct sum of $M_0=32m_1-60$
fields and $M-M_0=8$ quaternion algebras.

\smallskip
\noindent (iv.c) Suppose $m_1=2$ and $m_2\ge3$.  Then
$\C{Z}(G)\iso C_4\times C_{2^{m_2}}$ contains $N_1=4m_2+2$ \B{Q}-classes
and $M_1=8m_2-6$ $F$-classes.

Let $H=\langle \C{Z}(G),x\rangle=\langle t_1,t_2,x\rangle$.  Then $H/\langle s\rangle
\iso C_4\times C_{2^{m_2}}\times C_2$ has $8m_2+4$ \B{Q}-classes and $16m_2-12$ $F$-classes
while $\C{Z}(G)/\langle s\rangle\iso C_2\times C_{2^{m_2}}$ has $2m_2+2$
and $4m_2-2$, respectively.  \lemref{lem:xclasses} gives
$N_2=(8m_2+4)-(2m_2+2)=6m_2+2$ \B{Q}-classes involving $x$ and
$M_2=(16m_2-12)-(4m_2-2)=12m_2-10$ $F$-classes.

Let $H=\langle \C{Z}(G),y\rangle=\langle t_1,y\rangle$.
Then $H/\langle s\rangle\iso C_2\times C_{2^{m_2+1}}$
has $2m_2+4$ \B{Q}-classes and $4m_2+2$ $F$-classes, giving
$N_3=(2m_2+4)-(2m_2+2)=2$ \B{Q}-classes involving $y$ and $M_3=(4m_2+2)-(4m_2-2)=4$
$F$-classes.  Since $(xy)^2=st_2$, $\langle \C{Z}(G),xy\rangle=\langle t_1,xy\rangle$,
so the number of \B{Q}- and $F$-classes involving $xy$ are the same as those
involving $y$.  In all, we have $N=N_1+N_2+2N_3=10m_2+8$ \B{Q}-classes in $G$
and $M=M_1+M_2+2M_3=20m_2-8$ $F$-classes.

Now $G/G'=\langle \overline{t_1},\overline{x},\overline{y}\rangle
\iso C_2\times C_2\times C_{2^{m_2+1}}$
has $N_0=4m_2+8$ \B{Q}-classes and $M_0=8m_2+4$ $F$-classes.  We conclude
that $\B{Q}G$ is the direct sum of $N_0=4m_2+8$ fields and $N-N_0=6m_2$
quaternion algebras, while $FG$ is the direct sum of $M_0=8m_2+4$
fields and $M-M_0=12m_2-12$ quaternion algebras.

\medskip
We turn to the general case, assuming now that $m_1\ge 3$ and $m_2\ge 3$.
Again, our argument involves subcases.

\medskip
\noindent (a) Suppose first that $m_1-1\ge m_2+1$.
\lemref{lem:2cyclics} tells us the $\C{Z}(G)\iso C_{2^{m_1}}\times C_{2^{m_2}}$
has $N_1=2^{m_2}(3+m_1-m_2)-2$ cyclic subgroups, which is the number of
\B{Q}-classes.  Those ten subgroups which lie in subgroups isomorphic to $C_4\times C_4$
are also $F$-classes, while the others split into two $F$-classes---see \lemref{lem1}.
We obtain $M_1=10+2(N_1-10)=2N_1-10$ $F$-classes in the centre.
\par\smallskip
Since $m_1-1>m_2$, \lemref{lem:2cyclics} tells us
that $\C{Z}(G)/\langle s\rangle \iso C_{2^{m_1-1}}\times C_{2^{m_2}}$ has
$\alpha=2^{m_2}(3+m_1-1-m_2)-2=2^{m_2}(2+m_1-m_2)-2$ cyclic subgroups,
of which $10$ (those lying in a copy of $C_4\times C_4$) are generated
by elements of order at most $4$.  It follows that the number
of \B{Q}-classes in $\C{Z}(G)/\langle s\rangle$ is $\alpha$ and the number
of $F$-classes is $10+2(\alpha-10)=2\alpha-10$.
Let $H=\langle \C{Z}(G),x \rangle=\langle t_1,t_2,x\rangle$. Then
$H/\langle s\rangle=\langle \overline{t_1}, \overline{t_2}, \overline{x}\rangle
\iso C_{2^{m_1-1}}\times C_{2^{m_2}}\times C_2$ has $2\alpha$ cyclic subgroups,
twenty of which are generated by elements of order at most $4$.  So $H/\langle s\rangle$
has $2\alpha$ \B{Q}-classes and $20+2(2\alpha-20)=4\alpha-20$ $F$-classes.
Using \lemref{lem:xclasses}, we obtain $N_2=2\alpha-\alpha=\alpha$ \B{Q}-classes
and $M_2=(4\alpha-20)-(2\alpha-10)=2\alpha-10=2N_2-10$ $F$-classes.

Set $H=\langle\C{Z}(G),y\rangle=\langle t_1,y\rangle\iso C_{2^{m_1}}\times C_{2^{m_2+1}}$.
Now $H/\langle s\rangle
\iso C_{2^{m_1-1}}\times C_{2^{m_2+1}}$.  Using \lemref{lem:2cyclics} and noting
$m_1-1\ge m_2+1$, we see that $H/\langle s\rangle$ has
$\beta=2^{m_2+1}[3+(m_1-1)-(m_2+1)]-2=2^{m_2+1}(1+m_1-m_2)-2$
cyclic subgroups, ten of which are generated by elements of order at most $4$.
This gives $\beta$ \B{Q}-classes and $10+2(\beta-10)=2\beta-10$ $F$-classes.
Since $\C{Z}/\langle s\rangle$ has $\alpha$ cyclic subgroups (as above),
\lemref{lem:xclasses} gives $N_3=\beta-\alpha=2^{m_2}(m_1-m_2)$ \B{Q}-classes
involving $y$ and $M_3=(2\beta-10)-(2\alpha-10)=2(\beta-\alpha)=2N_3$ $F$-classes.

Since $(xy)^2=st_2$, the counts for $xy$ are the same as for $y$:  $N_4=N_3$
and $M_4=M_3$.  In all, the total number of \B{Q}-classes in $G$,
which is the number of simple components
in the Wedderburn decomposition of $\B{Q}G$, is $N=N_1+N_2+2N_3=2^{m_2}(5+4m_1-4m_2)-4$.
Similarly, the Wedderburn decomposition of $FG$ has $M=M_1+M_2+2M_3=(2N_1-10)+(2N_2-10)+
2(2N_3)=2N-20=2^{m_2+1}(5+4m_1-4m_2)-28$ simple components.

Now $G/G'=\langle \overline{t_1},\overline{x},\overline{y}\rangle
\iso C_{2^{m_1-1}}\times C_2\times C_{2^{m_2+1}}$ has $2\beta$ cyclic subgroups,
twenty of which are generated by elements of order at most $4$.  This gives
$N_0=2\beta$ \B{Q}-classes and $M_0=20+2(2\beta-20)=4\beta-20=2N_0-20$ $F$-classes.
Remember that these are the number of fields in the commutative
parts of $\B{Q}G$ and $FG$, respectively.  It follows that $\B{Q}G$ is the direct sum
of $N_0=2\beta=2^{m_2+2}(1+m_1-m_2)-4$ fields and $N-N_0=2^{m_2}$ quaternion algebras
and $FG$ is the direct sum of $M_0=4\beta-20=2^{m_2+3}(1+m_1-m_2)-28$ fields
and $M-M_0=2^{m_2+1}$ quaternion algebras.

\smallskip
\noindent (b) Now suppose $m_1-1< m_2+1$.  We consider two possibilities.
\smallskip

\noindent (b.i) First, suppose that $m_1-1=m_2$.
The centre of $G$, which is isomorphic to $C_{2^{m_2+1}}\times C_{2^{m_2}}$,
has $\alpha=2^{m_2}(3+m_2+1-m_2)-2=4(2^{m_2})-2$ cyclic
subgroups, of which ten are generated by elements of order at most $4$.
It follows that the centre has $N_1=4(2^{m_2})-2$ \B{Q}-classes
and $M_1=10+2(N_1-10)=2N_1-10$ $F$-classes.

The group $\C{Z}(G)/\langle s\rangle\iso C_{2^{m_2}}\times C_{2^{m_2}}$ has
$\beta=2^{m_2}(3)-2$ cyclic subgroups, of which ten are
generated by elements of order at most $4$, so the number of \B{Q}-classes
in the centre is $\beta$ and the number of $F$-classes is
$10+2(\beta-10)=2\beta-10$.
Let $H=\langle \C{Z}(G),x\rangle =\langle t_1,t_2,x\rangle$.
Then $H/\langle s\rangle\iso C_{2^{m_2}}\times C_{2^{m_2}}\times C_2$
has $2\beta$ cyclic subgroups, of which twenty are generated by elements
of order at most $4$, so $H/\langle s\rangle$ has $2\beta$ \B{Q}-classes
and $20+2(2\beta-20)=4\beta-20$ $F$-classes.   From \lemref{lem:xclasses},
we learn that there are $N_2=2\beta-\beta=\beta=2^{m_2}(3)-2$
\B{Q}-classes involving $x$ and $M_2=(4\beta-20)-(2\beta-10)=2\beta-10=2N_2-10$
$F$-classes.

Let $H=\langle\C{Z}(G),y\rangle=\langle t_1,y\rangle$.
Then $H/\langle s\rangle\iso C_{2^{m_2}}\times C_{2^{m_2+1}}$ has
$\gamma=2^{m_2}(3+1)-2=2^{m_2}(4)-2$ cyclic subgroups, of which ten
are generated by elements of order $4$.  It follows that this group
has $\gamma$ \B{Q}-classes and $2\gamma-10$ $F$-classes.   \lemref{lem:xclasses}
gives $N_3=\gamma-\beta=2^{m_2}$ \B{Q}-classes involving $y$
and $M_3=(2\gamma-10)-(2\beta-10)=2(\gamma-\beta)=2N_3$ $F$-classes.
Since $(xy)^2=st_2$, the counts with $xy$ are the same as with $y$:
$N_4=N_3$ and $M_4=M_3$.
The rational group algebra $\B{Q}G$ is the direct sum of
$N=N_1+N_2+2N_3=2^{m_2}(9)-4$ simple algebras and the group algebra
$FG$ is the direct sum of $M=M_1+M_2+2M_3=2N-20=2^{m_2+1}(9)-28$ simple algebras.
In this situation, $G/G'\iso C_{2^{m_2}}\times C_2\times C_{2^{m_2+1}}$
has $N_0=2\gamma=2^{m_2}(8)-4$ \B{Q}-classes
and $M_0=20+2(2\gamma-20)=4\gamma-20=2N_0-20=2^{m_2+2}(4)-28$ $F$-classes.
It follows that $\B{Q}G$ is the direct sum of $N_0=2\gamma=2^{m_2}(8)-4$
fields and $N-N_0=2^{m_2}$ quaternion algebras, while $FG$ is
the direct sum of $M=2^{m_2+1}$ quaternion algebras.

\medskip
\noindent (b.ii)
The remaining case is $m_1-1<m_2$, that is, $m_1\le m_2$.  Here,
\lemref{lem:2cyclics} gives the number of cyclic subgroups of
$\C{Z}(G)\iso C_{2^{m_1}}\times C_{2^{m_2}}$
to be $N_1=2^{m_1}(3+m_2-m_1)-2$.  Since ten of these are generated by
elements of order at most $4$, the centre has $N_1$ \B{Q}-classes
and $M_1=2N_1-10$ $F$-classes.

Since $\C{Z}(G)/\langle s\rangle\iso C_{2^{m_1-1}}\times C_{2^{m_2}}$ has
$\alpha=2^{m_1-1}[3+m_2-(m_1-1)]-2=2^{m_1-1}(4+m_2-m_1)-2$ cyclic subgroups,
of which ten are generated by elements of order at most $4$,  this group
has $\alpha$ \B{Q}-classes and $2\alpha-10$ $F$-classes.
Let $H=\langle \C{Z}(G),x\rangle$.  We have
$H/\langle s\rangle\iso C_{2^{m_1-1}}\times C_{2^{m_2}}\times C_2$
with $2\alpha$ cyclic subgroups, of which twenty are generated by elements
of order at most $4$.  These facts imply $2\alpha$ \B{Q}-classes and
$4\alpha-20$ $F$-classes.  \lemref{lem:xclasses} gives $N_2=2\alpha-\alpha
=\alpha=2^{m_1-1}(4+m_2-m_1)-2$ \B{Q}-classes and $M_2=(4\alpha-20)-(2\alpha-10)
=2\alpha-10=2N_2-10$ $F$-classes involving~$x$.

Let $H=\langle\C{Z}(G),y\rangle=\langle t_1,y\rangle$.  Then $H/\langle s\rangle
\iso C_{2^{m_1-1}}\times C_{2^{m_2+1}}$ has $\beta
=2^{m_1-1}[3+(m_2+1)-(m_1-1)]-2=2^{m_1-1}(5+m_2-m_1)-2$
cyclic subgroups, of which $10$ are generated by elements of order
at most $4$.  So this group has $\beta$ \B{Q}-classes and $2\beta-10$
$F$-classes.  \lemref{lem:xclasses} gives $N_3=\beta-\alpha=2^{m_1}-1$
\B{Q}-classes involving $y$ and $M_3=(2\beta-10)-(2\alpha-10)=2(\beta-\alpha)
=2N_3$ $F$-classes. As before, the numbers for $xy$ are the same,
so $\B{Q}G$ is the direct sum of $N=N_1+N_2+2N_3=2^{m_1-1}(12+3m_2-3m_1)-4$
simple algebras and $FG$ is the direct sum of
$M=M_1+M_2+2M_3=2N-20=2^{m_1}(12+3m_2-3m_1)-28$
simple algebras.
Since $G/G'\iso C_{2^{m_1-1}}\times C_2\times C_{2^{m_2+1}}$ and $2\beta$
cyclic subgroups, of which twenty are generated by elements of order
at most $4$, the number of \B{Q}-classes here is $N_0=2\beta=
2^{m_1}(5+m_2-m_1)-4$ and the number of $F$-classes is $M_0=4\beta-20=2N_0-20
=2^{m_1+1}(5+m_2-m_1)-28$.
It follows that $\B{Q}G$ is the direct sum of $N_0=2^{m_1}(5+m_2-m_1)-4$
fields and $N-N_0=2^{m_1-1}(2+m_2-m_1)$ quaternion algebras,
while $FG$ is the direct sum of $M_0=2^{m_1+1}(5+m_2-m_1)-28$ fields
and $M-M_0=2^{m_1}(2+m_2-m_1)$ quaternion algebras.

\begin{thm} Let $G$ be a group of type $\C{D}_3$  with $m_1\ge3$ and $m_2\ge3$.
\smallskip\par\noindent Set $N=\begin{cases}
    2^{m_2}(5+4m_1-4m_2)-4 & \text{if $m_1-1\ge m_2+1$} \\
    9(2^{m_2})-4 & \text{if $m_2=m_1-1$} \\
    2^{m_1-1}(12+3m_2-3m_1)-4 & \text{if $m_1\le m_2$,}
    \end{cases}$
\smallskip\par\noindent $N_0=\begin{cases}
            2^{m_2+2}(1+m_1-m_2)-4 & \text{if $m_1-1\ge m_2+1$} \\
            8(2^{m_2}) -4          & \text{if $m_2=m_1-1$} \\
            2^{m_1}(5+m_2-m_1)-4   & \text{if $m_1\le m_2$,}
            \end{cases}$
\bigskip\par\noindent $M=2N-20$ and $M_0=2N_0-20$.  Then $\B{Q}G$ is the direct sum of
$N_0$ fields and $N-N_0$ quaternion algebras.  For any finite field $K$ of
odd order $q$, $KG$ is the direct
sum of at least $M_0$ fields and $M-M_0$ quaternion algebras, these numbers being
achieved if $q\equiv3\pmod8$.
\end{thm}

\section{Groups of type $\C{D}_4$}
Let
\begin{multline*}
G=\langle t_1,t_2,x,y\mid t_1^{2^{m_1}}=t_2^{2^{m_2}}=1, \\
x^2=t_1, y^2=t_2, t_1,t_2 \text{ central}, (x,y)=t_1^{2^{m_1-1}}=s\rangle
\end{multline*}
be a group of type $\C{D}_4$.

Again, we consider first several small cases
and in the first few, let $K$
denote ambiguously \B{Q} or a finite field of order $q\equiv3\pmod8$.

\medskip
\noindent (i) If $m_1=m_2=1$, then  $\C{Z}(G)=C_2\times C_2$ and $s=t_1$.  The centre
has one element of order $1$ and three elements of order $2$, each defining
a conjugacy class of size $1$ and hence a $K$-class.
Set $H=\langle \C{Z}(G),x \rangle \iso C_4\times C_2$.
Then $H/\langle s\rangle\iso C_2\times C_2$ has four $K$-classes and
$\langle \C{Z}(G),x \rangle/\langle s\rangle \iso C_2$ has two $K$-classes,
so there are two classes involving $x$.
A similar calculation shows that there are also two classes involving $y$ and
another
two involving $xy$, so $KG$ is the direct sum of eight simple algebras.
Now $G/G'\iso C_2\times C_4$,
so $K[G/G']$ is the group algebra of $C_4$ with coefficients in $KC_2\iso K\oplus K$
and is, therefore, the direct sum of six fields.
Consequently,  $KG$ is the direct sum of six fields and two quaternion algebras.

\medskip
\noindent (ii.a)  Suppose $m_1=2$ and $m_2=1$.  Then $|G|=32$, $s=t_1^2$
and $\C{Z}(G)\iso C_4\times C_2$.
Since $\C{Z}(G)$ contains six cyclic subgroups, each of order at most $4$,
the centre of $G$ contains six $K$-classes.

Set $H=\langle \C{Z}(G),x \rangle=\langle x,t_2\rangle$.
Then $H/\langle s\rangle\iso C_4\times C_2$.
This group has six cyclic subgroups and
$\C{Z}(G)/\langle s\rangle\iso C_2\times C_2$ has four, so
$G$ has two $K$-classes involving $x$.
Let $H=\langle \C{Z}(G),y\rangle=\langle t_1,y\rangle$.
Then $H/\langle s\rangle\iso C_2\times C_4$  has
six cyclic subgroups while $\C{Z}(G)/\langle s\rangle$ has
four, so the number of $K$-classes involving $y$ is two.
Now $(xy)^2=st_1t_2=t_1^3t_2$, so
$H=\langle \C{Z}(G),xy\rangle=\langle t_2,xy\rangle$,
so $H/\langle s\rangle=\langle\overline{t_2},\overline{xy}\rangle\iso C_2\times C_4$
has six, so there are two $K$-classes involving $xy$.
In all, we have $12$ $K$-classes in $G$, so $KG$ has $12$ simple components.

Since $G/G'=\langle\overline{x},\overline{y}\rangle\iso C_2 \times C_4$
has six cyclic subgroups, all of order at most $4$,
$K[G/G']$  has six $K$-classes,
so $KG$ is the direct sum of six fields and six
quaternion algebras.

\smallskip
\noindent (ii.b) Suppose $m_1\ge3$ and $m_2=1$.
Then $\C{Z}(G)=C_{2^{m_1}}\times C_2$ contains $2(m_1+1)$ cyclic subgroups, six of
which are generated by elements of order at most $4$,
so the centre of $G$ contains $N_1=2m_1+2$ \B{Q}-classes.  Those
of order at most $4$
are also $F$-classes while the remaining $\B{Q}$-classes each split into
two $F$-classes
giving in all, $M_1=6+2(2m_1+2-6)=4m_1-2$ $F$-classes.

Set $H=\langle \C{Z}(G),x \rangle = \langle t_2,x\rangle$.
Then $H/\langle s \rangle=\langle\overline{t_2},\overline{x}\rangle
\iso C_2\times C_{2^{m_1}}$ contains $2[3+(m_1-1)]-2=2m_1+2$ cyclic
subgroups, six of which are generated by elements
of order at most $4$.  These numbers imply $2m_1+2$ \B{Q}-classes
and $6+2(2m_1+2-6)=4m_1-2$ $F$-classes.  The group $\C{Z}(G)/\langle s\rangle
\iso C_{2^{m_1-1}}\times C_2$ has $2m_1$ cyclic subgroups, six of which are
generated by elements of order at
most $4$, so this implies $2m_1$ \B{Q}-classes and
$6+2(2m_1-6)=4m_1-6$ $F$-classes involving $x$.
By \lemref{lem:xclasses}, there are $N_2=(2m_1+2)-2m_1=2$ \B{Q}-classes involving $x$
and $M_2=(4m_1-2)-(4m_1-6)=4$ $F$-classes.

If we take $H=\langle \C{Z}(G),y \rangle=\langle t_1,y\rangle$ then
$H/\langle s\rangle\iso C_{2^{m_1-1}}\times C_4$ has $4(3+m_1-1-2)-2=4m_1-2$ cyclic subgroups, ten
of which are generated by elements of order at most $4$.  This implies $4m_1-2$ \B{Q}-classes and
$10+2(4m_1-2-10)=8m_1-14$ $F$-classes involving $y$.
The group $\C{Z}(G)/\langle s\rangle$ was considered above.
We obtain $N_3=(4m_1-2)-2m_1=2m_1-2$ \B{Q}-classes involving $y$
and $M_3=(8m_1-14)-(4m_1-6)=4m_1-8$ $F$-classes.

Now $(xy)^2=st_1t_2$, so $H=\langle \C{Z}(G),xy\rangle=\langle t_2,xy\rangle$.
Thus $H/\langle s\rangle\iso C_2\times C_{2^{m_1}}$
has $2[(3+(m_1-1)]-2=2m_1+2$ \B{Q}-classes and $2(2m_1+2-6)+6=4m_1-2$ $F$-classes.
In all, this gives $N_4=(2m_1+2)-2m_1=2$ \B{Q}-classes involving $xy$
and $M_4=(4m_1-2)-(4m_1-6)=4$ $F$-classes.  Thus $\B{Q}G$ has $N=N_1+N_2+N_3+N_4=4m_1+4$
simple components and $FG$ has $M=M_1+M_2+M_3+M_4=8m_1-2$ simple components.

Now $G/G'=\langle \overline{x},\overline{y}\rangle \iso C_{2^{m_1}}\times C_4$
has $N_0=4[3+(m_1-2)]-2=4m_1+2$ cyclic subgroups,
$10$ of order at most $4$, so the commutative part of $\B{Q}G$ is the direct sum of
$4m_1+2$ fields, while the commutative part of $FG$ is the direct sum
of $M_0=2(4m_1+2-10)+10=8m_1-6$ fields.
It follows that
$\B{Q}G$ is the direct sum of $N_0=4m_1+2$ fields and $N-N_0=(4m_1+4)-(4m_1+2)=2$
quaternion algebras,
while $FG$ is the direct sum of $M_0=8m_1-6$ fields
and $M-M_0=(8m_1-2)-(8m_1-6)=4$ quaternion algebras.

\medskip
\noindent (iii.a) Suppose $m_1=1$ and $m_2=2$.
Then $|G|=32$, $\C{Z}(G)\iso C_2\times C_4$ and $s=t_1$.
As in case (ii.a), there are six \B{Q}-classes of the centre
and these are also $F$-classes.

Now $\langle \C{Z}(G),x \rangle/\langle s\rangle \iso C_2\times C_4$
contains six  \B{Q}-classes while $\C{Z}(G)/\langle s\rangle \iso C_4$
contains three, so the number of $K$-classes containing $x$ is three.
The group $\langle \C{Z}(G),y\rangle/\langle s\rangle\iso  C_8$ contains
four  \B{Q}-classes and $\langle \C{Z}(G)\rangle/\langle s\rangle \iso C_4$
contains three,  so there is just one \B{Q}-class in $G$ containing $y$.
As one of the \B{Q}-classes of $H/\langle s \rangle$
(corresponding to the elements of order $8$)
splits into two $F$-classes
and the \B{Q}-classes of $\langle \C{Z}(G) \rangle/\langle s\rangle$ do not split,
the number of $F$-classes containing $y$ is two.
A similar calculations holds for $xy$, so there are also one \B{Q}-class and
two $F$-classes containing $xy$.
All this shows that the total number of \B{Q}-classes of $G$ is $11$
and the total number of $F$-classes is $13$.

Since $G/G'=\langle \overline{x}\rangle\times\langle \overline{y}\rangle
\iso C_2\times C_8$, $K[G/G']=KC_2[C_8]\iso (K\oplus K)C_8\iso KC_8\oplus KC_8$
is the direct sum of $2(4)=8$ fields if $K=\B{Q}$ and
$2(5)=10$ fields if $K$ has finite order $q\equiv3\pmod8$.  It follows that
\B{Q}G is the direct sum of eight fields and three
quaternion algebras while $FG$ is the
direct sum of ten fields and three quaternion algebras,
each necessarily a ring of $2\times2$ matrices over a field.

\medskip
\noindent (iii.b) Suppose $m_1=1$ (so that $s=t_1$) and $m_2\ge3$.
Then $\C{Z}(G)=C_2\times C_{2^{m_2}}$ contains $2(m_2+1)$ cyclic subgroups, six of order at most $4$,
so the centre of $G$ contains $N_1=2m_2+2$ \B{Q}-classes.  Those of size at most $4$
are also $F$-classes while the remaining $\B{Q}$-classes each split into two $F$-classes
giving in all, $M_1=2(2m_2+2-6)+6=4m_2-2$ $F$-classes.

Set $H=\langle \C{Z}(G),x \rangle = \langle t_2,x\rangle$.
Then $H/\langle s \rangle=\langle\overline{t_2},\overline{x}\rangle
\iso C_{2^{m_2}}\times C_2$ contains $2m_2+2$ cyclic
subgroups, six of which are generated by elements
of order at most $4$.  These numbers imply $2m_2+2$ \B{Q}-classes
and $2(2m_2+2-6)+6=4m_2-2$ $F$-classes.  The group $\C{Z}(G)/\langle s\rangle
\iso C_{2^{m_2}}$ has $m_2+1$ cyclic subgroups, three of which are generated
by elements of order at
most $4$, so this implies $m_2+1$ \B{Q}-classes and
$2(m_2+1-3)+3=2m_2-1$ $F$-classes involving $x$.
By \lemref{lem:xclasses}, there are $N_2=(2m_2+2)-(m_2+1)=m_2+1$ \B{Q}-classes involving $x$
and $M_2=(4m_2-2)-(2m_2-1)=2m_2-1$ $F$-classes.

With $H=\langle \C{Z}(G),y\rangle=\langle t_1,y\rangle$,
$H/\langle s\rangle\iso C_{2^{m_2+1}}$ has $m_2+2$ cyclic subgroups, three
of which are generated by elements of order at most $4$.
This implies $m_2+2$ \B{Q}-classes and $2(m_2+2-3)+3=2m_2+1$
$F$-classes involving $y$.
The group $\C{Z}(G)/\langle s\rangle$ was considered above.
We obtain $N_3=(m_2+2)-(m_2+1)=1$ \B{Q}-class involving $y$
and $M_3=(2m_2+1)-(2m_2-1)=2$ $F$-classes.
Since $(xy)^2=st_1t_2=t_2$, the counts for $xy$ and $y$ are the same.
In all, the Wedderburn decomposition of $\B{Q}G$ has $N=N_1+N_2+2N_3=3m_2+5$
simple components and the Wedderburn decomposition
of $FG$ has $M=M_1+M_2+2M_3=6m_2+1$ simple components.

Now $G/G'=\langle \overline{x},\overline{y}\rangle \iso C_2\times C_{2^{m_2+1}}$
has $N_0=2(m_2+2)=2m_2+4$ cyclic subgroups,
six of order at most $4$, so the commutative part of $\B{Q}G$ is the direct sum of
$2m_2+4$ fields, while the commutative part of $FG$ is the direct sum
of $M_0=2(2m_2+4-6)+6=4m_2+2$ fields.
It follows that
$\B{Q}G$ is the direct sum of $N_0=2m_2+4$ fields and $N-N_0=(3m_2+5)-(2m_2+4)=m_2+1$
quaternion algebras,
while $FG$ is the direct sum of $M_0=4m_2+2$ fields
and $M-M_0=(6m_2+1)-(4m_2+2)=2m_2-1$ quaternion algebras.

\medskip
\noindent (iv.a) Suppose $m_1=m_2=2$.  Then $\C{Z}(G)\iso C_4\times C_4$
contains $N_1=10$ \B{Q}-classes that are also $F$-classes: $M_1=10$.
The group $H=\langle \C{Z}(G),x\rangle/\langle s\rangle\iso C_4\times C_4$
contains $10$ \B{Q}-classes that are also $F$-classes
while $\C{Z}(G)/\langle s\rangle \iso C_2\times C_4$ contains
six \B{Q}-classes that are also $F$-classes
so the total number of \B{Q}-classes in $G$ involving $x$ is $N_2=4$,
and $M_2=4$ as well.

Let $H=\langle \C{Z}(G),y\rangle\langle s\rangle\iso C_2\times C_8$.  This group
contains eight cyclic subgroups, six generated by elements of order
at most $4$.  So there are eight \B{Q}-classes and $6+2(8-6)=10$
$F$-classes.  It follows that the number of \B{Q}-classes involving
$y$ is $N_3=8-6=2$ and the number of $F$-classes involving
$y$ is $M_3=10-6=4$. The situation with regards classes involving $xy$
is similar:  $N_4=N_3$, $M_4=M_3$.  It follows that
the total number of \B{Q}-classes in $G$ is $N=N_1+N_2+2N_3=18$ and the total number
of $F$-classes is $M=M_1+M_2+2M_3=22$.

Now,  $G/G'\iso C_4\times C_8\times C_2$, so \lemref{lem:2cyclics} shows that
$\B{Q}[G/G']$ has $N_0=14$ simple components while $F[G/G']$ has $M_0=18$ components.
Consequently,
\B{Q}G is the direct sum of $N_0=14$ fields and $N-N_0=4$ quaternion algebras
while $FG$ is the
direct sum of $M_0=18$ fields and $M-M_0=4$ quaternion algebras,
each necessarily isomorphic to a ring of $2\times2$ matrices.

\smallskip
\noindent (iv.b) Suppose $m_1\ge3$ and $m_2=2$.  Then
$\C{Z}(G)\iso C_{2^{m_1}}\times C_4$ contains $N_1=4m_1+2$ \B{Q}-classes
and $M_1=8m_1-6$ $F$-classes.  (Again, we refer to \tabref{tab1}.)

Let $H=\langle \C{Z}(G),x\rangle=\langle x,t_2\rangle$.  Then $H/\langle s\rangle
\iso C_{2^{m_1}}\times C_4$ has $4m_1+2$ \B{Q}-classes and $8m_1-6$ $F$-classes
while $\C{Z}(G)/\langle s\rangle\iso C_{2^{m_1-1}}\times C_4$ has $4m_1-2$
and $8m_1-14$, respectively.  \lemref{lem:xclasses} gives
$N_2=(4m_1+2)-(4m_1-2)=4$ \B{Q}-classes involving $x$ and
$M_2=(8m_1-6)-(8m_1-14)=8$ $F$-classes.

Let $H=\langle \C{Z}(G),y\rangle=\langle t_1,y\rangle$.
Then $H/\langle s\rangle\iso C_{2^{m_1-1}}\times C_8$
has $8(m_1-1)-2=8m_1-10$ \B{Q}-classes and $16(m_1-1)-14=16m_1-30$ $F$-classes, giving
$N_3=(8m_1-10)-(4m_1-2)=4m_1-8$ \B{Q}-classes involving $y$ and
$M_3=(16m_1-30)-(8m_1-14)=8m_1-16$
$F$-classes.  Since $(xy)^2=st_1t_2$, $\langle \C{Z}(G),xy\rangle=\langle t_2,xy\rangle$,
so the number of \B{Q}- and $F$-classes involving $xy$ are the same as those
involving $x$.  In all, we have $N=N_1+2N_2+N_3=8m_1+2$ \B{Q}-classes in $G$
and $M=M_1+2M_2+M_3=16m_1-6$ $F$-classes.

Now $G/G'=\langle \overline{x},\overline{y}\rangle\iso C_{2^{m_1}}\times C_8$
has $N_0=8m_1-2$ \B{Q}-classes and $M_0=16m_1-14$ $F$-classes.  We conclude
that $\B{Q}G$ is the direct sum of $N_0=8m_1-2$ fields and $N-N_0=4$
quaternion algebras, while $FG$ is the direct sum of $M_0=16m_1-14$
fields and $M-M_0=8$ quaternion algebras.

\smallskip
\noindent (iv.c) Suppose $m_1=2$ and $m_2\ge3$.  Then
$\C{Z}(G)\iso C_4\times C_{2^{m_2}}$ contains $N_1=4m_2+2$ \B{Q}-classes
and $M_1=8m_2-6$ $F$-classes.

Let $H=\langle \C{Z}(G),x\rangle=\langle x,t_2\rangle$.  Then $H/\langle s\rangle
\iso C_4\times C_{2^{m_2}}$ has $4m_2+2$ \B{Q}-classes and $8m_2-6$ $F$-classes
while $\C{Z}(G)/\langle s\rangle\iso C_2\times C_{2^{m_2}}$ has $2m_2+2$
and $4m_2-2$, respectively.  \lemref{lem:xclasses} gives
$N_2=(4m_2+2)-(2m_2+2)=2m_2$ \B{Q}-classes involving $x$ and
$M_2=(8m_2-6)-(4m_2-2)=4m_2-4$ $F$-classes.

Let $H=\langle \C{Z}(G),y\rangle$.  Then $H/\langle s\rangle\iso C_2\times C_{2^{m_2+1}}$
has $2m_2+4$ \B{Q}-classes and $4m_2+2$ $F$-classes, giving
$N_3=(2m_2+4)-(2m_2+2)=2$ \B{Q}-classes involving $y$ and $M_3=(4m_2+2)-(4m_2-2)=4$
$F$-classes.  Since $(xy)^2=st_1t_2$, $\langle \C{Z}(G),xy\rangle=\langle t_1,xy\rangle$,
so the number of \B{Q}- and $F$-classes involving $xy$ are the same as those
involving $y$.  In all, we have $N=N_1+N_2+2N_3=6m_2+6$ \B{Q}-classes in $G$
and $M=M_1+M_2+2M_3=12m_2-2$ $F$-classes.

Now $G/G'=\langle \overline{x},\overline{y}\rangle\iso C_4\times C_{2^{m_2+1}}$
has $N_0=4m_2+6$ \B{Q}-classes and $M_0=8m_2+2$ $F$-classes.  We conclude
that $\B{Q}G$ is the direct sum of $N_0=4m_2+6$ fields and $N-N_0=2m_2$
quaternion algebras, while $FG$ is the direct sum of $M_0=8m_2+2$
fields and $M-M_0=4m_2-4$ quaternion algebras.

\bigskip
Now we turn to the general case, assuming that $m_1\ge3$ and $m_2\ge3$.
\lemref{lem:2cyclics} allows to us to determine the number $N_1$ of
cyclic subgroups of $\C{Z}(G)\iso C_{2^{m_1}}\times C_{2^{m_2}}$.
Of these, $10$ are generated by elements of order at most $4$ (those
lying in $C_4\times C_4$),
so the number of \B{Q}-classes in $\C{Z}(G)$ is $N_1$
and the number of $F$-classes is $M_1=10+2(N_1-10)=2N_1-10$.
Denote by $\alpha$ the number of cyclic subgroups of $\C{Z}(G)/\langle s\rangle
\iso C_{2^{m-1}}\times C_{2^{m_2}}$.  Since $m_1-1\ge 2$, ten of these
subgroups are generated by elements of order at most $4$, so there
are $\alpha$ and $10+2(\alpha-10)=2\alpha-10$ \B{Q}- and $F$-classes, respectively,
in this quotient.  Let $H=\langle\C{Z}(G),x\rangle=\langle x,t_2\rangle$.
Then $H/\langle s\rangle\iso C_{2^{m_1}}\times C_{2^{m_2}}$ has
$N_1$ and $2N_1-10$ \B{Q}- and $F$-classes, respectively,
so, using \lemref{lem:xclasses}, we know the numbers of \B{Q}- and $F$-classes
involving $x$:  $N_2=N_1-\alpha$ and $M_2=(2N_1-10)-(2\alpha-10)=2(N_1-\alpha)=2N_2$.

Let $H=\langle \C{Z}(G),y\rangle=\langle t_1,y\rangle$.
Suppose $H/\langle s\rangle\iso C_{2^{m_1-1}}\times C_{2^{m_2+1}}$
has $\beta$ cyclic subgroups.  Again, since $m_1-1\ge2$, ten
of these are generated by elements of order at most $4$, giving
$\beta$ and $2\beta-10$ \B{Q}- and $F$-classes, respectively.
From \lemref{lem:xclasses}, we learn that the numbers of \B{Q}-
and $F$-classes involving $y$ are $N_3=\beta-\alpha$
and $M_3=(2\beta-10)-(2\alpha-10)=2(\beta-\alpha)=2N_3$, respectively.

Let $H=\langle \C{Z}(G),xy\rangle$.  This group is
$\langle t_1\rangle\times \langle xy\rangle$ or $\langle t_2\rangle\times\langle xy\rangle$
according as $m_1\le m_2$ or $m_2>m_1$, respectively.  In any event,
letting $N_4$ denote the number of \B{Q}-classes involving $xy$ (as usual),
the number of $F$-classes involving $y$ is $M_4=2N_4$, as with previous situations.
If $N=N_1+N_2+N_3+N_4$, then $\B{Q}G$ is the direct sum of $N$
simple algebras, then $FG$ is the direct sum of $M=M_1+M_2+M_3+M_4
=(2N_1-10)+2N_2+2N_3+2N_4=2N-10$ simple algebras.

Now $G/G'=\langle \overline{x}\rangle\times\langle\overline{y}\rangle
\iso C_{2^{m_1}}\times C_{2^{m_2+1}}$ and \lemref{lem:2cyclics} allows
to compute $N_0$, the number of cyclic subgroups of this group, which
is also the number of \B{Q}-classes of $G/G'$
and the number simple components of $\B{Q}[G/G']$, the commutative
part of $\B{Q}G$.  On the other hand, the number of $F$-classes
of $G/G'$ is $M_0=10+2(N_0-10)=2N_0-10$.

We conclude that
$\B{Q}G$ is the direct sum of $N_0$ fields and $N-N_0$ quaternion
algebras, while $FG$ is the direct sum of $M_0=2N_0-10$ fields
and $M-M_0=2(N-N_0)$ quaternion algebras.

\begin{thm} Let $G$ be a group of type $\C{D}_4$ with $m_1\ge3$ and $m_2\ge3$.
Let $N=N_1+N_2+N_3+N_4$ where $N_1$, $N_2$,
$N_3$,  $N_4$ are, respectively, the numbers of cyclic subgroups of \C{Z}(G) and the numbers
of \B{Q}-classes involving $x$, $y$ and $xy$.  Let $N_0$
be the number of cyclic subgroups of $G/G'$.
(Note that these numbers
can be determined by \lemref{lem:2cyclics} and \lemref{lem:xclasses}.)
Then $\B{Q}G$ is the direct sum of $N_0$ fields and $N-N_0$ quaternion algebras.
Let $M=2N-10$ and $M_0=2N_0-10$.
For any finite field $K$ of odd order $q$, the Wedderburn decomposition
of the group algebra $KG$ has at least $M$ simple
components, a number that is achieved if
$q\equiv3\pmod{8}$, in which case
$KG$ is the direct sum of $M_0=2N_0-10$  fields and $M-M_0=2(N-N_0)$ quaternion algebras.
\end{thm}

\section{Groups of type $\C{D}_5$}\label{sec:d5}
Once again, we begin with a lemma that enables us to determine the number of cyclic
subgroups in
the centre of the groups of interest in this section.

\begin{lem}\label{lem:3cyclics} Let $A=C_{2^a}\times C_{2^b}\times C_{2^c}$
be the direct product of cyclic groups of orders $2^a$, $2^b$ and $2^c$
with $a\ge b\ge c\ge1$.  The number of cyclic subgroups of $A$ of order $2^k$
is
\begin{itemize}
\item $7(2^{2(k-1)})$ if $1\le k\le c$,
\item $3(2^{c+k-1})$  if $c< k\le b$, and
\item $2^{b+c}$  if $b< k\le a$.
\end{itemize}
In all, $A$ has $N=\tfrac73(4^c-1)+3(2^c)(2^b-2^c)+(a-b)2^{b+c}+1$
cyclic subgroups (and hence $N$ \B{Q}-classes).
\end{lem}
\begin{proof}
Suppose $1\le k\le c$.  An element $(u,v,w)\in A$ satisfies $(u,v,w)^{2^k}=1$ if
and only if it belongs to a subgroup of $A$ isomorphic to
$C_{2^k}\times C_{2^k}\times C_{2^k}$, so the
number of elements in $A$ of order $2^k$ is $|C_{2^k}\times C_{2^k}\times C_{2^k}|$
less $|C_{2^{k-1}}\times C_{2^{k-1}}\times C_{2^{k-1}}|$, that is,
$2^{3k}-2^{3(k-1)}=2^{3(k-1)}(2^3-1)=7(2^{3(k-1)})$.
Since there are $\phi(2^k)=2^{k-1}$ generators of any cyclic subgroup
of order $2^k$, the number of cyclic subgroups of order $2^k$ in this first case
is $7(2^{3(k-1)})/2^{k-1}=7(2^{2(k-1)})$.
\smallskip\par Suppose $c<k\le b$.  An element $(u,v,w)\in A$ has order $2^k$
if and only if it has order $2^k$ in a copy of $C_{2^k}\times C_{2^k}\times C_{2^c}$.
Using \lemref{lem:2cyclics}, we see that there are $3(4^{k-1})2^c$ such elements.
With reasoning as before, the number of cyclic subgroups of this order is therefore
$3(2^{k-1})2^c$.
\smallskip\par Suppose that $b<k\le a$.  Then $(u,v,w)\in A$ has order
$2^k$ if and only if it has order $2^k$ in a subgroup of $A$ isomorphic to
$C_{2^k}\times C_{2^b}\times C_{2^c}$.  There are $2^{k-1}2^{b+c}$ such elements,
so the number of cyclic subgroups is $2^{b+c}$.

The number of cyclic subgroups of order $2^k$, with $1\le k\le c$, is
\begin{multline*}
7+7(2^2)+7(2^4)+7(2^6)+\cdots +7(2^{2(c-1)})\\
 =7(1+2^2+2^4+\cdots+2^{2(c-1)})=\tfrac73(2^{2c}-1)=\tfrac73(4^c-1).
\end{multline*}
The number of cyclic subgroups of order $2^k$ with $c<k\le b$ is
\begin{multline*}
3(2^c)(2^c+2^{c+1}+2^{c+2}+\cdots+2^{b-1}) \\
\begin{array}{cl}
= &3(2^c)(2^c)(1+2+2^2+\cdots+2^{b-c-1}) \\
= &3(2^c)(2^c)(2^{b-c}-1) =3(2^c)(2^b-2^c)
\end{array}
\end{multline*}
and the number of cyclic subgroups of order $2^k$
with $b<k\le a$ is $(a-b)2^{b+c}$.
Including the trivial subgroup, the total number of cyclic subgroups of $A$ is as stated.
\end{proof}

We refer the reader to \tabref{tab1} where many useful consequences of this lemma
appear.  To determine the numbers involved, use
\lemref{lem:3cyclics} to obtain the number of \B{Q}-classes
in the direct product of three cyclic groups
and then \lemref{lem1} to see which of these split into (two) $F$-classes and
which do not, to get the number of $F$-classes.  Specifically,
if there are $\alpha$ \B{Q}-classes, $\beta$ of which correspond to cyclic groups
generated by an element of order at most $4$, then there are $\alpha+2(\alpha-\beta)
=2\alpha-\beta$ $F$-classes.

\medskip

Now let
\begin{multline*}
G=\langle t_1,t_2,t_3,x,y\mid t_1^{2^{m_1}}=t_2^{2^{m_2}}=t_3^{2^{m_3}}=1, \\
x^2=t_2, y^2=t_3, \; t_1,t_2,t_3 \text{ central}, (x,y)=t_1^{2^{m_1-1}}=s\rangle
\end{multline*}
be a group of type $\C{D}_5$.

Again, our analysis of the semisimple structure of $KG$, for various
fields $K$, begins with a study of some special cases.
\bigskip

\noindent (i) Assume $m_1=m_2=m_3=1$.  The results in this section hold
for any field $K$ of characteristic different from $2$.  We have $s=t_1$ and
$\C{Z}(G)\iso C_2\times C_2\times C_2$ with eight cyclic subgroups,
all of suitably small order, so there are $8$ $K$-classes of $\C{Z}(G)$.
Since $\langle\C{Z}(G),x\rangle=\langle t_1,x,t_2\rangle$,
$\langle \C{Z}(G),x \rangle/\langle s\rangle=\langle \overline{x},\overline{t_2}\rangle
\iso C_4\times C_2$ has six cyclic subgroups.
Since $\C{Z}(G)/\langle s\rangle\iso C_2\times C_2$ has four,
\lemref{lem:xclasses} says there are two $K$-classes involving $x$.
Similarly, there are two classes involving $y$.
Now $xy$ has order $4$, since $(xy)^2=t_1t_2t_3$, so
$\langle \C{Z}(G),xy \rangle/\langle s\rangle=\langle \overline{t_2},\overline{xy}\rangle
\iso  C_2\times C_4$ has six cyclic subgroups.  This implies,
another two $K$-classes, these involving $xy$.
In all, there are $8+3(2)=14$ $K$-classes in $G$.

As $G/G'\iso C_4\times C_4$,  $K[G/G']$ is the direct sum of ten fields,
so $KG$ is the direct sum of ten fields and four quaternion algebras.

\medskip\noindent(ii.a) Assume $m_1=1$, $m_2=1$, $m_3=2$.  (Conclusions
when $m_1=1$, $m_2=2$, $m_3=1$ are the same.)  Again we have $s=t_1$.
In this situation, $\C{Z}(G)\iso C_2\times C_2\times C_4$ has $12$ cyclic subgroups,
so there are $12$ $K$-classes in the centre, where again $K$ denotes
any field of characteristic different from $2$.
Now $\langle\C{Z}(G),x\rangle/\langle s\rangle=\langle\overline{x},\overline{t_3}\rangle
\iso C_4\times C_4$, so there are $10$ cyclic subgroups, all of suitably
small order, giving $10$ $K$-classes.
As $\C{Z}(G)/\langle s\rangle\iso C_2 \times C_4$ has $6$, \lemref{lem:xclasses}
gives us four $K$-classes involving $x$.
Now $\langle\C{Z}(G),y\rangle/\langle s\rangle
\iso \langle\overline{t_2},\overline{y}\rangle$ and $\langle \C{Z}(G),xy\rangle/\langle s\rangle=
\langle\overline{t_2},\overline{xy}\rangle$ (because $(xy)^2=st_2t_3$).  These
groups are both isomorphic to
$C_2\times C_8$, a group which contains eight cyclic subgroups,
six of which are generated by elements of order at most $4$.
These numbers imply $8-6=2$ \B{Q}-classes and $6+2(2)-6=4$ $F$-classes.
In all, there are $12+4+2+2=20$ \B{Q}-classes and $12+4+4+4=24$ $F$-classes.

Since $G/G'=\langle\overline{x},\overline{y}\rangle\iso C_4\times C_8$ has $14$
cyclic subgroups, of which ten are generated by elements of order at most $4$, the commutative
part of $\B{Q}G$ is the direct sum of $14$ fields and the commutative part
of $FG$ is the direct sum of $10+2(4)=18$ fields.  It follows that $\B{Q}G$
is the direct sum of $14$ fields and six quaternion algebras while $FG$
is the direct sum of $18$ fields and six quaternion algebras.

\medskip\noindent (ii.b) Assume $m_1=1$, $m_2=1$ and $m_3\ge3$.
As in case ii.a, we have $s=t_1$, but here,
$\C{Z}(G)\iso C_2\times C_2\times C_{2^{m_3}}$ has $N_1=4(m_3+1)=4m_3+4$ cyclic subgroups,
of which $12$ are generated by elements of order at most $4$.  So there are $N_1$ \B{Q}-classes in the centre
and $M_1=12+2(N_1-12)=2N_1-12=8m_3-4$ $F$-classes.
Now $\langle\C{Z}(G),x\rangle/\langle s\rangle=\langle\overline{x},\overline{t_3}\rangle
\iso C_4\times C_{2^{m_3+1}}$ has $4[3+(m_3+1)-2]-2=4m_3+6$ cyclic subgroups,
ten of which are generated by elements of order at most $4$.  This means $4m_3+6$ \B{Q}-classes
and $10+2(4m_3+6-10)=8m_3+2$ $F$-classes.
Now $\C{Z}(G)/\langle s\rangle\iso C_2 \times C_{2^{m_3}}$ has $2(m_3+1)=2m_3+2$
cyclic subgroups, six of which are generated by elements of order at most $4$.  This means $2m_3+2$
\B{Q}-classes and $6+2(2m_3+2-6)=4m_3-2$ $F$-classes.  \lemref{lem:xclasses}
now tells us that there are $N_2=(4m_3+6)-(2m_3+2)=2m_3+4$ \B{Q}-classes involving $x$
and $M_2=(8m_3+2)-(4m_3-2)=4m_3+4$ $F$-classes involving $x$.
Now $\langle\C{Z}(G),y\rangle/\langle s\rangle
\iso \langle\overline{t_2},\overline{y}\rangle$ and $\langle \C{Z}(G),xy\rangle/\langle s\rangle=
\langle\overline{t_2},\overline{xy}\rangle$ (because $(xy)^2=st_2t_3$).  These
groups are both isomorphic to $C_2\times C_{2^{m_3+1}}$, a group which contains $2(m_3+2)=2m_3+4$
cyclic subgroups,
six of which are generated by elements of order at most $4$, so this group contains $2m_3+4$ \B{Q}-classes
and $6+2(2m_3+4-6)=4m_3+2$ $F$-classes.  By \lemref{lem:xclasses},
we have $N_3=N_4=(2m_3+4)-(2m_3+2)=2$
\B{Q}-classes and $M_3=M_4=(4m_3+2)-(4m_3-2)=4$ $F$-classes.
In all, there are $N=N_1+N_2+2N_3=6m_3+12$ \B{Q}-classes and $M=M_1+M_2+2M_3=12m_3+8$ $F$-classes,
so the decompositions of $\B{Q}G$ and $FG$ have, respectively, $N=6m_3+12$ and $M=12m_3+8$
simple components.

Finally, as $G/G'=\langle\overline{x},\overline{y}\rangle\iso C_4\times C_{2^{m_3+1}}$ has
$4[3+(m_3+1)-2]-2=4m_3+6$ cyclic subgroups, of which ten are generated by elements of order at most $4$,
the commutative
part of $\B{Q}G$ is the direct sum of $4m_3+6$ fields and the commutative part
of $FG$ is the direct sum of $10+2(4m_3+6-10)=8m_3+2$ fields.  It follows that $\B{Q}G$
is the direct sum of $4m_3+6$ fields and $2m_3+6$ quaternion algebras while $FG$
is the direct sum of $8m_3+2$ fields and $4m_3+6$ quaternion algebras.

\medskip\noindent (iii.a) Suppose $m_1=1$ and $m_2=m_3=2$.
Then $s=t_1$ and $\C{Z}(G)\iso C_2\times C_4\times C_4$ has twenty cyclic subgroups,
all (clearly) of order at most $4$.
Here, we have
$\langle \C{Z}(G),x \rangle/\langle s\rangle=\langle\overline{x},\overline{t_3}\rangle$,
$\langle\C{Z}(G),y\rangle/\langle s\rangle=\langle\overline{t_2},\overline{y}\rangle
=\langle \C{Z}(G),xy \rangle/\langle s\rangle$ (since $(xy)^2=st_2t_3$).  Each of these
groups, being isomorphic to $C_8\times C_4$ has $14$ cyclic subgroups,
ten of which are generated by elements of order at most four.  So we have $14$ \B{Q}-classes
and $10+2(4)=18$ $F$-classes in this group.
As $\C{Z}(G)/\langle s\rangle\iso C_4\times C_4$ has $10$ cyclic subgroups,
all of order at most $4$, \lemref{lem:xclasses} tells
us that there are four \B{Q}-classes involving each of $x$, $y$, $xy$
and $18-10=8$ $F$-classes.  In all, we have $20+3(4)=32$ \B{Q}-classes and $20+3(8)=44$
$F$-classes.

Finally, as $G/G'\iso C_8\times C_8$ has $22$ cyclic subgroups, $10$
of order at most $4$, the commutative part of $\B{Q}G$ is the direct sum
of $22$ fields and the commutative part of $FG$ is the direct sum
of $10+2(12)=34$ fields.  We conclude that $\B{Q}G$ is the direct sum
of $22$ fields and ten quaternion algebras, while $FG$ is the direct sum of $34$ fields
and $10$ quaternion algebras.

\medskip\noindent (iii.b) Assume $m_1=1$, $m_2\ge3$ and $m_3=2$.  (By symmetry,
our conclusions when $m_1=1$, $m_2=2$, $m_3\ge3$ will be the same.)
As in case iii.a, $s=t_1$, but here
$\C{Z}(G)\iso C_2\times C_{2^{m_2}}\times C_4$, a group with $8[3+(m_2-1)]-4=8m_2+4$
cyclic subgroups, $20$ of which are generated by elements of order at most $4$.  So the centre contains
$N_1=8m_2+4$ \B{Q}-classes and $M_1=20+2(N_1-20)=16m_2-12$ $F$-classes.
Now $\langle \C{Z}(G),x \rangle/\langle s\rangle=\langle\overline{x},\overline{t_3}\rangle
\iso C_{2^{m_2+1}}\times C_4$
has $4[3+(m_2+1)-2]-2=4m_2+6$ cyclic subgroups, ten generated by elements of order at most $4$.
So there are $4m_2+6$ \B{Q}-classes and $10+2(4m_2+6-10)=8m_2+2$ $F$-classes in this group.
As $\C{Z}(G)/\langle s\rangle\iso C_{2^{m_2}}\times C_4$ has $4m_2+2$ cyclic subgroups,
ten generated by elements of order at most $4$, this group has $4m_2+2$ \B{Q}-classes and $10+2(4m_2+2-10)=8m_2-6$
$F$-classes.  \lemref{lem:xclasses} gives $N_2=(4m_2+6)-(4m_2+2)=4$ \B{Q}-classes
and $M_2=(8m_2+2)-(8m_2-6)=8$ $F$-classes involving $x$.
The situation with $xy$ is the same as with $x$, so $N_4=N_2$ and $M_4=M_2$.
Turning to $y$, we have $\langle \C{Z}(G),y \rangle/\langle s\rangle=\langle\overline{t_2},\overline{y}\rangle
\iso C_{2^{m_2}}\times C_8$
with $8(3+m_2-3)-2=8m_2-2$ cyclic subgroups, ten generated
by elements of order at most $4$.
So there are $8m_2-2$ \B{Q}-classes and $10+2(8m_2-2-10)=16m_2-14$ $F$-classes in this group
and hence $N_3=(8m_2-2)-(4m_2+2)=4m_2-4$ \B{Q}-classes and $M_3=(16m_2-14)-(8m_2-6)=8m_2-8$
$F$-classes involving $y$.  In all, there are $N=N_1+2N_2+N_3=12m_2+8$ \B{Q}-classes
and $M=M_1+2M_2+M_3=24m_2-4$ $F$-classes.
\par Finally, as $G/G'\iso C_{2^{m_2}}\times C_4$, the commutative part of $\B{Q}G$
is the direct sum of $N_0=4m_2+2$ fields, while the commutative part of $FG$
is the direct sum of $M_0=8m_2-6$ fields.   We conclude that $\B{Q}G$ is the direct sum
of $N_0=4m_2+2$ fields and $N-N_0=8m_2+6$ quaternion algebras,
while $FG$ is the direct sum of $M_0=8m_2-6$ fields
and $M-M_0=16m_2+2$ quaternion algebras.

\medskip\noindent (iii.c) Assume $m_1=1$, $m_2\ge3$ and $m_3\ge3$.  We also assume
$m_2\ge m_3$  (by symmetry, the numbers we obtain will be the same if $m_3\ge m_2$)
and suppose first that $m_2>m_3$.

Again we have $s=t_1$, but here
$\C{Z}(G)\iso C_2\times C_{2^{m_2}}\times C_{2^{m_3}}$, a group with $N_1=2[2^{m_3}(3+m_2-m_3)-2]
=2^{m_3+1}(3+m_2-m_3)-4$ cyclic subgroups, $20$ of which are generated by elements of order at most $4$.
So the centre contains
$N_1$ \B{Q}-classes and $M_1=20+2(N_1-20)=2^{m_3+2}(3+m_2-m_3)-28$ $F$-classes.
Now $\langle \C{Z}(G),x \rangle/\langle s\rangle
=\langle\overline{x},\overline{t_3}\rangle
\iso C_{2^{m_2+1}}\times C_{2^{m_3}}$
has $2^{m_3}(3+m_2+1-m_3)-2=2^{m_3}(4+m_2-m_3)-2$ cyclic subgroups, ten
generated by elements of order at most $4$.
So there are $2^{m_3}(4+m_2-m_3)-2$ \B{Q}-classes and $10+2[2^{m_3}(4+m_2-m_3)-2-10]=
2^{m_3+1}(4+m_2-m_3)-14$ $F$-classes in this group.
As $\C{Z}(G)/\langle s\rangle\iso C_{2^{m_2}}\times C_{2^{m_3}}$ has $2^{m_3}(3+m_2-m_3)-2$
cyclic subgroups,
ten generated by elements of order at most $4$, this group has $2^{m_3}(3+m_2-m_3)-2$ \B{Q}-classes and
$10+2[2^{m_3}(3+m_2-m_3)-2-10]=2^{m_3+1}(3+m_2-m_3)-14$
$F$-classes.  \lemref{lem:xclasses} gives $N_2=2^{m_3}$ \B{Q}-classes
and $M_2=2^{m_3+1}$ $F$-classes involving $x$.
The situation with $xy$ is the same as with $x$, so $N_4=N_2$ and $M_4=M_2$.
Turning to $y$, we have $\langle \C{Z}(G),y \rangle/\langle s\rangle
=\langle\overline{t_2},\overline{y}\rangle
\iso C_{2^{m_2}}\times C_{2^{m_3+1}}$
with $2^{m_3+1}[3+m_2-(m_3+1)]-2=2^{m_3+1}(2+m_2-m_3)-2$ cyclic subgroups, ten generated by elements of order at most $4$.
So there are $2^{m_3+1}(2+m_2-m_3)-2$ \B{Q}-classes and
$10+2[2^{m_3+1}(2+m_2-m_3)-2-10]=2^{m_3+2}(2+m_2-m_3)-14$ $F$-classes in this group
and hence $N_3=2^{m_3}(1+m_2-m_3)$ \B{Q}-classes and $M_3=2^{m_3+1}(1+m_2-m_3)$
$F$-classes involving $y$.  In all, there are $N=N_1+2N_2+N_3=2^{m_3}(9+3m_2-3m_3)-4$ \B{Q}-classes
and $M=M_1+2M_2+M_3=2^{m_3+1}(9+3m_2-3m_3)-28$ $F$-classes.
\par Finally, as $G/G'\iso C_{2^{m_2+1}}\times C_{2^{m_3+1}}$, the commutative part of $\B{Q}G$
is the direct sum of $N_0=2^{m_3+1}[3+(m_2+1)-(m_3+1)]-2=2^{m_3+1}(3+m_2-m_3)-2$ fields, while the commutative part of $FG$
is the direct sum of $M_0=10+2(N_0-10)=2N_0-10=2^{m_3+2}(3+m_2-m_3)-14$ fields.
We conclude that $\B{Q}G$ is the direct sum
of $N_0$ fields and $N-N_0=2^{m_3}(3+m_2-m_3)-2$ quaternion algebras,
while $FG$ is the direct sum of $M_0$ fields
and $M-M_0=2^{m_3+1}(3+m_2-m_3)-14$ quaternion algebras.
\smallskip

Suppose $m_2=m_3$.
We have $s=t_1$ and
$\C{Z}(G)\iso C_2\times C_{2^{m_2}}\times C_{2^{m_2}}$, a group with $N_1=2[2^{m_2}(3)-2]
=2^{m_2+1}(3)-4$ cyclic subgroups, $20$ of which are generated by elements of order at most $4$.
So the centre contains
$N_1$ \B{Q}-classes and $M_1=20+2(N_1-20)=2^{m_2+2}(3)-28$ $F$-classes.
Now $\langle \C{Z}(G),x \rangle/\langle s\rangle=\langle\overline{x},\overline{t_3}\rangle
\iso C_{2^{m_2+1}}\times C_{2^{m_2}}$
has $2^{m_2}(3+1)=2^{m_2+2}$ cyclic subgroups, ten generated by elements of order at most $4$.
So there are $2^{m_2+2}$ \B{Q}-classes and $10+2[2^{m_2+2}-10]=
2^{m_2+3}-10$ $F$-classes in this group.
As $\C{Z}(G)/\langle s\rangle\iso C_{2^{m_2}}\times C_{2^{m_3}}$ has $2^{m_2}(3)$
cyclic subgroups,
ten generated by elements of order at most $4$, this group has $2^{m_2}(3)$ \B{Q}-classes and
$10+2(2^{m_2}(3)-10)=2^{m_2+1}(3)-10$
$F$-classes.  \lemref{lem:xclasses} gives $N_2=2^{m_2}$ \B{Q}-classes
and $M_2=2^{m_2+1}$ $F$-classes involving $x$.
Since $\langle\C{Z}(G),xy\rangle/\langle s\rangle=\langle \overline{xy},\overline{t_3}\rangle$,
the situation with $xy$ is the same as with $x$.
We have $N_4=N_2$ and $M_4=M_2$.
Turning to $y$, we have $\langle \C{Z}(G),y \rangle/\langle s\rangle
=\langle\overline{t_2},\overline{y}\rangle
\iso C_{2^{m_2}}\times C_{2^{m_2+1}}$
with $2^{m_2}(3+1)-2=2^{m_2+2}-2$ cyclic subgroups, ten generated by elements of order at most $4$.
So there are $2^{m_2+2}-2$ \B{Q}-classes and
$10+2[2^{m_2+2}-2-10]=2^{m_2+3}-14$ $F$-classes in this group
and hence $N_3=2^{m_2}-2$ \B{Q}-classes and $M_3=2^{m_2+1}-4$
$F$-classes involving $y$.  In all, there are $N=N_1+2N_2+N_3=2^{m_2}(9)-6$ \B{Q}-classes
and $M=M_1+2M_2+M_3=2^{m_2+1}(9)-32$ $F$-classes.
\par Finally, as $G/G'\iso C_{2^{m_2+1}}\times C_{2^{m_2+1}}$, the commutative parts of $\B{Q}G$
and of $FG$ are as before:  $N_0=2^{m_2+1}(3)-2$ and $M_0=2^{m_2+2}(3)-14$.
We conclude that $\B{Q}G$ is the direct sum
of $N_0$ fields and $N-N_0=2^{m_2}(3)-4$ quaternion algebras,
while $FG$ is the direct sum of $M_0$ fields
and $M-M_0=2^{m_2+1}(3)-18$ quaternion algebras.

\medskip\noindent (iv.a) Assume $m_1=2$, $m_2=m_3=1$.
Then  $s=t_1^2$ and $\C{Z}(G)\iso C_4\times C_2\times C_2$ has
$12$ cyclic subgroups, all of suitably small order, so the centre
has $12$ $K$-classes whether  $K=\B{Q}$ or a field of order $q\equiv3\pmod8$.
Now, $\langle \C{Z}(G),x \rangle/\langle s\rangle
\iso C_2\times C_4\times C_2$ has twelve cyclic subgroups and
$\C{Z}(G)/\langle s\rangle\iso C_2 \times C_2\times C_2$ has eight,
so there are four $K$-classes involving $x$.
Similarly, there are also four classes involving $y$ and four involving $xy$
(because $(xy)^2=st_2t_3$), giving $12+3(4)=24$ $K$-classes in all.
Since, $G/G'\iso C_2\times C_4\times C_4$ has $20$ $K$-classes,
it follows that  $KG$ is the direct sum of $20$ fields and four quaternion algebras.

\medskip\noindent (iv.b)  Assume $m_1\ge3$, $m_2=m_3=1$.
Here $\C{Z}(G)\iso C_{2^{m_1}}\times C_2\times C_2$ has
$4(m_1+1)=4m_1+4$ cyclic subgroups, twelve of order at most $4$,
so the centre has $N_1=4m_1+4$ \B{Q}-classes and $M_1=12+2(N_1-12)=8m_1-4$ $F$-classes.
Now, $\langle \C{Z}(G),x \rangle/\langle s\rangle
\iso C_{2^{m_1-1}}\times C_4\times C_2$ has $2(4)[3+(m_1-1)-2)-2]=8m_1-4$ cyclic subgroups,
$20$ of order at most $4$, so this group has $8m_1-4$ \B{Q}-classes and
$20+2(8m_1-4-20)=16m_1-28$ $F$-classes.
The group $\C{Z}(G)/\langle s\rangle\iso C_{2^{m_1-1}} \times C_2\times C_2$ has
$4m_1$ cyclic subgroups, of which eight are generated by elements of order at most $4$, giving
$4m_1$ \B{Q}-classes and $8+2(4m_1-8)=8m_1-8$ $F$-classes.  Using \lemref{lem:xclasses},
we see that there are $N_2=4m_1-4$ \B{Q}-classes and $M_2=8m_1-20$ $F$-classes
involving $x$.  The numbers are the same for classes involving $y$ and $xy$.
In all, we have $N=N_1+3N_2=16m_1-8$ \B{Q}-classes and $M=M_1+3M_2=32m_1-64$ $F$-classes.
Since, $G/G'\iso C_{2^{m_1-1}}\times C_4\times C_4$ has $\tfrac73(16-1)+(m_1-1-2)2^4+1=16m_1-12$
cyclic subgroups, $36$ of which are generated by elements of order at most $4$, there are $N_0=16m_1-12$
\B{Q}-classes in $G/G'$ and $M_0=36+2(N_0-36)=32m_1-60$ $F$-classes.
It follows that  $\B{Q}G$ is the direct sum of $16m_1-12$ fields and $N-N_0=4$
quaternion algebras while $FG$ is the direct sum of $32m_1-60$ fields and
four quaternion algebras.

\medskip\noindent (v.a) Suppose $m_1=2$, $m_2=1$, $m_3=2$. (The conclusions
when $m_1=2$, $m_2=2$, $m_3=1$ are the same.)
Here we have $\C{Z}(G)=C_4\times C_2\times C_4$ and $s=t_1^2$.
The centre has $20$ cyclic subgroups all of order at most $4$, so there
are $20$ \B{Q}-classes and $20$ $F$-classes.
The group $\langle \C{Z}(G),x\rangle/\langle s\rangle
=\langle\overline{t_1},\overline{x},\overline{t_3}\rangle
\iso C_2\times C_4\times C_4$ has
$20$ cyclic subgroups and $\C{Z}(G)/\langle s\rangle\iso C_2\times C_2\times C_4$ has $12$,
so there are eight $K$-classes involving $x$ where $K$ is any field
of characteristic different from $2$.
Also $\langle \C{Z}(G),y\rangle/\langle s\rangle\iso\times C_2\times C_2\times C_8$ has
$16$ cyclic subgroups, $12$ of which are generated by elements of order at most $4$.
This gives four \B{Q}-classes and $12+2(4)-12=8$ $F$-classes involving $y$.
Since $(xy)^2=st_2t_3$, $\langle \C{Z}(G),xy \rangle/\langle s\rangle
=\langle\overline{t_1},\overline{t_2},\overline{xy}\rangle
\iso C_2\times C_2 \times C_8$ has $16$ cyclic subgroups, $12$ of which are generated by elements of order at most $4$.  This gives another four \B{Q}-classes involving $xy$
and $12+2(4)-12=8$ $F$-classes.  In all, this group has $20+8+4+4=36$ \B{Q}-classes
and $20+8+8+8=44$ $F$-classes.

Since $G/G'\iso C_2\times C_2\times C_4$, we see that $K[G/G']$ has $12$ simple components
whether $K=\B{Q}$ or $K=F$.  Thus $\B{Q}G$ is the direct
sum of $12$ fields and $24$ quaternion algebras
while $FG$ is the direct sum of $12$ fields and $34$ quaternion algebras.

\medskip\noindent (v.b) Suppose $m_1\ge3$, $m_2=1$, $m_3=2$. (By symmetry, the
case $m_1\ge3$, $m_2=2$, $m_3=1$ will yield the same results.)
Here we have $\C{Z}(G)=C_{2^{m_1}}\times C_2\times C_4$, a group with
$2[4(3+m_1-2)-2]=8m_1+4$ cyclic subgroups, $20$ of order at most $4$, so there
are $N_1=8m_1+4$ \B{Q}-classes and $M_1=20+2(8m_1+4-20)=16m_1-12$ $F$-classes.
The group $\langle \C{Z}(G),x\rangle/\langle s\rangle
=\langle\overline{t_1},\overline{x},\overline{t_3}\rangle
\iso C_{2^{m_1}}\times C_4\times C_4$ has
$\tfrac73(16-1)+(m_1-2)2^4+1=16m_1+4$ cyclic subgroups, of which $36$ are generated by elements of order at
most $4$, we have $16m_1+4$ \B{Q}-classes and $36+2(16m_1+4-36)=32m_1-32$ $F$-classes.
Since $\C{Z}(G)/\langle s\rangle\times C_{2^{m_1-1}}\times C_2\times C_4$ has
$2[4(3+m_1-1-2)-2]=8m_1-4$ cyclic subgroups, of which $20$ are generated by elements of order at
most $4$, it has $8m_1-4$ \B{Q}-classes and $20+2(8m_1-4-20)=16m_1-28$ $F$-classes.
By \lemref{lem:xclasses}, there are $N_2=8m_1+8$ and $M_2=16m_1-4$ $F$-classes
involving $x$.
Now $\langle \C{Z}(G),y\rangle/\langle s\rangle\iso C_{2^{m_1-1}}\times C_2\times C_8$ has
$2[8(3+m_1-1-3)-2]=16m_1-20$ cyclic subgroups, $20$ of which are generated by elements of order at most $4$.
This implies $16m_1-20$ \B{Q}-classes and $20+2(16m_1-20-20)=32m_1-60$ $F$-classes,
so there are $N_3=8m_1-16$ \B{Q}-classes and $M_3=16m_1-32$ $F$-classes involving $y$.
Since $(xy)^4=t_3^2$, $\langle \C{Z}(G),xy \rangle/\langle s\rangle
=\langle\overline{t_1},\overline{t_2},\overline{xy}\rangle
\iso C_{2^{m_1-1}}\times C_2 \times C_8$, the numbers $N_4,M_4$
are, respectively, $N_3,M_3$.  In all, this group has $N=N_1+N_2+2N_3=32m_1-20$ \B{Q}-classes
and $M=M_1+M_2+2M_3=64m_1-80$ $F$-classes.

Since $G/G'\iso C_{2^{m_1-1}}\times C_2\times C_4$ has $2[4(3+m_1-1-2)-2]=8m_1-4$
cyclic subgroups, of which $20$ are generated by elements of order at most $4$,
we see that there are $N_0=8m_1-4$ \B{Q}-classes and $M_0=20+2(N_0-20)=16m_1-28$ $F$-classes
in $G/G'$.  It follows that $\B{Q}G$ is the direct
sum of $8m_1-4$ fields and $24m_1-16$ quaternion algebras
while $FG$ is the direct sum of $16m_1-28$ fields and $48m_1-52$ quaternion algebras.

\medskip\noindent (v.c) Suppose $m_1=2$, $m_2=1$, $m_3\ge3$.
Here we have $\C{Z}(G)=C_4\times C_2\times C_{2^{m_3}}$ and $s=t_1^2$.
The centre has $2[4(3+m_3-2)-2]=8m_3+4$ cyclic subgroups, of which $20$
are generated by elements of order at most $4$, so there
are $N_1=8m_3+4$ \B{Q}-classes and $M_1=20+2(N_1-20)=2N_1-20=16m_3-12$ $F$-classes.
The group $\langle \C{Z}(G),x\rangle/\langle s\rangle
=\langle\overline{t_1},\overline{x},\overline{t_3}\rangle
\iso C_2\times C_4\times C_{2^{m_3}}$ has
$2[4(3+m_3-2)-2]=8m_3+4$ cyclic subgroups, of which $20$ are generated by elements of order
at most
$4$, giving $8m_3+4$ \B{Q}-classes and $20+2(8m_3+4-20)=16m_3-12$ $F$-classes.
The group
$\C{Z}(G)/\langle s\rangle\times C_2\times C_2\times C_{2^{m_3}}$ has $4(m_3+1)$ cyclic subgroups,
of which $12$ are generated by elements of order at most $4$,
so there are $4m_3+4$ \B{Q}-classes and $12+2(4m_3+4-12)=8m_3-4$ $F$-classes.
By \lemref{lem:xclasses}, we have $N_2=4m_3$ \B{Q}-classes and $M_2=8m_3-8$ $F$-classes.
Now $\langle \C{Z}(G),y\rangle/\langle s\rangle\iso C_2\times C_2\times C_{2^{m_3+1}}$
and  $\langle \C{Z}(G),xy \rangle/\langle s\rangle
=\langle\overline{t_1},\overline{t_2},\overline{xy}\rangle$ is the same group.
Each has
$4(m_3+2)$ cyclic subgroups, $12$ of which are generated by elements of order at most $4$.
This gives $4m_3+8$ \B{Q}-classes and $12+2(4m_3+8-12)=8m_3+4$ $F$-classes in each case,
hence $N_3=N_4=4$ \B{Q}-classes and $M_3=M_4=8$ $F$-classes involving $y$.
In all, this group has $N=N_1+N_2+2N_3=12m_3+12$ \B{Q}-classes
and $M=M_1+M_2+2M_3=24m_3-4$ $F$-classes.

Since $G/G'\iso C_2\times C_4\times C_{2^{m_3+1}}$
has $N_0=2[4(3+m_3+1-2)-2]=8m_3+12$ cyclic subgroups, of which $20$
are generated by elements of order at most
$4$, this group has $M_0=20+2(N_0-20)=16m_3+4$ $F$-classes.
It follows that $\B{Q}G$ is the direct
sum of $8m_3+12$ fields and $4m_3$ quaternion algebras
while $FG$ is the direct sum of $16m_3+4$ fields and $8m_3-8$ quaternion algebras.

\smallskip\noindent (v.d) Suppose $m_1\ge3$, $m_2=1$ and
$m_3\ge3$. (The numbers when $m_1\ge3$, $m_2\ge3$ and $m_3=1$
are the same, by symmetry.)  There are several subcases.

First suppose $m_1\ge m_3+2$.  We have $\C{Z}(G)=C_{2^{m_1}}\times C_2\times C_{2^{m_3}}$
which contains $N_1=2[2^{m_3}(3+m_1-m_3)-2]=2^{m_3+1}(3+m_1-m_3)-4$ cyclic subgroups,
of which $20$ are generated by elements of order at most $4$.  So we have $N_1$ \B{Q}-classes
and $M_1=20+2(N_1-20)=2^{m_3+2}(3+m_1-m_3)-28$ $F$-classes.
Let $H=\langle \C{Z}(G),x\rangle$. Then $H/\langle s\rangle=\langle\overline{t_1},\overline{x},\overline{t_3}\rangle
\iso C_{2^{m_1-1}}\times C_4\times C_{2^{m_3}}$, a group containing $\tfrac73(15)
+3(4)(2^{m_3}-4)+(m_1-1-m_3)2^{m_3+2}+1=2^{m_3+2}(2+m_1-m_3)-12$ cyclic subgroups, and hence this number
of \B{Q}-classes, $36$ of which correspond to subgroups
generated by elements of order at most $4$.
This gives $2^{m_3+3}(2+m_1-m_3)-60$ $F$-classes in $H/\langle s\rangle$.
Now $\C{Z}(G)/\langle s\rangle\iso C_{2^{m_1-1}}\times C_2\times C_{2^{m_3}}$
has $2^{m_3+1}(3+m_1-1-m_3)-4=2^{m_3+1}(2+m_1-m_3)-4$ cyclic subgroups, and hence this number
of \B{Q}-classes, $20$ of which correspond to subgroups
generated by elements of order at most $4$.
This gives $2^{m_3+2}(2+m_1-m_3)-28$ $F$-classes in $\C{Z}(G)/\langle s\rangle$.
Using \lemref{lem:xclasses}, we get $N_2=2^{m_3+1}(2+m_1-m_3)-8$ \B{Q}-classes
involving $x$ and $M_2=2^{m_3+2}(2+m_1-m_3)-52$ $F$-classes.

Let $H=\langle \C{Z}(G),y\rangle$. Then $H/\langle s\rangle=\langle\overline{t_1},\overline{t_2},\overline{y}\rangle
\iso C_{2^{m_1-1}}\times C_2\times C_{2^{m_3+1}}$, a group containing
$2\{2^{m_3+1}[3+(m_1-1)-(m_2+1)]-2\}=2^{m_3+2}(1+m_1-m_3)-4$ cyclic subgroups, and hence this number
of \B{Q}-classes, $20$ of which correspond to subgroups
that are generated by elements of order at most $4$.
This gives $2^{m_3+3}(1+m_1-m_3)-28$ $F$-classes in $H/\langle s\rangle$.
Using \lemref{lem:xclasses}, we get $N_3=2^{m_3+1}(m_1-m_3)$ \B{Q}-classes
involving $y$ and $M_3=2^{m_3+2}(m_1-m_3)$ $F$-classes.
The situation is the same with $xy$, so the numbers $N_4,M_4$ of \B{Q}- and $F$-classes
involving $xy$ are $N_3,M_3$, respectively.
In all, we have $N=N_1+N_2+2N_3=2^{m_3+1}(5+4m_1-4m_3)-12$ \B{Q}-classes
and $M=M_1+M_2+2M_3=2^{m_3+2}(5+4m_1-4m_3)-80$ $F$-classes.

Finally, we have $G/G'=\langle \overline{t_1},\overline{x},\overline{y}\rangle
\iso C_{2^{m_1-1}}\times C_4\times C_{2^{m_3+1}}$, which has
$N_0=\tfrac73(16-1)+3(4)(2^{m_3+1}-4)+[(m_1-1)-(m_3+1)]2^{2+m_3+1}+1=2^{m_3+3}(1+m_1-m_3)-12$
cyclic subgroups, of which $36$ are generated by elements of order at most $4$, so
the number of \B{Q}-classes in $G/G'$ is $N_0$ and the number of $F$-classes
is $M_0=36+2(N_0-36)=2^{m_3+4}(1+m_1-m_3)-60$.

We find that $\B{Q}G$ is the direct sum of $N_0$ fields and $N-N_0=2^{m_3+1}$
quaternion algebras, while $FG$ is the direct sum of $M_0$ fields and $M-M_0=2^{m_3+2}-20$
quaternion algebras.

\smallskip
Now suppose $m_1=m_3+2$.  We have $\C{Z}(G)=C_{2^{m_3+2}}\times C_2\times C_{2^{m_3}}$
which contains $N_1=2[2^{m_3}(3+m_3+2-m_3)-2]=2^{m_3+1}(5)-4$ cyclic subgroups,
of which $20$ are generated by elements of order at most $4$.  So we have $N_1$ \B{Q}-classes
and $M_1=20+2(N_1-20)=2^{m_3+2}(5)-28$ $F$-classes.
Let $H=\langle \C{Z}(G),x\rangle$. Then $H/\langle s\rangle=\langle\overline{t_1},\overline{x},\overline{t_3}\rangle
\iso C_{2^{m_3+1}}\times C_4\times C_{2^{m_3}}$, a group containing $\tfrac73(15)
+3(4)(2^{m_3}-4)+2^{m_3+2}+1=2^{m_3+3}-12$ cyclic subgroups, and hence this number
of \B{Q}-classes, $36$ corresponding to subgroups generated by elements of order at most $4$.
This gives $2^{m_3+4}-60$ $F$-classes in $H/\langle s\rangle$.
Now $\C{Z}(G)/\langle s\rangle\iso C_{2^{m_3+1}}\times C_2\times C_{2^{m_3}}$
has $2^{m_3+3}-4$ cyclic subgroups, and hence this number
of \B{Q}-classes, $20$ corresponding to subgroups generated by elements of order at most $4$.
This gives $2^{m_3+4}-28$ $F$-classes in $\C{Z}(G)/\langle s\rangle$.
Using \lemref{lem:xclasses}, we get $N_2=2^{m_3+3}-8$ \B{Q}-classes
involving $x$ and $M_2=2^{m_3+4}-52$ $F$-classes.

Let $H=\langle \C{Z}(G),y\rangle$. Then $H/\langle s\rangle=\langle\overline{t_1},\overline{t_2},\overline{y}\rangle
\iso C_{2^{m_3+1}}\times C_2\times C_{2^{m_3+1}}$, a group containing
$6(2^{m_3+1})-4=3(2^{m_3+2})-4$ cyclic subgroups, and hence this number
of \B{Q}-classes, $20$ of which correspond to subgroups
generated by elements of order at most $4$.
This gives $3(2^{m_3+3})-28$ $F$-classes in $H/\langle s\rangle$.
Using \lemref{lem:xclasses}, we get $N_3=2^{m_3+2}$ \B{Q}-classes
involving $y$ and $M_3=2^{m_3+3}$ $F$-classes.
The situation is the same with $xy$, so the numbers $N_4,M_4$ of \B{Q}- and $F$-classes,
involving $xy$, respectively, are the same as $N_3,M_3$.
In all, we have $N=N_1+N_2+2N_3=13(2^{m_3+1})-12$ \B{Q}-classes
and $M=M_1+M_2+2M_3=13(2^{m_3+2})-80$ $F$-classes.

Finally, we have $G/G'=\langle \overline{t_1},\overline{x},\overline{y}\rangle
\iso C_{2^{m_3+1}}\times C_4\times C_{2^{m_3+1}}$, which has
$N_0=\tfrac73(16-1)+3(4)(2^{m_3+1}-4)+1=12(2^{m_3+1})-12$
cyclic subgroups, of which $36$ are generated by elements of order at most $4$, so
the number of \B{Q}-classes in $G/G'$ is $N_0$ and the number of $F$-classes
is $M_0=36+2(N_0-36)=24(2^{m_3+1})-60$.

We find the $\B{Q}G$ is the direct sum of $N_0$ fields and $N-N_0=2^{m_3+1}$
quaternion algebras, while $FG$ is the direct sum of $M_0$ fields and $M-M_0=2^{m_3+2}-20$
quaternion algebras.

\smallskip\par
Finally, suppose $m_3\ge m_1$.
We have $\C{Z}(G)=C_{2^{m_1}}\times C_2\times C_{2^{m_3}}$
which contains $N_1=2^{m_1+1}(3+m_3-m_1)-4$ cyclic subgroups,
of which $20$ are generated by elements of order at most $4$.  So we have $N_1$ \B{Q}-classes
and $M_1=20+2(N_1-20)=2^{m_1+2}(3+m_3-m_1)-28$ $F$-classes.
Let $H=\langle \C{Z}(G),x\rangle$.
Then $H/\langle s\rangle=\langle\overline{t_1},\overline{x},\overline{t_3}\rangle
\iso C_{2^{m_1-1}}\times C_4\times C_{2^{m_3}}$, a group containing
$2^{m_1+1}(4+m_3-m_2)-12$ cyclic subgroups, and hence this number
of \B{Q}-classes, $36$ corresponding to subgroups generated by elements of order at most $4$.
This gives $2^{m_1+2}(4+m_3-m_1)-60$ $F$-classes in $H/\langle s\rangle$.
Now $\C{Z}(G)/\langle s\rangle\iso C_{2^{m_1-1}}\times C_2\times C_{2^{m_3}}$
has $2^{m_1}(4+m_3-m_1)-4$ cyclic subgroups, and hence this number
of \B{Q}-classes, $20$ of which correspond to subgroups
that are generated by elements of order at most $4$.
This gives $2^{m_1+1}(4+m_3-m_1)-28$ $F$-classes in $\C{Z}(G)/\langle s\rangle$.
Using \lemref{lem:xclasses}, we get $N_2=2^{m_1}(4+m_3-m_1)-8$ \B{Q}-classes
involving $x$ and $M_2=2^{m_1+1}(4+m_3-m_1)-32$ $F$-classes.

Let $H=\langle \C{Z}(G),y\rangle$. Then $H/\langle s\rangle=\langle\overline{t_1},\overline{t_2},\overline{y}\rangle
\iso C_{2^{m_1-1}}\times C_2\times C_{2^{m_3+1}}$, a group containing
$2^{m_1}(5+m_3-m_1)-4$ cyclic subgroups, and hence this number
of \B{Q}-classes, $20$ corresponding to subgroups generated by elements of order at most $4$.
This gives $2^{m_1+1}(5+m_3-m_1)-28$ $F$-classes in $H/\langle s\rangle$.
Using \lemref{lem:xclasses}, we get $N_3=2^{m_1}$ \B{Q}-classes
involving $y$ and $M_3=2^{m_1+1}$ $F$-classes.
The situation is the same with $xy$, so the numbers $N_4,M_4$ of \B{Q}- and $F$-classes,
involving $xy$ are, respectively, $N_3,M_3$.
In all, we have $N=N_1+N_2+2N_3=2^{m_1}(12+3m_1-3m_3)-12$ \B{Q}-classes
and $M=M_1+M_2+2M_3=2^{m_1+1}(12+3m_1-3m_3)-60$ $F$-classes.

Finally, we have $G/G'=\langle \overline{t_1},\overline{x},\overline{y}\rangle
\iso C_{2^{m_1-1}}\times C_4\times C_{2^{m_3+1}}$, which has
$N_0=2^{m_1+1}(5+m_3-m_1)-12$
cyclic subgroups, of which $36$ are generated by elements of order at most $4$, so
the number of \B{Q}-classes in $G/G'$ is $N_0$ and the number of $F$-classes
is $M_0=36+2(N_0-36)=2^{m_1+2}(5+m_3-m_1)-60$.

We find the $\B{Q}G$ is the direct sum of $N_0$ fields and $N-N_0=2^{m_1}(2+m_3-m_1)$
quaternion algebras, while $FG$ is the direct sum of $M_0$ fields and $M-M_0=2^{m_1+1}(2+m_3-m_1)$
quaternion algebras.

\medskip\noindent (vi.a) Suppose $m_1=m_2=m_3=2$.
Thus $s=t_1^2$ and $\C{Z}(G)\iso C_4\times C_4\times C_4$ has $36$ cyclic subgroups all
of order at most $4$, so there are $36$ $K$-classes in the centre whether $K=\B{Q}$
or a field $F$ of order $q\equiv3\pmod8$.  The
groups $\langle \C{Z}(G),x \rangle/\langle s\rangle
=\langle \overline{t_1},\overline{x},\overline{t_3}\rangle$,
$\langle \C{Z}(G),y \rangle/\langle s\rangle
=\langle\overline{t_1},\overline{t_2},\overline{y}\rangle$
and $\langle \C{Z}(G),xy \rangle/\langle s\rangle
=\langle\overline{t_1},\overline{t_2},\overline{xy}\rangle$,
each being isomorphic to $C_2\times C_8\times C_4$, have $28$ cyclic subgroups each,
$20$ of order at most $4$.  So each has $28$ \B{Q}-classes and $20+2(8)=36$ $F$-classes.
The group
$\C{Z}(G)/\langle s\rangle \iso C_2 \times C_4\times C_4$ has $20$ $K$-classes,
whether $K=\B{Q}$ or $K=F$, so
there are eight \B{Q}-classes involving each of $x$, $y$, $xy$
and $16$ $F$-classes.  In all, the group algebra $\B{Q}G$
in this case has $36+3(8)=60$ simple components
while $FG$ has $36+3(16)=84$.

Also $G/G'\iso C_2\times C_8\times C_8$ has $44$ cyclic subgroups, $20$ of order
at most $4$, so the commutative part of $\B{Q}G$ is the direct sum
of $44$ fields and the commutative part of $FG$ is the direct sum
of $20+2(24)=68$ fields.  Thus
$\B{Q}G$ is the sum of $44$ fields and $16$ quaternion algebras
while $FG$ is the sum of $68$ fields and $16$ quaternion algebras.

\medskip\noindent (vi.b) Assume $m_1=2$, $m_2\ge3$ and $m_3=2$.  (By symmetry,
the case $m_1=2$, $m_2=2$, $m_3\ge3$ will yield the same results.)
We have $\C{Z}(G)= C_4\times C_{2^{m_2}}\times C_4$
with $N_1=\tfrac73(4^2-1)+(m_2-2)2^{2+2}+1=16m_2+4$
cyclic subgroups, $36$ of which are generated by elements of order at most $4$, so
there are $N_1$ \B{Q}-classes and $M_1=20+2(N_1-20)=32m_2-28$ $F$-classes.

Let $H=\langle \C{Z}(G),x\rangle$.  Then $H/\langle s\rangle=\langle\overline{t_1},\overline{x},\overline{t_3}\rangle
\iso C_2\times C_{2^{m_2+1}}\times C_4$ has $2[4(3+m_2+1-2)-2]=8m_2+12$
cyclic subgroups (and hence \B{Q}-classes),
of which $20$
are generated by elements of order at most $4$, so that are $16m_2+4$ $F$-classes here.
We have $\C{Z}(G)/\langle s\rangle\iso C_2\times C_{2^{m_2}}\times C_4$
with $8m_2+4$
cyclic subgroups (and hence \B{Q}-classes), of which $20$ are generated by elements of order at most $4$, giving
$16m_2-12$ $F$-classes.  With the help of \lemref{lem:xclasses},
we learn that there are $N_2=8$ \B{Q}-classes and $M_2=16$
$F$-classes involving~$x$.

Let $H=\langle \C{Z}(G),y\rangle$.  Then
$H/\langle s\rangle=\langle\overline{t_1},\overline{t_2},\overline{y}\rangle
\iso C_2\times C_{2^{m_2}}\times C_8$ has $16m_2-4$ cyclic subgroups (and hence \B{Q}-classes),
of which $20$
are generated by elements of order at most $4$, so that are $32m_2-28$ $F$-classes here.
From \lemref{lem:xclasses}, we discover $N_3=8m_2-8$ and $M_3=16m_2-16$.
The situation with $xy$ is similar.  The numbers $N_4,M_4$ are $N_3,M_3$, respectively.
In all we have $N=N_1+N_2+2N_3=32m_2-4$ \B{Q}-classes and
$M=M_1+M_2+2M_3=64m_2-44$ $F$-classes.

Now $G/G'\iso C_2\times C_{2^{m_2+1}}\times C_8$, so previous calculations
give $N_0=16m_2+14$ \B{Q}-classes and $M_0=32m_2+8$ $F$-classes.
In this situation, $\B{Q}G$ is the direct sum
of $N_0$ fields and $N-N_0=16m_2-18$ quaternion algebras
and $FG$ is the direct sum of $M_0$ fields and $M-M_0=32m_2-52$ quaternion
algebras.

\medskip\noindent (vi.c) Assume $m_1\ge3$, $m_2=2$ and $m_3\ge2$.  (The case $m_1\ge3$,
$m_2\ge3$, $m_3=2$ is the same, by symmetry.)

First suppose $m_1\ge m_3+2$.  We have $\C{Z}(G)= C_{2^{m_1}}\times C_4\times C_{2^{m_3}}$
with $N_1=\tfrac73(2^4-1)+3(4)(2^{m_3}-4)+(m_1-m_3)2^{m_3+2}+1=2^{m_3+2}(3+m_1-m_3)-12$
cyclic subgroups, $36$ of which are generated by elements of order at most $4$, so
there are $N_1$ \B{Q}-classes and $M_1=36+2(N_1-36)=2^{m_3+3}(3+m_1-m_3)-48$ $F$-classes.

Let $H=\langle \C{Z}(G),x\rangle$.  Then $H/\langle s\rangle=\langle\overline{t_1},\overline{x},\overline{t_3}\rangle
\iso C_{2^{m_1-1}}\times C_8\times C_{2^{m_3}}$ has $\tfrac73(4^3-1)+3(8)(2^{m_2}-8)
+(m_1-1-m_3)2^{m_3+3}+1=(2+m_1-m_3)2^{m_3+3}-44$ cyclic subgroups (and hence \B{Q}-classes),
of which $36$
are generated by elements of order at most $4$, so that are $2^{m_3+4}(2+m_1-m_3)-124$ $F$-classes here.
We have $\C{Z}(G)/\langle s\rangle\iso C_{2^{m_1-1}}\times C_4\times C_{2^{m_3}}$
with $\tfrac73(4^2-1)+3(4)(2^{m_3}-4)+(m_1-1-m_3)2^{m_3+2}+1=2^{m_3+2}(2+m_1-m_3)-12$
cyclic subgroups (and hence \B{Q}-classes), of which $36$ are generated by elements of order at most $4$, giving
$2^{m_3+3}(2+m_1-m_3)-60$ $F$-classes.  With the help of \lemref{lem:xclasses},
we obtain $N_2=2^{m_3+2}(2+m_1-m_3)-32$ \B{Q}-classes and $M_2=2^{m_3+3}(2+m_1-m_3)-64$
$F$-classes involving~$x$.

Let $H=\langle \C{Z}(G),y\rangle$.  Then
$H/\langle s\rangle=\langle\overline{t_1},\overline{t_2},\overline{y}\rangle
\iso C_{2^{m_1-1}}\times C_4\times C_{2^{m_3+1}}$ has $\tfrac73(4^2-1)+3(4)(2^{m_3+1}-4)
+[(m_1-1)-(m_3+1)]2^{m_3+1+2}+1=(1+m_1-m_3)2^{m_3+3}-12$ cyclic subgroups (and hence \B{Q}-classes),
of which $36$
are generated by elements of order at most $4$, so that are $2^{m_3+4}(1+m_1-m_3)-60$ $F$-classes here.
By \lemref{lem:xclasses},
we have $N_3=2^{m_3+2}(m_1-m_3)$ \B{Q}-classes and $M_3=2^{m_3+3}(m_1-m_3)$
$F$-classes involving $y$.
The situation with $xy$ is similar.  The numbers $N_4,M_4$ are $N_3,M_3$, respectively.
In all we have $N=N_1+N_2+2N_3=2^{m_3+2}(5+4m_1-4m_3)-44$ \B{Q}-classes and
$M=M_1+M_2+2M_3=2^{m_3+3}(5+4m_1-4m_3)-112$ $F$-classes.

Now $G/G'\iso C_{2^{m_1-1}}\times C_8\times C_{2^{m_3}}$, so previous calculations
give $N_0=2^{m_3+3}(2+m_1-m_3)-44$ \B{Q}-classes and $M_0=2^{m_3+4}(2+m_1-m_3)-124$ $F$-classes.
In this situation, we see that $\B{Q}G$ is the direct sum
of $N_0$ fields and $N-N_0=2^{m_3+2}(1+2m_1-2m_3)$ quaternion algebras
and $FG$ is the direct sum of $M_0$ fields and $M-M_0=2^{m_3+3}(1+2m_1-2m_3)+12$ quaternion
algebras.

\smallskip  Now assume $m_1=m_3+1$.
We have $\C{Z}(G)= C_{2^{m_3+1}}\times C_4\times C_{2^{m_3}}$
with $N_1=2^{m_3+4}-12$
cyclic subgroups, $36$ of which are generated by elements of order at most $4$, so
there are $N_1$ \B{Q}-classes and $M_1=36+2(N_1-36)=2^{m_3+5}-48$ $F$-classes.

Let $H=\langle \C{Z}(G),x\rangle$.  Then $H/\langle s\rangle=\langle\overline{t_1},\overline{x},\overline{t_3}\rangle
\iso C_{2^{m_3}}\times C_8\times C_{2^{m_3}}$ has $(3)2^{m_3+3}-44$ cyclic subgroups
(and hence \B{Q}-classes), of which $36$
are generated by elements of order at most $4$, so that are $2^{m_3+4}(3)-124$ $F$-classes here.
We have $\C{Z}(G)/\langle s\rangle\iso C_{2^{m_3}}\times C_4\times C_{2^{m_3}}$
with $2^{m_3+2}(3)-12$
cyclic subgroups (and hence \B{Q}-classes), of which $36$ are generated by elements of order at most $4$, giving
$2^{m_3+3}(3)-60$ $F$-classes.  With the help of \lemref{lem:xclasses},
we have $N_2=3(2^{m_3+2})-32$ \B{Q}-classes and $M_2=3(2^{m_3+3})-64$
$F$-classes involving~$x$.

Let $H=\langle \C{Z}(G),y\rangle$.  Then
$H/\langle s\rangle=\langle\overline{t_1},\overline{t_2},\overline{y}\rangle
\iso C_{2^{m_3}}\times C_4\times C_{2^{m_3+1}}$ has $2^{m_3+4}-12$ cyclic subgroups (and hence \B{Q}-classes),
of which $36$
are generated by elements of order at most $4$, so that are $2^{m_3+5}-60$ $F$-classes here.
By \lemref{lem:xclasses},
we have $N_3=2^{m_3+2}$ \B{Q}-classes and $M_3=2^{m_3+3}$
$F$-classes involving $y$.
The situation with $xy$ is similar.  The numbers $N_4,M_4$ are $N_3,M_3$, respectively.
In all we have $N=N_1+N_2+2N_3=9(2^{m_3+2})-44$ \B{Q}-classes and
$M=M_1+M_2+2M_3=9(2^{m_3+3})-112$ $F$-classes.

Now $G/G'\iso C_{2^{m_3}}\times C_8\times C_{2^{m_3}}$, so previous calculations
give $N_0=3(2^{m_3+3})-44$ \B{Q}-classes and $M_0=3(2^{m_3+4})-124$ $F$-classes.
In this situation, we see that $\B{Q}G$ is the direct sum
of $N_0$ fields and $N-N_0=3(2^{m_3+2})$ quaternion algebras,
and $FG$ the direct sum of $M_0$ fields and $M-M_0=3(2^{m_3+3})+12$ quaternion
algebras.

\smallskip Finally assume $m_1\le m_3$.
We have $\C{Z}(G)=C_{2^{m_1}}\times C_4\times C_{2^{m_3}}$
with $N_1=2^{m_1+2}(3+m_3-m_1)-12$
cyclic subgroups, $36$ of which are generated by elements of order at most $4$, so
there are $N_1$ \B{Q}-classes and $M_1=2^{m_1+3}(3+m_3-m_1)-60$ $F$-classes.

Let $H=\langle \C{Z}(G),x\rangle$.  Then $H/\langle s\rangle
=\langle\overline{t_1},\overline{x},\overline{t_3}\rangle
\iso C_{2^{m_1-1}}\times C_8\times C_{2^{m_3}}$ has $(4+m_3-m_1)2^{m_1+2}-44$
cyclic subgroups (and hence \B{Q}-classes), of which $36$
are generated by elements of order at most $4$, so that are $2^{m_1+3}(4+m_3-m_1)-124$ $F$-classes here.
We have $\C{Z}(G)/\langle s\rangle\iso C_{2^{m_1-1}}\times C_4\times C_{2^{m_3}}$
with $2^{m_1+1}(4+m_3-m_1)-12$
cyclic subgroups (and hence \B{Q}-classes), of which $36$ are generated by elements of order at most $4$, giving
$2^{m_1+2}(4+m_3-m_1)-60$ $F$-classes.  With the help of \lemref{lem:xclasses},
we discover $N_2=2^{m_1+1}(4+m_3-m_1)-32$ \B{Q}-classes and $M_2=2^{m_1+2}(4+m_3-m_1)-64$
$F$-classes involving~$x$.

Let $H=\langle \C{Z}(G),y\rangle$.  Then
$H/\langle s\rangle=\langle\overline{t_1},\overline{t_2},\overline{y}\rangle
\iso C_{2^{m_1-1}}\times C_4\times C_{2^{m_3+1}}$ has $2^{m_1+1}(5+m_3-m_1)-12$
cyclic subgroups (and hence \B{Q}-classes), of which $36$
are generated by elements of order at most $4$, so that are $2^{m_1+2}(5+m_3-m_1)-60$ $F$-classes here.
By \lemref{lem:xclasses},
we obtain $N_3=2^{m_1+1}$ \B{Q}-classes and $M_3=2^{m_1+2}$
$F$-classes involving $y$.
The situation with $xy$ is similar.  The numbers $N_4,M_4$ are $N_3,M_3$, respectively.
In all we have $N=N_1+N_2+2N_3=2^{m_1+1}(12+3m_1-3m_3)-48$ \B{Q}-classes and
$M=M_1+M_2+2M_3=2^{m_1+2}(12+3m_3-3m_1)-124$ $F$-classes.

Now $G/G'\iso C_{2^{m_1-1}}\times C_8\times C_{2^{m_3}}$, so previous calculations
give $N_0=2^{m_1+1}(8+m_3-m_1)-44$ \B{Q}-classes and $M_0=2^{m_1+2}(8+m_3-m_1)-124$ $F$-classes.
In this situation, we see that $\B{Q}G$ is the direct sum
of $N_0$ fields and $N-N_0=2^{m_1+1}(4+2m_3-2m_1)-4$ quaternion algebras,
and $FG$ the direct sum of $M_0$ fields and $M-M_0=2^{m_1+2}(4+2m_3-2m_1)$ quaternion
algebras.

\smallskip

Turning to the general situation, we now assume that $G$ is a group
from class $\C{D}_5$ with $m_1\ge3$, $m_2\ge3$ and $m_3\ge3$.
\lemref{lem:3cyclics} allows to us to determine the number $N_1$  of cyclic
subgroups of $\C{Z}(G)$, which is also the number of \B{Q}-classes.
Since the centre contains $36$ subgroups
generated by elements of order at most $4$  (subgroups lying in a copy
of $C_4\times C_4\times C_4$), the number of $F$-classes in the
centre is $M_1=36+2(N_1-36)=2N_1-36$.

Let $\alpha$ be the number of cyclic subgroups in $\C{Z}(G)/\langle s\rangle
\iso C_{2^{m_1-1}}\times C_{2^{m_2}}\times C_{2^{m_3}}$.
Then this quotient contains $\alpha$ \B{Q}-classes and $36+2(\alpha-36)=2\alpha-36$
$F$-classes.  (Here is where the fact $m_1-1\ge2$ is crucial.)
Let $H=\langle \C{Z}(G),x\rangle=\langle t_1,x,t_3\rangle$.
and let $\beta$ be the number
of cyclic subgroups of
$H/\langle s\rangle\iso C_{2^{m_1-1}}\times C_{2^{m_2+1}}\times C_{2^{m_3}}$.
Thus $H/\langle s\rangle$ has $\beta$ \B{Q}-classes and $2\beta-36$ $F$-classes.
\lemref{lem:xclasses} then tells us there are $N_2=\beta-\alpha$
and $M_2=(2\beta-36)-(2\alpha-36)=2N_2$ $F$-classes involving~$x$.

Similarly, letting
$N_3$
and $N_4$  denote the numbers of \B{Q}-classes involving $y$ and $xy$, respectively
and as usual, we have $M_3=2N_3$ and $M_4=2N_4$ $F$-classes involving $x$
and $y$, respectively.
So $\B{Q}G$ is the direct sum of $N=N_1+N_2+N_3+N_4$ simple algebras, while
$FG$ is the direct sum of $M=M_1+M_2+M_3+M_4=(2N_1-36)+2N_2+2N_3+2N_4=2N-36$
simple algebras.

Let $N_0$ be the number of cyclic subgroups of $G/G'
=\langle\overline{t_1}\rangle\times\langle\overline{x}\rangle\times\langle\overline{y}\rangle
\iso C_{2^{m_1-1}}\times C_{2^{m_2+1}}\times C_{2^{m_3+1}}$.
Thus $G/G'$ has $N_0$ \B{Q}-classes and $M_0=2N_0-36$ $F$-classes
We conclude that
$\B{Q}G$ is the direct sum of $N_0$ fields and $N-N_0$ quaternion
algebras, while $FG$ is the direct sum of $M_0=2N_0-36$ fields
and $M-M_0=2(N-N_0)$ quaternion algebras.

\begin{thm} Let $G$ be a group of type $\C{D}_5$ with $m_1\ge3$, $m_2\ge3$,
$m_3\ge3$.   Denote by $N_1$, $N_2$,
$N_3$,  $N_4$ the number of cyclic subgroups of \C{Z}(G) and the numbers
of \B{Q}-classes involving $x$, $y$ and $xy$, respectively.  (Note that these numbers
can be determined by \lemref{lem:3cyclics} and \lemref{lem:xclasses}.)
Let $N=N_1+N_2+N_3+N_4$ and let $N_0$ be the number of cyclic subgroups of $G/G'$,
a number that is also readily available using \lemref{lem:3cyclics}.
Then $\B{Q}G$ is the direct sum of $N_0$ fields and $N-N_0$ quaternion algebras.
Let $K$ be a finite field of odd order $q$.  Then there are at least
$2N-36$ simple components in the Wedderburn decomposition of $KG$.
This minimal number is realized if $q\equiv3\pmod8$, in which case
$KG$ is the direct sum of $2N_0-36$  fields and $2(N-N_0)$
quaternion algebras, each necessarily a $2\times 2$ matrix ring.
\end{thm}

\section{The Structure of a Semisimple Finite-dimensional Alternative Loop Algebra}
An \emph{alternative ring}
is a ring that satisfies the right and left alternative laws, $(yx)x=yx^2$ and
$x(xy)=x^2y$.
A \emph{Moufang loop} is a loop that satisfies
the (right) Moufang identity $[(xy)z]y=x[y(zy)]$.
Any group is a Moufang loop and, in some respects, Moufang
loops in general are not far removed from groups.  They
are, for example, \emph{diassociative} (the subloop generated
by any two elements is always associative), and there are
a number of constructions of Moufang loops which consist
of ``doubling'' a group in various ways.  We describe
one of importance here.

Let $G$ be a nonabelian group with an involution $g\mapsto g^*$
(an antiautomorphism of order two) which is such that
$gg^*\in\C{Z}(G)$, the centre of $G$, for all $g\in G$.
Let $g_0\in\C{Z}(G)$ be an element fixed by $*$ and let $u$ be
an element not in $G$.  Extend the product in $G$ to the
set $L=G\cup Gu$ by means of the rules
\begin{align*}
g(hu) &= (hg)u \\
(gu)h &= (gh^*)u \\
(gu)(hu) &= g_0h^*g
\end{align*}
for $g,h\in G$.
Then $L$ is a Moufang loop denoted $M(G,*,g_0)$.
If $G/\C{Z}(G)\iso C_2\times C_2$,  then
the commutator subgroup $G'=\{1,s\}$
is central of order two and  the involution on $G$ is given by
\begin{equation}\label{eq:star}
g^* = \begin{cases}
         g & \text{if $g\in\C{Z}(G)$} \\
       sg & \text{if $g\notin\C{Z}(G)$.}
      \end{cases}
\end{equation}
In this situation,
the loop $L=M(G,*,g_0)$ is \emph{RA (ring alternative)} by which we mean
that over any (commutative associative) coefficient ring $R$
(with $1$), the loop ring $RL$ is alternative, but not associative.
In fact, this construction accounts for all RA loops
\cite[Theorem IV.3.1]{EGG:96}.
In any RA loop $M(G,*,g_0)$, the unique nonidentity
commutator of $G$ is a unique nonidentity commutator and
a unique nonidentity associator,
which we consistently label~$s$.

It transpires that there are exactly seven classes of indecomposable finite RA loops
\cite[Theorem V.3.1]{EGG:96}.  They are described in \tabref{tab:indecomps}.
In six of these classes, the groups defining
the loops come from one of the five classes
$\C{D}_1,\C{D}_2,\C{D}_3,\C{D}_4,\C{D}_5$ introduced in \secref{sec:intro}.
In the seventh class, the groups
are the direct products of a group in $\C{D}_5$ with a cyclic group.

\begin{table}
\footnotesize
\caption{The Seven Classes of Indecomposable RA Loops \label{tab:indecomps}}
\begin{equation*}
\begin{array}{|c|c|c|c|c|c|} \hline\hline
\text{\emph{Loop Class}} & \text{\emph{Centre}} & x^2 & y^2 & \text{\emph{Group Class}} & u^2=g_0 \\ \hline
\C{L}_1 & \langle t_1\rangle & 1 & 1  &  \C{D}_1  &   1 \\
\C{L}_2 & \langle t_1\rangle & t_1 & t_1 &  \C{D}_2  &   t_1 \\
\C{L}_3 & \langle t_1\rangle\times \langle t_2\rangle & 1 & t_2 &  \C{D}_3  &  1 \\
\C{L}_4 & \langle t_1\rangle\times \langle t_2\rangle & t_1 & t_2 &  \C{D}_4 &  t_1 \\
\C{L}_5 & \langle t_1\rangle\times \langle t_2\rangle\times \langle  t_3 \rangle
& t_2 & t_3 &  \C{D}_5 &   1 \\
\C{L}_6 & \langle t_1\rangle\times \langle t_2\rangle\times \langle  t_3\rangle
& t_2 & t_3 &  \C{D}_5 &   t_1  \\
\C{L}_7 & \langle t_1\rangle\times \langle t_2\rangle\times \langle  t_3\rangle\times \langle  t_4\rangle
& t_2 & t_3 &  \C{D}_5 \times \langle t_4\rangle  & t_4  \\ \hline\hline
\end{array}
\end{equation*}
\end{table}

Our goal now is to show how to apply the results of previous
sections in order to determine the number and nature of the simple components of $KL$,
$K$ a field of characteristic different from $2$, in many situations.
\medskip

Let $L=M(G,*,g_0)$ be an RA loop and let
$R$ denote a commutative, associative coefficient ring with $1$.  Then every
element of the loop ring $RL$ can be written $x+yu$ with $x,y\in RG$: we write
$RL=RG+RGu$.  The involution on $G$ extends to $RG$ via $(\sum\alpha_g g)^*
=\sum\alpha_g g^*$, and multiplication in $RL$ is given by
\begin{equation}\label{eq1}
(a+bu)(x+yu)=(ax+g_0y^*b)+(ya+bx^*)u
\end{equation}
for $a,b,x,y\in RG$.  Defining $\pi_G(x+yu)=x$ and $\pi_u(x+yu)=y$ and also, for
any subset $S$ of $RL$, $\pi_G(S)$ and $\pi_u(S)$ in the obvious way, we have
the following elementary result.

\begin{prop}\label{prop1} If $I$ is an ideal of $RL$, then $\pi_G(I)=\pi_u(I)$
is an ideal of $RG$ which is invariant under $*$.   Conversely, if $J$ is an
ideal of $RG$ which is invariant under $*$, then $I=J+Ju$ is an ideal of $RL$.
\end{prop}
\begin{proof}
Let $I$ be an ideal of $RL$.  The sets $I_G=\pi_G(I)$ and $I_u=\pi_u(I)$ are
certainly additive subgroups of $RG$.  Let $x_1\in I_G$.  Then there
exists $x=x_1+x_2u\in I$ for some $x_2\in RG$.  If $a\in RG$,
then $a(x_1+x_2u)=ax_1+(x_2a)u\in I$ and $(x_1+x_2u)a=x_1a+(x_2a^*)u\in I$,
so $ax_1$ and $x_1a$ are both in $I_G$.   Thus $I_G$ is an ideal of $RG$.
Also, $xu=g_0x_2+x_1u\in I$, so $x_1\in I_u$ giving $I_G\subseteq I_u$.
Now let $x_2\in I_u$.  There exists $x=x_1+x_2u\in I$ for some $x_1\in RG$.
As seen, $g_0x_2\subseteq I_G$, so $x_2=g_0^{-1}g_0x_2$ is in the ideal $I_G$.
This shows that $I_u\subseteq I_G$, so the sets are equal (and an ideal).
Finally, $u(x_1+x_2u)=g_0x_2^*+x_1^*u\in I$ implies
$I_G^*\subseteq I_u=I_G$.  The other inclusion follows from $I_G=(I_G^*)^*$.
\par Conversely, if $J$ is an ideal of $RG$ which is invariant under $*$,
then $I=J+Ju$ is an additive group and
the rule \eqref{eq1} for multiplication in $RL$ shows that $I$
is closed under multiplication on the right and left.
\end{proof}

Now let $L=M(G,*,g_0)$ be a finite RA loop
and suppose that $K$ is a field such that $\ch K\notdivides |G|$.  Then
$KG=A_1\oplus A_2\oplus\cdots\oplus A_n$
is the direct sum of simple rings which are minimal ideals of $KG$ \cite[Theorem VI.4.3]{EGG:96}.
Writing $1=e_1+e_2+\cdots +e_n$ as the sum of primitive
idempotents $e_i\in A_i$ and $s=s_1+s_2+\cdots +s_n$, $s_i\in A_i$,
then each $s_i$ is a central element that squares to $1=e_i$ in the simple ring $A_i$,
so $s_i=\gamma_i e_i$ with $\gamma_i=\pm1$.  (It will be important to remember
that the case $\gamma_i=+e_i$ occurs precisely when the component $A_i$
is commutative \cite[Proposition VI.4.6]{EGG:96}.)  Since $e_i^*=e_i$,
each $A_i=(KG)e_i$ is invariant under $*$,
so $B_i=A_i+A_iu$ is an ideal of $KL$ by \propref{prop1}, and
$KL=B_1\oplus B_2\oplus\cdots\oplus B_n$.

Clearly, for $a,b,x,y\in A_i$,  multiplication in $B_i$ is defined by
\begin{displaymath}
(a+bu)(x+yu)=(ax+\gamma y^*b)+(ya+bx^*)u, \quad a,b,x,y\in A_i,
\end{displaymath}
with $\gamma$ the projection of $g_0$ in $A_i$.

For $x\in KL$, define $n(x)=xx^*$.  This element
is central in $KL$ \cite[Corollary III.4.3]{EGG:96}
and so the restriction of $n$ to $A_i$ is a multiplicative
map---$n(xy)=n(x)n(y)$---into the field which is the centre of $A_i$.
Moreover $n$ is nondegenerate on $KL$ \cite[Proposition III.5.1]{EGG:96},
so its restriction is nondegenerate on $A_i$.  It follows that $A_i$
and $B_i$ are \emph{composition algebras}---see Sections
1 and 2 of \cite[Chapter 2]{ZS3:82}.  Since each $A_i$
is either a field or a quaternion algebra \cite[Corollary VI.4.8]{EGG:96},
there are two possibilities for the algebra $B_i$:
\begin{itemize}
\item if $A_i$ is a field,
$B_i=A_i\oplus A_i$ or a separable quadratic field;
\item if $A_i$ is a quaternion algebra, $B_i$
is a so-called \emph{Cayley algebra}, a certain $8$-dimensional algebra which is not associative.
\end{itemize}
In the case that $A_i$ is a field and $B_i=A_i\oplus A_i$, the algebra $B_i$
is called \emph{split}, a case guaranteed if $u^2$ is the square of an element in
$A_i$---see \cite[p.~17]{EGG:96} or \cite[p.~30]{ZS3:82}.
Any Cayley algebra is similarly either a division
algebra or a unique \emph{split} algebra, the latter occurring if and only if $B_i$
contains a nontrivial idempotent (there are other equivalent conditions).

We have established most of the content of this theorem.

\begin{thm}\label{thm:final} Let $L=M(G,*,g_0)=G\cup Gu$ be a finite RA loop and $K$
a field of characteristic not dividing $|L|$.
Then the group algebra $KG=\oplus A_i$ is
the direct sum of simple algebras $A_i$, each a field or a quaternion
algebra and the loop algebra  $KL=\oplus (A_i+A_iu)$ is the
direct sum of simple algebras $B_i=A_i+A_iu$, each a field, or
the direct sum $A_i\oplus A_i$ of two fields if $A_i$ is a field,
or a Cayley algebra.
\par Suppose that $L$ is a loop
in one of the five classes $\C{L}_1,\ldots,\C{L}_5$
and that there are $n$ fields and $m$ quaternion algebras in
the decomposition of $KG$.   Then
the loop algebra  $KL$ is the
direct sum of $2n$ fields and $m$ Cayley algebras.
\end{thm}
\begin{proof}
It remains just to prove the last statement.  In the cases of $\C{L}_1$, $\C{L}_3$
and $\C{L}_5$, we have $u^2=1$.  Let $\hat{u}=\tfrac12(1+u)$ and let $e$
be an idempotent defining the simple component $A_i$; that is, $A_i=(KG)e_i$.
Then $e\hat{u}=\tfrac12e+\tfrac12eu\in B_i$ is idempotent,
but neither $0$ nor $1$, because the support\footnote{The \emph{support}
of an element $\alpha=\sum\alpha_g g$ in a loop ring is the set of
loop elements $g$ that actually appear
in the representation of $\alpha$, that is, with nonzero coefficients.}  of $\tfrac12e$
is in $G$ while none of the support of $\tfrac12eu$ is in $G$.  Thus $B_i$ is not a field.
In the case of $\C{L}_2$ and $\C{L}_4$, we have $u^2=t_1=x^2$.  Set $v=x^{-1}u$.
Then $v^2=s$ so $v^4=1$.  Set $\hat{v}=\tfrac14(1+v+v^2+v^3)=\tfrac14(1+v+s+sv)=\tfrac14(1+v)(1+s)$.
Thus $\hat{v}e=\tfrac12e(1+v)$ in the field $A_i$ and hence, as in the first case, this
is an idempotent different from $0,1$ in $B_i$.
\end{proof}

\begin{rem} If $L$ belongs to $\C{L}_6$ or $\C{L}_7$, it may or
may not be the case that a field in the decomposition of $KG$ splits
into two fields in the decomposition of $KL$.  General conditions are
messy to write down though in any specific case, the above methods
still permit us to describe the structure of $KL$.  Many examples
are given in \secref{sec:l7}.
\end{rem}

\subsection*{Examples}
If $L=M(G,*,1)$ is in class $\C{L}_1$ with $m\ge2$, Theorems~\ref{thm:d1}
and \ref{thm:final} tell us that $\B{Q}L$ is the direct sum
of $8m$ fields and one Cayley algebra because $\B{Q}G$ is the
direct sum of $4m$ fields and one quaternion algebra.
The same theorems show that for a finite field $K$ (of characteristic different
from $2$), there are at least $2(8m-12)+2=16m-22$ simple components in the Wedderburn
decomposition of $KL$.  Moreover, if $K=F$ is a finite field of order $q\equiv3\pmod8$, then
$FL$ has precisely this number of simple components, of which $2(8m-12)=16m-24$
are fields and two are (necessarily split) Cayley algebras.

Similarly, from \thmref{thm:d2} and \thmref{thm:final}, we know that
if $L=M(G,*,g_0)$ is an RA loop of type $\C{L}_2$ and $m\ge3$, then $\B{Q}L$ is the
direct sum of $4m+4$ fields and one Cayley algebra while, if
$K$ is a finite field of characteristic different from $2$, $KL$ is the direct
sum of at least $2(4m-2)+2=8m-2$ simple algebras, this number being achieved
for finite fields of order $q\equiv3\pmod8$.  Specifically, with $F$ such a field,
the loop algebra $FL$ is the direct sum of precisely
$8m-4$ fields and two (split) Cayley algebras.

\subsection*{The Isomorphism Problem for Loops: Some Negative Answers}
It is well known that given two groups $G$ and $H$ and a field $K$, an
isomorphism $KG\iso KH$ need not imply $G\iso H$.  For example, if $G$
and $H$ are abelian groups of the same finite order $n$, then
the complex group algebras $\B{C}G$ and $\B{C}H$
are each the direct sum of $n$ copies of \B{C}.  Even more,
the two nonabelian groups of order $p^3$, $p\ne2$, have isomorphic group algebras
over any field of characteristic different from $p$.
In a similar vein, we note that
over any field $K$ of characteristic different from $2$,
each of the two RA loops of order $16$ has a loop algebra
that is the direct sum of four copies of $K$ and a Cayley algebra \cite[Corollary VII.2.3]{EGG:96}.
Over a finite field, any Cayley algebra is split (and hence unique), so the Cayley loop
and the loop denoted $M_{16}(Q_8,2)$ provide an instance of two nonisomorphic
loops with loop algebras isomorphic over any finite field of characteristic different
from~$2$ (and hence over any field of positive characteristic not $2$).

It is harder to find examples of nonisomorphic loops $L$ and $M$
with $\B{Q}L\iso \B{Q}M$.  One example, with one loop indecomposable
and the other not, appears in \cite{EGG:96} (Example XI.2.6).
We use the results of this paper to produce an example in which each of the loops
in question is \textbf{indecomposable}.

Let $L=M(G,*,g_0)$ be the loop from class $\C{L}_6$ defined by
the parameters $m_1=1$, $m_2=2$ and $m_3=1$.
Let $e\in \B{Q}G$ be an idempotent such that $A_i=(\B{Q}G)e$ is a field.
Since $m_1=1$, $(t_1e)^2=e$, and so (the projection in the field $A_i$) of) $t_1e$ is $\pm 1$.
Clearly $t_1e$ is a square if it is $+1$, but it is also
a square if is $-1$
because $t_2e$ is a $4$th root of unity in $A_i$ (so $t_1e=(t_2e)^2$).
Since in class
$\C{L}_6$, $u^2=t_1$, we see that $u^2$ is a square in each of those simple components
of $\B{Q}G$ which is a field.
It follows that the composition
algebra $A_i+A_iu$ splits into the sum of two copies of the field $A_i$,
and this occurs for each of the $14$ fields appearing in the
Wedderburn decomposition of $\B{Q}G$---see Case ii.a of \secref{sec:d5}.
So $\B{Q}L$ is the direct sum of
$28$ fields and six Cayley algebras.  By \thmref{thm:final}, such is also
the case for loop $M$ in $\C{L}_5$ defined by $m_1=1$, $m_2=2$, $m_3=1$ because the groups
defining $L$ and $M$ are the same in each case (the centres are $C_2\times C_4\times C_2$,
$x^2=t_2$, $y^2=t_3$).   Thus we have
two nonisomorphic loops $L$ and $M$ (of order $128$) with $\B{Q}L\iso \B{Q}M$.

Readers may wonder if $128$ is the smallest order for which there exists
a counterexample to the isomorphism problem with both loops indecomposable.
The answer is ``yes,'' as we proceed to explain.

With reference to \cite[Table V.4]{EGG:96}, we begin by noting
that there are two indecomposable RA loops
of order $16$.  In one case, $G=D_4$ and $\B{Q}G\iso 4\B{Q}\oplus M_2(\B{Q})$,
so $\B{Q}L$ is the direct sum of eight copies of \B{Q} and the split Cayley algebra.
In the other case, $G=Q_8$, $\B{Q}G\iso 4\B{Q}\oplus \B{H}$, \B{H} denoting
Hamilton's quaternion division algebra over \B{Q},
so $\B{Q}L$ is the direct sum of eight copies of \B{Q} and the Cayley division algebra.

There are four nonisomorphic indecomposable RA loops of order $32$:
\begin{itemize}
\item $L_1=M(16\Gamma_2b,*,1)$ of type $\C{L}_1$ with $m_1=2$ and $G=16\Gamma_2b$
in the notation of \cite{Hall:64};
\item $L_2=M(16\Gamma_2d,*,t_1)$ of type $\C{L}_2$ with $m_1=2$ and $G=16\Gamma_2d$;
\item $L_3=M(16\Gamma_2c_1,*,1)$ of type $\C{L}_3$ with $m_1=m_2=1$; and
\item $L_4=M(16\Gamma_2c_2,*,t_1)$ of type $\C{L}_4$ with $m_1=m_2=1$.
\end{itemize}
Results of this paper show that the commutative part of $\B{Q}L_1$ is the direct sum of
sixteen fields,
whereas the commutative parts of the other three loop algebras are the direct sum of twelve
fields.  The loop algebra $\B{Q}L_2$ has just one Cayley algebra in its Wedderburn
decomposition and so is isomorphic to neither $\B{Q}L_3$ nor $\B{Q}L_4$, each of which
are the direct sum of twelve fields and \textbf{two} Cayley algebras.  The latter two
loop algebras are distinguished by the fact that both Cayley algebras in the decomposition
of $\B{Q}L_3$ are split while the (rational) loop algebra of $L_4$ contains one Cayley division
algebra (because both quaternion algebras in the first group algebra are split while the
second group algebra contains a copy of Hamilton's division algebra------see
\cite{Vergara:06, Vergara:97} where the respective groups are labelled 32/9 and 32/10).

\begin{table}
\footnotesize
\caption{\label{tab2}}
\begin{center}
\begin{tabular}{lll}  \hline\hline \\[-4pt]
\text{Loop $L$} & \multicolumn{1}{c}{\text{Class}} & \multicolumn{1}{c}{\text{$\B{Q}L$}} \\ \\[-4pt]\hline \\[-4pt]
$L_5=M(32\Gamma_2g,*,1)$ & $\C{L}_1$, $m_1=3$ & $24$ fields, one Cayley algebra \\
$L_6=M(32\Gamma_2k,*,t_1)$ & $\C{L}_2$, $m_1=3$ & $16$ fields, one Cayley algebra \\
$L_7=M(32\Gamma_2f,*,1)$ & $\C{L}_3$, $m_1=2$, $m_2=1$ & $24$ fields, two Cayley algebras \\
$L_8=M(32\Gamma_2i,*,t_1)$ & $\C{L}_4$, $m_1=2$, $m_2=1$ & $12$ fields, six Cayley algebras \\
$L_9=M(32\Gamma_2j_1,*,1)$ & $\C{L}_3$, $m_1=1$, $m_2=2$ & $16$ fields, three Cayley algebras \\
$L_{10}=M(32\Gamma_2j_2,*,t_1)$ & $\C{L}_4$,  $m_1=1$, $m_2=2$ & $16$ fields, three Cayley algebras \\
$L_{11}=M(32\Gamma_2h,*,1)$ & $\C{L}_5$, $m_1=m_2=m_3=1$ & $20$ fields, four Cayley algebras \\
$L_{12}=M(32\Gamma_2h,*,t_1)$ & $\C{L}_6$,  $m_1=m_2=m_3=1$ & $20$ fields, four Cayley algebras \\ \hline\hline
\end{tabular}
\end{center}
\end{table}

There are eight nonisomorphic indecomposable RA loops of order $64$, pertinent data of
each being summarized in \tabref{tab2}.   Clearly, it suffices to distinguish the loop
algebras $\B{Q}L_9$ and $\B{Q}L_{10}$, and the loop algebras $\B{Q}L_{11}$ and $\B{Q}L_{12}$.
For this, we simply note
that in the Wedderburn decompositions of $\B{Q}L_9$ and $\B{Q}L_{11}$,
the Cayley algebras are all split, but
this is not the case in $\B{Q}L_{10}$ and $\B{Q}L_{12}$ (the groups $32\Gamma_2j_1$
and $32\Gamma_2j_2$ are labelled 32/20 and 32/21 in \cite{Vergara:06, Vergara:97}).

\section{Loops from class $\C{L}_7$} \label{sec:l7}
Turning to indecomposable loops from the class $\C{L}_7$, which are formed
by doubling groups of the form $G_0\times \langle t_4\rangle$ with $G_0$
a group in class $\C{D}_5$,  experience
shows that we will need to know
the number of cyclic subgroups in the direct product of four cyclic groups,
each of $2$-power
order.  This information is given by our final lemma.

\begin{lem}\label{lem:4cyclics} Let
$A=C_{2^a}\times C_{2^b}\times C_{2^c}\times C_{2^d}$
be the direct product of cyclic groups of orders $2^a$, $2^b$, $2^c$
and $2^d$ with $a\ge b\ge c\ge d$.   The number of cyclic subgroups of $A$ of order $2^k$ is
\begin{itemize}
\item $15(2^{3(k-1)})$ if $1\le k\le d$,
\item $7(2^d)2^{2(k-1)}$  if $d< k\le c$,
\item $3(2^{c+d+k-1})$  if $c< k\le b$, and
\item $2^{b+c+d}$ if $b<k\le a$.
\end{itemize}
Including the trivial subgroup, in all, $A$ has
\begin{equation*}
\tfrac{15}{7}(8^d-1)+\tfrac73(2^d)(4^c-4^d)
+3(2^{c+d})(2^b-2^c)+(a-b)2^{b+c+d}+1
\end{equation*}
cyclic subgroups.
\end{lem}
\begin{proof}
Suppose $1\le k\le d$.  Then $(t,u,v,w)\in A$ has order $2^k$
if and only if $(t,u,v,d)$ has order $2^k$
in a copy of $C_{2^k}\times C_{2^k}\times C_{2^k}\times C_{2^k}$ within $A$.  The
number of such elements is the order of $C_{2^k}\times C_{2^k}\times C_{2^k}\times C_{2^k}$
less the order of $C_{2^{k-1}}\times C_{2^{k-1}}\times C_{2^{k-1}}\times C_{2^{k-1}}$,
which is $2^{4k}-2^{4(k-1)}=2^{4(k-1)}(2^4-1)=15(2^{4(k-1)})$.
Since there are $\phi(2^k)=2^{k-1}$ elements of order $2^k$ in any cyclic subgroup
of order $2^k$, the number of cyclic subgroups of order $2^k$ in this first case
is $15(2^{3(k-1)})$.  In all, this gives
\begin{multline*}
15+15(2^3)+15(2^6)+15(2^9)+\cdots +15(2^{3(d-1)})\\
 =15(1+2^3+2^6+\cdots+2^{3(d-1)})=\tfrac{15}{7}(2^{3d}-1)
\end{multline*}
cyclic subgroups of order $2^k$ with $1\le k\le d$.
\smallskip\par Suppose $d<k\le c$.  An element $(t,u,v,w)$ has order $2^k$
if and only if it lives in a copy of $C_{2^k}\times C_{2^k}\times C_{2^k}\times C_{2^d}$.
Using \lemref{lem:3cyclics}, the number of elements of order $2^k$ in $A$
is $7(2^{3(k-1)})2^d$ and the number of cyclic subgroups of this order is
$7(2^{2(k-1)})2^d$.  In all, we have
\begin{multline*}
7(2^d)(2^{2d}+2^{2(d+1)}+2^{2(d+2)}+\cdots+2^{2(d-1)}) \\
\begin{array}{l}
=7(2^d)(2^d)(1+2^2+2^4+\cdots+2^{2(c-d-1)}\\[4pt]
=7(2^{3d})\frac{2^{2(c-d)}-1}{3} =\frac73(2^{3d})(2^{2(c-d)}-1)
=\frac73(2^d)(2^{2c}-2^{2d})
\end{array}
\end{multline*}
cyclic subgroups with $d<k\le c$.
\smallskip\par Suppose $c<k\le b$.  An element $(t,u,v,w)$ has order $2^k$
if and only if it lives in a copy of $C_{2^k}\times C_{2^k}\times C_{2^c}\times C_{2^d}$.
Using \lemref{lem:2cyclics}, the number of elements of order $2^k$ with $k$
in the indicated range is $3(4^{k-1})(2^c)(2^d)$ and the corresponding
number of subgroups is $3(2^{k-1})(2^{c+d})$.  In all, we have
\begin{multline*}
3(2^{c+d})(2^c+2^{c+1}+2^{c+2}+\cdots + 2^{b-1}) \\
\begin{array}{l}
=3(2^{c+d})(2^c)(1+2+2^2+\cdots + 2^{b-c-1}) \\[4pt]
   =3(2^{c+d})(2^c)(2^{b-c}-1)=3(2^{c+d})(2^b-2^c)
\end{array}
\end{multline*}
cyclic subgroups of order $2^k$, $c<k\le b$.
\smallskip\par Suppose finally that $b<k\le a$.  Then $(t,u,v,w)$ has order
$2^k$ if and only if it lives in a subgroup of $A$ isomorphic to $C_{2^k}\times C_{2^b}
\times C_{2^c}\times C_{2^d}$.
The number of such elements is $2^{k-1}2^{b+c+d}$, so the
number of cyclic subgroups is $2^{b+c+d}$ and the total number for all $k$
with $b<k\le a$ is $(a-b)2^{b+c+d}$.
\end{proof}

In \tabref{tab1}, we record some specific consequences of this lemma
which hasten many of the calculations that follow.
\lemref{lem:4cyclics} provides the number of \B{Q}-classes
in the direct product of four cyclic groups
while \lemref{lem1}, which tells us which of these split into (two) $F$-classes and
which do not, allows us to determine the number of $F$-classes.  Thus,
if there are $\alpha$ \B{Q}-classes, $\beta$ of which correspond to cyclic subgroups
generated by an element of order at most $4$, then there are $\alpha+2(\alpha-\beta)
=2\alpha-\beta$ $F$-classes.

\begin{table}[p]
\footnotesize
\begin{tabular}{ccc} \hline\hline
Group & Number of \B{Q}-classes & Number of $F$-classes \\ \hline \\[-6pt]
$C_2\times C_2$ & $4$ & $4$ \\
$C_4\times C_2$ & $6$ & $6$ \\
$C_4\times C_4$ & $10$ & $10$ \\
$C_8\times C_2$ & $8$ & $10$ \\
$C_8\times C_4$ & $14$ & $18$ \\
$C_8\times C_8$ & $22$ & $34$ \\
$C_{2^a}\times C_2$ & $2a+2$ & $4a-2$ \\
$C_{2^a}\times C_4$ & $4a+2$ & $8a-6$ \\
$C_{2^a}\times C_8$ & $8a-2$ & $16a-14$ \\[8pt]
$C_2\times C_2\times C_2$ & $8$ & $8$ \\
$C_4\times C_2\times C_2$ & $12$ & $12$ \\
$C_4\times C_4\times C_2$ & $20$ & $20$ \\
$C_4\times C_4\times C_4$ & $36$ & $36$ \\
$C_8\times C_2\times C_2$ & $16$ & $20$ \\
$C_8\times C_4\times C_2$ & $28$ & $36$ \\
$C_8\times C_4\times C_4$ & $52$ & $68$ \\
$C_8\times C_8\times C_2$ & $44$ & $68$ \\
$C_8\times C_8\times C_4$ & $84$ & $132$ \\
$C_{2^a}\times C_2\times C_2$ & $4a+4$ & $8a-4$ \\
$C_{2^a}\times C_4\times C_2$ & $8a+4$ & $16a-12$ \\
$C_{2^a}\times C_4\times C_4$ & $16a+4$ & $32a-28$ \\
$C_{2^a}\times C_8\times C_2$ & $16a-4$ & $32a-28$ \\
$C_{2^a}\times C_8\times C_4$ & $32a-12$ & $64a-60$ \\
$C_{2^a}\times C_8\times C_8$ & $64a-44$ & $128a-124$ \\
$C_{2^a}\times C_{2^b}\times C_2$ & $2^{b+1}(3+a-b)-4$ & $2^{b+2}(3+a-b)-28$ \\
$C_{2^a}\times C_{2^b}\times C_4$ & $2^{b+2}(3+a-b)-12$ & $2^{b+3}(3+a-b)-60$ \\
$C_{2^a}\times C_{2^b}\times C_8$ & $2^{b+3}(3+a-b)-44$ & $2^{b+4}(3+a-b)-124$ \\[8pt]
$C_2\times C_2\times C_2\times C_2$ & $16$ & $16$ \\
$C_4\times C_2\times C_2\times C_2$ & $24$ & $24$ \\
$C_4\times C_4\times C_2\times C_2$ & $40$ & $40$ \\
$C_4\times C_4\times C_4\times C_2$ & $72$ & $72$ \\
$C_4\times C_4\times C_4\times C_4$ & $136$ & $136$ \\
$C_8\times C_2\times C_2\times C_2$ & $32$ & $40$ \\
$C_8\times C_4\times C_2\times C_2$ & $56$ & $72$ \\
$C_8\times C_4\times C_4\times C_2$ & $104$ & $136$ \\
$C_8\times C_8\times C_2\times C_2$ & $88$ & $136$ \\
$C_8\times C_8\times C_4\times C_2$ & $168$ & $264$\\
$C_{2^a}\times C_2\times C_2\times C_2$ & $8a+8$ & $16a-8$  \\
$C_{2^a}\times C_4\times C_2\times C_2$ & $16a+8$ & $32a-24$ \\
$C_{2^a}\times C_4\times C_4\times C_2$ & $32a+8$ & $64a-56$ \\
$C_{2^a}\times C_4\times C_4\times C_4$ & $64a+8$ & $128a-120$ \\
$C_{2^a}\times C_8\times C_2\times C_2$ & $32a-8$ & $64a-56$ \\
$C_{2^a}\times C_8\times C_4\times C_2$ & $64a-24$ & $128a-120$ \\
$C_{2^a}\times C_8\times C_8\times C_2$ & $128a-88$ & $256a-248$ \\
$C_{2^a}\times C_{2^b}\times C_2\times C_2$ & $2^{b+2}(3+a-b)-8$ & $2^{b+3}(3+a-b)-56$ \\
$C_{2^a}\times C_{2^b}\times C_4\times C_2$ & $2^{b+3}(3+a-b)-24$ & $2^{b+4}(3+a-b)-120$ \\
$C_{2^a}\times C_{2^b}\times C_4\times C_4$ & $2^{b+4}(3+a-b)-56$ & $2^{b+5}(3+a-b)-248$ \\
$C_{2^a}\times C_{2^b}\times C_8\times C_2$ & $2^{b+4}(3+a-b)-88$ & $2^{b+5}(3+a-b)-248$ \\ \hline\hline \\[-6pt]
\end{tabular}
\caption{$a\ge3$, $b\ge3$, $a\ge b$.  \label{tab1}}
\end{table}

To investigate the structure of $KL$ with $L=M(G_0\times\langle t_4\rangle,*,t_4)$
a loop in class $\C{L}_7$, we begin as usual
by studying the Wedderburn decomposition of $KG$,
$G=G_0\times\langle t_4\rangle$ and $G_0$ a group from class
$\C{D}_5$.  Thus $x^2=t_2$ and $y^2=t_3$.
It follows that $(xy)^2=st_2t_3$ so that modulo
$\langle s\rangle$, the group generated by $\C{Z}(G)$ and $xy$
is the direct product $\langle xy\rangle\times\langle t\rangle\times\langle t_4\rangle$
with $t$ either $t_2$ or $t_3$, whichever has smaller order.

As previously, we begin with small cases and, again,
use $F$ to denote a finite field of
order $q\equiv3\pmod8$ and $K$ a field which is ambiguously \B{Q} or~$F$.
Since we expect our methods will by now be clear to faithful readers,  we simply
record the salient facts in tables.

\smallskip\par\noindent (i) $m_1=m_2=m_3=m_4=1$, $s=t_1$.

\begin{center}
\small
\begin{tabular}{ll}
\multicolumn{1}{c}{Group} & \multicolumn{1}{c}{No. of $K$-classes} \\ \hline \\[-6pt]
$\C{Z}(G)\iso C_2\times C_2\times C_2\times C_2$ & $N_1=16$  \\
$\C{Z}(G)/\langle s\rangle\iso C_2\times C_2\times C_2$ & $8$  \\
$\langle\C{Z}(G),x\rangle/\langle s\rangle\iso C_4\times C_2\times C_2$ & $12$ ($N_2=4$) \\
$\langle\C{Z}(G),y\rangle/\langle s\rangle\iso C_2\times C_4\times C_2$ & $12$ ($N_3=4$)  \\
$\langle\C{Z}(G),xy\rangle/\langle s\rangle\iso C_4\times C_2\times C_2$ & $12$ ($N_4=4$) \\
$G/G'\iso C_4\times C_4\times C_2$ & $N_0=20$ \\ \hline
\end{tabular}
\end{center}
\medskip

Since $N=N_1+3N_2=28$, $KG$ the direct sum of $N_0=20$ fields and $N-N_0=8$
quaternion algebras.

\newpage

\par\noindent (ii.a) $m_1=m_2=m_3=1$, $m_4=2$, $s=t_1$.
\begin{center}
\small
\begin{tabular}{ll}
\multicolumn{1}{c}{Group} & \multicolumn{1}{c}{No. of $K$-classes} \\ \hline \\[-6pt]
$\C{Z}(G)\iso C_2\times C_2\times C_2\times C_4$ & $N_1=24$  \\
$\C{Z}(G)/\langle s\rangle\iso C_2\times C_2\times C_4$ & $12$  \\
$\langle\C{Z}(G),x\rangle/\langle s\rangle\iso C_4\times C_2\times C_4$ & $20$ ($N_2=8$)  \\
$\langle\C{Z}(G),y\rangle/\langle s\rangle\iso C_2\times C_4\times C_4$ & $20$ ($N_3=8$)  \\
$\langle\C{Z}(G),xy\rangle/\langle s\rangle\iso C_2\times C_4\times C_4$ & $20$ ($N_4=8$)  \\
$G/G'\iso C_4\times C_4\times C_4$ & $N_0=36$ \\ \hline
\end{tabular}
\end{center}
\medskip

Since $N=N_1+3N_2=48$, $KG$ the direct sum of $N_0=36$ fields and $N-N_0=12$
quaternion algebras.

\smallskip\par\noindent (ii.b) Suppose $m_1=m_2=m_3=1$ and $m_4\ge3$, $s=t_1$.
\begin{center}
\small
\begin{tabular}{lll}
\multicolumn{1}{c}{Group} & \multicolumn{1}{c}{No. of \B{Q}-classes} & \multicolumn{1}{c}{No. of $F$-classes} \\ \hline \\[-6pt]
$\C{Z}(G)$ \\
 \quad $\iso C_2\times C_2\times C_2\times C_{2^{m_4}}$ & $N_1=8m_4+8$ & $M_1=16m_4-8$ \\
$\C{Z}(G)/\langle s\rangle$ \\
 \quad $\iso C_2\times C_2\times C_{2^{m_4}}$ & $4m_4+4$ & $8m_4-4$ \\
$\langle\C{Z}(G),x\rangle/\langle s\rangle$ \\
 \quad $\iso C_4\times C_2\times C_{2^{m_4}}$ & $8m_4+4$ ($N_2=4m_4$)& $16m_4-12$ ($M_2=8m_4-8$)\\
$\langle\C{Z}(G),y\rangle/\langle s\rangle$ \\
 \quad $\iso C_2\times C_4\times C_{2^{m_4}}$ & $8m_4+4$ ($N_3=4m_4$) & $16m_4$ ($M_3=8m_4-8$) \\
$\langle\C{Z}(G),xy\rangle/\langle s\rangle$ \\
 \quad $\iso C_2\times C_4\times C_{2^{m_4}}$ & $8m_4+4$ ($N_4=4m_4$) & $16m_4$ ($M_4=8m_4-8$)\\
$G/G'$ \\
 \quad $\iso C_4\times C_4\times C_{2^{m_4}}$ & $N_0=16m_4+4$ & $M_0=32m_4-28$ \\ \hline
\end{tabular}
\end{center}
\medskip

Since $N=N_1+3N_2=20m_4+8$ and $M=M_1+3M_2=40m_4-32$, $\B{Q}G$ is
the direct sum of $N_0=16m_4+4$ fields and $N-N_0=4m_4+4$ quaternion algebras while
$FG$ is the direct sum of $M_0=32m_4-28$ fields and $M-M_0=8m_4-4$ quaternion algebras.

\bigskip\par\noindent (iii.a) $m_1=m_2=1$, $m_3=2$, $m_4=1$, $s=t_1$.  (The conclusions in
the case $1,2,1,1$ are the same.)
\begin{center}
\small
\begin{tabular}{lll}
\multicolumn{1}{c}{Group} & \multicolumn{1}{c}{No. of \B{Q}-classes} & \multicolumn{1}{c}{No. of $F$-classes} \\ \hline \\[-6pt]
$\C{Z}(G)\iso C_2\times C_2\times C_4\times C_2$ & $N_1=24$ & $M_1=24$ \\
$\C{Z}(G)/\langle s\rangle\iso C_2\times C_4\times C_2$ & $12$ & $12$ \\
$\langle\C{Z}(G),x\rangle/\langle s\rangle\iso C_4\times C_4\times C_2$ & $20$ ($N_2=8$)& $20$ ($M_2=8$)\\
$\langle\C{Z}(G),y\rangle/\langle s\rangle\iso C_2\times C_8\times C_2$ & $16$ ($N_3=4$) & $20$ ($M_3=8$) \\
$\langle\C{Z}(G),xy\rangle/\langle s\rangle\iso C_2\times C_8\times C_2$ & $16$ ($N_4=4$) & $20$ ($M_4=8$)\\
$G/G'\iso C_4\times C_8\times C_2$ & $N_0=28$ & $M_0=36$ \\ \hline
\end{tabular}
\end{center}
\medskip

Since $N=N_1+N_2+2N_3=40$ and $M=M_1+3M_2=48$,
$\B{Q}G$ the direct sum of $N_0=28$ fields and $N-N_0=12$
quaternion algebras while $FG$ is the direct sum of $M_0=36$ fields
and $M-M_0=12$ quaternion algebras.

\newpage
\smallskip\par\noindent (iii.b) $m_1=m_2=1$, $m_3\ge3$, $m_4=1$, $s=t_1$.  (The conclusions
in the case $1,m_2\ge3,1,1$ are the same.)
\bigskip
\begin{center}
\small
\begin{tabular}{ll}
\multicolumn{1}{c}{Group} & \multicolumn{1}{c}{No. of \B{Q}- and $F$-classes} \\ \hline \\[-6pt]
$\C{Z}(G)\iso C_2\times C_2\times C_{2^{m_3}}\times C_2$ & $N_1=8m_3+8$ \\
                                                         & $M_1=16m_3-8$ \\[4pt]
$\C{Z}(G)/\langle s\rangle\iso C_2\times C_{2^{m_3}}\times C_2$ & \B{Q}: $4m_3+4$ \\
                                                              & $F$: $8m_3-4$ \\[4pt]
$\langle\C{Z}(G),x\rangle/\langle s\rangle\iso C_4\times C_{2^{m_3}}\times C_2$ & \B{Q}: $8m_3+4$ ($N_2=4m_3$) \\
                   & $F$: $16m_3-12$ (so $M_2=8m_3-8$)\\[4pt]
$\langle\C{Z}(G),y\rangle/\langle s\rangle\iso C_2\times C_{2^{m_3+1}}\times C_2$ & \B{Q}: $4m_3+8$ ($N_3=4$) \\
                       & $F$: $8m_3+4$ (so $M_3=8$) \\[4pt]
$\langle\C{Z}(G),xy\rangle/\langle s\rangle\iso C_2\times C_{2^{m_3+1}}\times C_2$ & \B{Q}: $4m_3+8$ ($N_4=4$) \\
               & $F$: $8m_3+4$ ($M_4=8$)\\[4pt]
$G/G'\iso C_4\times C_{2^{m_3+1}}\times C_2$ & $N_0=8m_3+12$ \\
                                             & $M_0=16m_3+4$ \\ \hline
\end{tabular}
\end{center}
\bigskip

Since $N=N_1+N_2+2N_3=12m_3+16$ and $M=M_1+M_2+2M_3=24m_3$,
$\B{Q}G$ is the direct sum of $N_0=8m_3+12$ fields and $N-N_0=4m_3+4$
quaternion algebras while $FG$ is the direct sum of $M_0=16m_3+4$ fields
and $M-M_0=8m_3-4$ quaternion algebras.

\bigskip\par\noindent (iv.a) $m_1=m_2=1$, $m_3=m_4=2$, $s=t_1$. (The conclusions in
the case $1,2,1,2$ are the same.)
\bigskip
\begin{center}
\small
\begin{tabular}{ccc}
Group & No. of \B{Q}-classes & No. of $F$-classes \\ \hline \\[-6pt]
$\C{Z}(G)\iso C_2\times C_2\times C_4\times C_4$ & $N_1=40$ & $M_1=40$ \\
$\C{Z}(G)/\langle s\rangle\iso C_2\times C_4\times C_4$ & $20$ & $20$ \\
$\langle\C{Z}(G),x\rangle/\langle s\rangle\iso C_4\times C_4\times C_4$ & $36$ ($N_2=16$)& $36$ ($M_2=16$)\\
$\langle\C{Z}(G),y\rangle/\langle s\rangle\iso C_2\times C_8\times C_4$ & $28$ ($N_3=8$) & $36$ ($M_3=16$) \\
$\langle\C{Z}(G),xy\rangle/\langle s\rangle\iso C_2\times C_8\times C_4$ & $28$ ($N_4=8$) & $36$ ($M_4=16$)\\
$G/G'\iso C_4\times C_8\times C_4$ & $N_0=52$ & $M_0=68$ \\ \hline
\end{tabular}
\end{center}
\bigskip

Since $N=N_1+N_2+2N_3=72$ and $M=M_1+3M_2=88$, $\B{Q}G$ is
the direct sum of $N_0=52$ fields and $N-N_0=20$ quaternion
algebras while $FG$ is the direct sum of $M_0=68$ fields and
$M-M_0=20$ quaternion algebras.

\newpage
\smallskip\par\noindent (iv.b) $m_1=m_2=1$, $m_3\ge 3$, $m_4=2$, $s=t_1$. (The conclusions
in the case $1,m_2\ge3,1,2$ are the same.)
\bigskip
\begin{center}
\small
\begin{tabular}{ll}
\multicolumn{1}{c}{Group} & \multicolumn{1}{c}{No. of \B{Q}- and $F$-classes} \\ \hline \\[-6pt]
$\C{Z}(G)\iso C_2\times C_2\times C_{2^{m_3}}\times C_4$ & \B{Q}: $N_1=16m_3+8$ \\
                          & $F$: $M_1=32m_3-24$ \\[4pt]
$\C{Z}(G)/\langle s\rangle\iso C_2\times C_{2^{m_3}}\times C_4$ & \B{Q}: $8m_3+4$ \\
                  & $F$: $16m_3-12$ \\[4pt]
$\langle\C{Z}(G),x\rangle/\langle s\rangle\iso C_4\times C_{2^{m_3}}\times C_4$ & $\B{Q}\colon 16m_3+4$ ($N_2=8m_3$) \\
              & $F\colon 32m_3-28$ ($M_2=16m_3-16$)\\[4pt]
$\langle\C{Z}(G),y\rangle/\langle s\rangle\iso C_2\times C_{2^{m_3+1}}\times C_4$ & $\B{Q}\colon 8m_3+12$ ($N_3=8$) \\
             & $F\colon 16m_3+4$ ($M_3=16$) \\[4pt]
$\langle\C{Z}(G),xy\rangle/\langle s\rangle\iso C_2\times C_{2^{m_3+1}}\times C_4$ & $\B{Q}\colon 8m_3+12$ ($N_4=8$) \\
              & $F\colon 16m_3+4$ ($M_4=16$)\\[4pt]
$G/G'\iso C_4\times C_{2^{m_3+1}}\times C_4$ & $N_0=16m_3+20$ \\
                                             & $M_0=32m_3+4$ \\ \hline
\end{tabular}
\end{center}
\bigskip

Since $N=N_1+N_2+2N_3=24m_3+24$ and $M=M_1+M_2+2M_3=48m_3-8$,
$\B{Q}G$ is the direct sum of $N_0=16m_3+20$ fields and $N-N_0=8m_3+4$
quaternion algebras while $FG$ is the direct sum of $M_0=32m_3+4$ fields
and $M-M_0=16m_3-12$ quaternion algebras.

\bigskip\par\noindent (iv.c) $m_1=m_2=1$, $m_3=2$, $m_4\ge3$, $s=t_1$. (The conclusions
in the case $1,2,1,m_4\ge3$ are the same.)
\bigskip
\begin{center}
\small
\begin{tabular}{ll}
\multicolumn{1}{c}{Group} & \multicolumn{1}{c}{No. of \B{Q}- and $F$-classes} \\ \hline \\[-6pt]
$\C{Z}(G)\iso C_2\times C_2\times C_4\times C_{2^{m_4}}$ & $N_1=16m_4+8$ \\
                                & $M_1=32m_4-24$ \\[4pt]
$\C{Z}(G)/\langle s\rangle\iso C_2\times C_4\times C_{2^{m_4}}$ & $\B{Q}\colon 8m_4+4$ \\
                                 & $F\colon 16m_4-12$ \\[4pt]
$\langle\C{Z}(G),x\rangle/\langle s\rangle\iso C_4\times C_4\times C_{2^{m_4}}$ & $\B{Q}\colon 16m_4+4$ ($N_2=8m_4$) \\
              & $F\colon 32m_4-28$ ($M_2=16m_4-16$)\\[4pt]
$\langle\C{Z}(G),y\rangle/\langle s\rangle\iso C_2\times C_8\times C_{2^{m_4}}$ & $\B{Q}\colon 16m_4-4$ ($N_3=8m_4-8$) \\
                 & $F\colon 32m_4-28$ ($M_3=16m_4-16$) \\[4pt]
$\langle\C{Z}(G),xy\rangle/\langle s\rangle\iso C_2\times C_8\times C_{2^{m_4}}$ & $\B{Q}\colon 16m_4-4$ ($N_4=8m_4-8$) \\
                      & $F\colon 32m_4-28$ ($M_4=16m_4-16$)\\[4pt]
$G/G'\iso C_4\times C_8\times C_{2^{m_4}}$ & $N_0=32m_4-12$ \\
                                           & $M_0=64m_4-60$ \\ \hline
\end{tabular}
\end{center}
\bigskip

Since $N=N_1+N_2+2N_3=40m_4-8$ and $M=M_1+M_2+2M_3=80m_4-72$,
$\B{Q}G$ is the direct sum of $N_0=32m_4-12$ fields and $N-N_0=8m_4+4$
quaternion algebras while $FG$ is the direct sum of $M_0=64m_4-60$ fields
and $M-M_0=16m_4-12$ quaternion algebras.

\newpage
\smallskip\par\noindent (iv.d) $m_1=m_2=1$, $m_3\ge3$, $m_4\ge3$, $s=t_1$. (The conclusions
in the case $1,m_2\ge3,1,m_4\ge3$ are the same.)
\medskip

Assume first that $m_3=m_4$.
\begin{center}
\small
\begin{tabular}{lll}
\multicolumn{1}{c}{Group} & \multicolumn{1}{c}{No. of \B{Q}-classes} & \multicolumn{1}{c}{No. of $F$-classes} \\ \hline \\[-6pt]
$\C{Z}(G)\iso C_2\times C_2\times C_{2^{m_3}}\times C_{2^{m_3}}$ & $N_1=2^{m_3+2}(3)-8$
                            & $M_1=2^{m_3+3}(3)-56$ \\[4pt]
$\C{Z}(G)/\langle s\rangle\iso C_2\times C_{2^{m_3}}\times C_{2^{m_3}}$ & $2^{m_3+1}(3)-4$
                  & $2^{m_3+2}(3)-28$ \\[4pt]
$\langle\C{Z}(G),x\rangle/\langle s\rangle$ & $2^{m_3+2}(3)-12$ & $2^{m_3+3}(3)-60$ \\
 \qquad $\iso C_4\times C_{2^{m_3}}\times C_{2^{m_3}}$ & ($N_2=2^{m_3+1}(3)-8$)  & ($M_2=2^{m_3+2}(3)-32$)\\[4pt]
$\langle\C{Z}(G),y\rangle/\langle s\rangle$ & $2^{m_3+1}(4)-4$ & $2^{m_4+2}(4)-28$ \\
 \qquad $\iso C_2\times C_{2^{m_3+1}}\times C_{2^{m_3}}$  &  ($N_3=2^{m_3+1}$)  & ($M_3=2^{m_3+2}$) \\[4pt]
$\langle\C{Z}(G),xy\rangle/\langle s\rangle$ & $2^{m_3+1}(4)-4$  & $2^{m_4+2}(4)-28$ \\
 \qquad  $\iso C_2\times C_{2^{m_3+1}}\times C_{2^{m_3}}$ &  ($N_4=2^{m_3+1}$) & ($M_4=2^{m_4+2}$)\\[4pt]
$G/G'\iso C_4\times C_{2^{m_3+1}}\times C_{2^{m_3}}$ & $N_0=2^{m_3+2}(4)-12$
                                                     & $M_0=2^{m_3+3}(4)-60$ \\ \hline
\end{tabular}
\end{center}
Since $N=N_1+N_2+2N_3=2^{m_3+1}(11)-16$ and $M=M_1+M_2+2M_3=2^{m_3+2}(11)-88$,
$\B{Q}G$ is the direct sum of $N_0=2^{m_3+4}-12$ fields and $N-N_0=2^{m_3+1}(3)-4$
quaternion algebras while $FG$ is the direct sum of $M_0=2^{m_3+5}-60$ fields
and $M-M_0=2^{m_3+2}(3)-28$ quaternion algebras.
\medskip

Now assume $m_3>m_4$.
\begin{center}
\small
\begin{tabular}{ll}
\multicolumn{1}{c}{Group} & \multicolumn{1}{c}{No. of \B{Q}- and $F$-classes} \\ \hline \\[-6pt]
$\C{Z}(G)\iso C_2\times C_2\times C_{2^{m_3}}\times C_{2^{m_4}}$ & $N_1=2^{m_4+2}(3+m_3-m_4)-8$ \\
                            & $M_1=2^{m_4+3}(3+m_3-m_4)-56$ \\[4pt]
$\C{Z}(G)/\langle s\rangle\iso C_2\times C_{2^{m_3}}\times C_{2^{m_4}}$ & $\B{Q}\colon 2^{m_4+1}(3+m_3-m_4)-4$ \\
                  & $F\colon 2^{m_4+2}(3+m_3-m_4)-28$ \\[4pt]
$\langle\C{Z}(G),x\rangle/\langle s\rangle\iso C_4\times C_{2^{m_3}}\times C_{2^{m_4}}$ & $\B{Q}\colon  2^{m_4+2}(3+m_3-m_4)-12$ \\
                               & \qquad ($N_2=2^{m_4+1}(3+m_3-m_4)-8$) \\
             & $F\colon 2^{m_4+3}(3+m_3-m_4)-60$ \\
                               & \qquad ($M_2=2^{m_4+2}(3+m_3-m_4)-32$)\\[4pt]
$\langle\C{Z}(G),y\rangle/\langle s\rangle\iso C_2\times C_{2^{m_3+1}}\times C_{2^{m_4}}$ & $\B{Q}\colon 2^{m_4+1}(4+m_3-m_4)-4$ \\
                             & \qquad ($N_3=2^{m_4+1}$) \\
                  & $F\colon 2^{m_4+2}(4+m_3-m_4)-28$ \\
                             & \qquad ($M_3=2^{m_4+2}$) \\[4pt]
$\langle\C{Z}(G),xy\rangle/\langle s\rangle\iso C_2\times C_{2^{m_3+1}}\times C_{2^{m_4}}$ & $\B{Q}\colon 2^{m_4+1}(4+m_3-m_4)-4$ \\
                             & \qquad ($N_4=2^{m_4+1}$) \\
                         & $F\colon 2^{m_4+2}(4+m_3-m_4)-28$ \\
                             & \qquad ($M_4=2^{m_4+2}$)\\[4pt]
$G/G'\iso C_4\times C_{2^{m_3+1}}\times C_{2^{m_4}}$ & $N_0=2^{m_4+2}(4+m_3-m_4)-12$ \\
                                                     & $M_0=2^{m_4+3}(4+m_3-m_4)-60$ \\ \hline
\end{tabular}
\end{center}
\medskip

Since $N=N_1+N_2+2N_3=2^{m_4+1}(11+3m_3-3m_4)-16$ and $M=M_1+M_2+2M_3=2^{m_4+2}(11+3m_3-3m_4)-88$,
$\B{Q}G$ is the direct sum of $N_0=2^{m_4+2}(4+m_3-m_4)-12$ fields and $N-N_0=2^{m_4+1}(3+m_3-m_4)-4$
quaternion algebras while $FG$ is the direct sum of $M_0=2^{m_4+3}(4+m_3-m_4)-60$ fields
and $M-M_0=2^{m_4+2}(3+m_3-m_4)-28$ quaternion algebras.

\bigskip
Finally, assume that $m_3<m_4$.
\bigskip
\begin{center}
\small
\begin{tabular}{ll}
\multicolumn{1}{c}{Group} & \multicolumn{1}{c}{No. of \B{Q}- and $F$-classes} \\ \hline \\[-6pt]
$\C{Z}(G)\iso C_2\times C_2\times C_{2^{m_3}}\times C_{2^{m_4}}$ & $N_1=2^{m_3+2}(3+m_4-m_3)-8$ \\
                          & $M_1=2^{m_3+3}(3+m_4-m_3)-56$ \\[4pt]
$\C{Z}(G)/\langle s\rangle\iso C_2\times C_{2^{m_3}}\times C_{2^{m_4}}$ & $\B{Q}\colon 2^{m_3+1}(3+m_4-m_3)-4$ \\
                          & $F\colon 2^{m_3+2}(3+m_4-m_3)-28$ \\[4pt]
$\langle\C{Z}(G),x\rangle/\langle s\rangle\iso C_4\times C_{2^{m_3}}\times C_{2^{m_4}}$ & $\B{Q}\colon 2^{m_3+2}(3+m_4-m_3)-12$ \\
                                  & \qquad ($N_2=2^{m_3+1}(3+m_4-m_3)-8$) \\
                          & $F\colon 2^{m_3+3}(3+m_4-m_3)-60$ \\
                                  & \qquad ($M_2=2^{m_3+2}(3+m_4-m_3)-32$)\\[4pt]
$\langle\C{Z}(G),y\rangle/\langle s\rangle\iso C_2\times C_{2^{m_3+1}}\times C_{2^{m_4}}$ & $\B{Q}\colon 2^{m_3+2}(2+m_4-m_3)-4$ \\
                                  & \qquad ($N_3=2^{m_3+1}(1+m_4-m_3)$) \\
                          & $F\colon 2^{m_3+3}(2+m_4-m_3)-28$ \\
                                  & \qquad ($M_3=2^{m_3+2}(1+m_4-m_3)$) \\[4pt]
$\langle\C{Z}(G),xy\rangle/\langle s\rangle\iso C_2\times C_{2^{m_3+1}}\times C_{2^{m_4}}$ & $\B{Q}\colon 2^{m_3+2}(2+m_4-m_3)-4$ \\
                                  & \qquad ($N_4=2^{m_3+1}(1+m_4-m_3)$) \\
                          & $F\colon 2^{m_3+3}(2+m_4-m_3)-28$ \\
                                  & \qquad ($M_4=2^{m_3+2}(1+m_4-m_3)$)\\[4pt]
$G/G'\iso C_4\times C_{2^{m_3+1}}\times C_{2^{m_4}}$ & $N_0=2^{m_3+3}(2+m_4-m_3)-12$ \\
                                                     & $M_0=2^{m_3+4}(2+m_4-m_3)-60$ \\ \hline
\end{tabular}
\end{center}
\bigskip

 Since $N=N_1+N_2+2N_3=2^{m_3+1}(11+5m_4-5m_3)-16$ and
$M=M_1+M_2+2M_3=2^{m_3+2}(11+5m_4-5m_3)-88$, $\B{Q}G$ is the
direct sum of $N_0=2^{m_3+2}(2+m_4-m_3)-12$ fields and
$N-N_0=2^{m_3+1}(3+m_4-m_3)-4$ quaternion algebras while $FG$
is the direct sum of $M_0=2^{m_3+4}(2+m_4-m_3)-60$ fields and
$M-M_0=2^{m_3+2}(3+m_4-m_3)-28$ quaternion algebras.

\bigskip\par\noindent (v.a) $m_1=1$, $m_2=m_3=2$, $m_4=1$, $s=t_1$.
\bigskip
\begin{center}
\small
\begin{tabular}{ccc}
Group & No. of \B{Q}-classes & No. of $F$-classes \\ \hline \\[-6pt]
$\C{Z}(G)\iso C_2\times C_4\times C_4\times C_2$ & $N_1=40$ & $M_1=40$ \\
$\C{Z}(G)/\langle s\rangle\iso C_4\times C_4\times C_2$ & $20$ & $20$ \\
$\langle\C{Z}(G),x\rangle/\langle s\rangle\iso C_8\times C_4\times C_2$ & $28$ ($N_2=8$)& $36$ ($M_2=16$)\\
$\langle\C{Z}(G),y\rangle/\langle s\rangle\iso C_4\times C_8\times C_2$ & $28$ ($N_3=8$) & $36$ ($M_3=16$) \\
$\langle\C{Z}(G),xy\rangle/\langle s\rangle\iso C_4\times C_8\times C_2$ & $28$ ($N_4=8$) & $36$ ($M_4=16$)\\
$G/G'\iso C_8\times C_8\times C_2$ & $N_0=44$ & $M_0=68$ \\ \hline
\end{tabular}
\end{center}
\bigskip

Since $N=N_1+3N_2=64$ and $M=M_1+3M_2=88$, $\B{Q}G$ is
the direct sum of $N_0=44$ fields and $N-N_0=20$
quaternion algebras while $FG$ is the direct sum of $M_0=68$ fields
and $M-M_0=20$ quaternion algebras.

\newpage
\par\noindent (v.b) $m_1=1$, $m_2\ge3$, $m_3=2$, $m_4=1$, $s=t_1$. (The conclusions
in the case $1,2,m_3\ge3,1$ are the same.)
\begin{center}
\small
\begin{tabular}{ll}
\multicolumn{1}{c}{Group} & \multicolumn{1}{c}{No. of \B{Q}- and $F$-classes} \\ \hline \\[-6pt]
$\C{Z}(G)\iso C_2\times C_{2^{m_2}}\times C_4\times C_2$ & $N_1=16m_2+8$ \\
               & $M_1=32m_2-24$ \\[4pt]
$\C{Z}(G)/\langle s\rangle\iso C_{2^{m_2}}\times C_4\times C_2$ & $\B{Q}\colon 8m_2+4$ \\
               & $F\colon 16m_2-12$ \\[4pt]
$\langle\C{Z}(G),x\rangle/\langle s\rangle\iso C_{2^{m_2+1}}\times C_4\times C_2$ & $\B{Q}\colon 8m_2+12$ ($N_2=8$) \\
               & $F\colon 16m_2+4$ ($M_2=16$)\\[4pt]
$\langle\C{Z}(G),y\rangle/\langle s\rangle\iso C_{2^{m_2}}\times C_8\times C_2$ & $\B{Q}\colon 16m_2-4$ ($N_3=8m_2-8$) \\
               & $F\colon 32m_2-28$ ($M_3=16m_2-16$) \\[4pt]
$\langle\C{Z}(G),xy\rangle/\langle s\rangle\iso C_{2^{m_2+1}}\times C_4\times C_2$ & $\B{Q}\colon 8m_2+12$ ($N_4=8$) \\
               & $F\colon 16m_2+4$ ($M_4=16$)\\[4pt]
$G/G'\iso C_{2^{m_2+1}}\times C_8\times C_2$ & $N_0=16m_2+12$ \\
               & $M_0=32m_2+4$ \\ \hline
\end{tabular}
\end{center}
\medskip

Since $N=N_1+2N_2+N_3=24m_2+16$ and $M=M_1+2M_2+M_3=48m_2-8$, $\B{Q}G$
is the direct sum of $N_0=16m_2+12$ fields
and $N-N_0=8m_2+4$
quaternion algebras while $FG$ is the direct sum of $M_0=32m_2+4$ fields
and $M-M_0=16m_2-12$ quaternion algebras.

\bigskip\par\noindent (v.c) $m_1=1$, $m_2\ge3$, $m_3\ge3$, $m_4=1$, $s=t_1$.
By symmetry in $x$ and $y$, we may assume that $m_2\ge m_3$.  We deal with the case $m_2>m_3$ first.
\begin{center}
\small
\begin{tabular}{ll}
\multicolumn{1}{c}{Group} & \multicolumn{1}{c}{No. of \B{Q}- and $F$-classes} \\ \hline \\[-6pt]
$\C{Z}(G)\iso C_2\times C_{2^{m_2}}\times C_{2^{m_3}}\times C_2$ & $N_1=2^{m_3+2}(3+m_2-m_3)-8$ \\
                                            & $M_1=2^{m_3+3}(3+m_2-m_3)-56$ \\[4pt]
$\C{Z}(G)/\langle s\rangle\iso C_{2^{m_2}}\times C_{2^{m_3}}\times C_2$ & $\B{Q}\colon 2^{m_3+1}(3+m_2-m_3)-4$ \\
                                            & $F\colon 2^{m_3+2}(3+m_2-m_3)-28$ \\[4pt]
$\langle\C{Z}(G),x\rangle/\langle s\rangle\iso C_{2^{m_2+1}}\times C_{2^{m_3}}\times C_2$ & $\B{Q}\colon 2^{m_3+1}(4+m_2-m_3)-4$ \\
                                                       &\qquad ($N_2=2^{m_3+1}$) \\
                                            & $F\colon 2^{m_3+2}(4+m_2-m_3)-28$ \\
                                                       & \qquad ($M_2=2^{m_3+2}$)\\[4pt]
$\langle\C{Z}(G),y\rangle/\langle s\rangle\iso C_{2^{m_2}}\times C_{2^{m_3+1}}\times C_2$ & $\B{Q}\colon 2^{m_3+2}(2+m_2-m_3)-4$ \\
                                                       & \qquad ($N_3=2^{m_3+1}(1+m_2-m_3)$) \\
                                            & $F\colon 2^{m_3+3}(2+m_2-m_3)-28$ \\
                                                       & \qquad ($M_3=2^{m_3+2}(1+m_2-m_3)$) \\[4pt]
$\langle\C{Z}(G),xy\rangle/\langle s\rangle\iso C_{2^{m_2+1}}\times C_{2^{m_3}}\times C_2$ & $\B{Q}\colon 2^{m_3+1}(4+m_2-m_3)-4$ \\
                                                       & \qquad ($N_4=2^{m_3+1}$) \\
                                            & $F\colon 2^{m_3+2}(4+m_2-m_3)-28$ \\
                                                       & \qquad ($M_4=2^{m_3+2}$)\\[4pt]
$G/G'\iso C_{2^{m_2+1}}\times C_{2^{m_3+1}}\times C_2$ & $N_0=2^{m_3+2}(3+m_2-m_3)-4$ \\
                                            & $M_0=2^{m_3+3}(3+m_2-m_3)-28$ \\ \hline
\end{tabular}
\end{center}
\medskip

Since $N=N_1+2N_2+N_3=2^{m_3+1}(9+3m_2-3m_3)-8$ and $M=M_1+2M_2+M_3=2^{m_3+2}(9+3m_2-3m_3)-56$,
$\B{Q}G$ is the direct sum of $N_0=2^{m_3+2}(3+m_2-m_3)$ fields and $N-N_0=2^{m_3+1}(3+m_2-m_3)-4$
quaternion algebras while $FG$ is the direct sum of $M_0=2^{m_3+3}(3+m_2-m_3)-28$ fields
and $M-M_0=2^{m_3+2}(3+m_2-m_3)-28$ quaternion algebras.

\medskip
When $m_2=m_3$, here's what happens.
\bigskip

\begin{center}
\small
\begin{tabular}{ll}
\multicolumn{1}{c}{Group} & \multicolumn{1}{c}{No. of \B{Q}- and $F$-classes} \\ \hline \\[-6pt]
$\C{Z}(G)\iso C_2\times C_{2^{m_2}}\times C_{2^{m_2}}\times C_2$ & $N_1=2^{m_2+2}(3)-8$ \\
                                                                 & $M_1=2^{m_2+3}(3)-56$ \\[4pt]
$\C{Z}(G)/\langle s\rangle\iso C_{2^{m_2}}\times C_{2^{m_2}}\times C_2$ & $\B{Q}\colon 2^{m_2+1}(3)-4$ \\
                            & $F\colon 2^{m_2+2}(3)-28$ \\[4pt]
$\langle\C{Z}(G),x\rangle/\langle s\rangle\iso C_{2^{m_2+1}}\times C_{2^{m_2}}\times C_2$ & $\B{Q}\colon 2^{m_2+3}-4$ ($N_2=2^{m_2+1}$) \\
                        & $F\colon 2^{m_2+4}-28$ ($M_2=2^{m_2+2}$)\\[4pt]
$\langle\C{Z}(G),y\rangle/\langle s\rangle\iso C_{2^{m_2}}\times C_{2^{m_2+1}}\times C_2$ & $\B{Q}\colon 2^{m_2+3}-4$ ($N_3=2^{m_2+1}$) \\
                        & $F\colon 2^{m_2+4}-28$ ($M_3=2^{m_2+2}$) \\[4pt]
$\langle\C{Z}(G),xy\rangle/\langle s\rangle\iso C_{2^{m_2}}\times C_{2^{m_2+1}}\times C_2$ & $\B{Q}\colon 2^{m_2+3}-4$ ($N_4=2^{m_2+1}$) \\
                        & $F\colon 2^{m_2+4}-28$ ($M_4=2^{m_2+2}$)\\[4pt]
$G/G'\iso C_{2^{m_2+1}}\times C_{2^{m_2+1}}\times C_2$ & $N_0=2^{m_2+2}(3)-4$ \\
                        & $M_0=2^{m_2+3}(3)-28$ \\ \hline
\end{tabular}
\end{center}
\medskip

Since $N=N_1+3N_2=2^{m_2+1}(9)-8$ and $M=M_1+3M_2=2^{m_2+2}(9)-56$,
$\B{Q}G$ is the direct sum of $N_0=2^{m_2+2}(3)-4$ fields and $N-N_0=2^{m_2+1}(3)-4$
quaternion algebras while $FG$ is the direct sum of $M_0=2^{m_2+3}(3)-28$ fields
and $M-M_0=2^{m_2+2}(3)-28$ quaternion algebras.

\bigskip\par\noindent (vi.a) $m_1=1$, $m_2=m_3=m_4=2$, $s=t_1$.
\bigskip
\begin{center}
\small
\begin{tabular}{ccc}
Group & No. of \B{Q}-classes & No. of $F$-classes \\ \hline \\[-6pt]
$\C{Z}(G)\iso C_2\times C_4\times C_4\times C_4$ & $N_1=72$ & $M_1=72$ \\
$\C{Z}(G)/\langle s\rangle\iso C_4\times C_4\times C_4$ & $36$ & $36$ \\
$\langle\C{Z}(G),x\rangle/\langle s\rangle\iso C_8\times C_4\times C_4$ & $52$ ($N_2=16$)& $68$ ($M_2=32$)\\
$\langle\C{Z}(G),y\rangle/\langle s\rangle\iso C_4\times C_8\times C_4$ & $52$ ($N_3=16$) & $68$ ($M_3=32$) \\
$\langle\C{Z}(G),xy\rangle/\langle s\rangle\iso C_4\times C_8\times C_4$ & $52$ ($N_4=16$) & $68$ ($M_4=32$)\\
$G/G'\iso C_8\times C_8\times C_4$ & $N_0=84$ & $M_0=132$ \\ \hline
\end{tabular}
\end{center}
\bigskip

 Since $N=N_1+3N_2=120$ and $M=M_1+3M_2=168$, $\B{Q}G$ is
the direct sum of $N_0=84$ fields and $N-N_0=36$ quaternion
algebras while $FG$ is the direct sum of $M_0=132$ fields and
$M-M_0=36$ quaternion algebras.

\newpage
\par\noindent (vi.b) $m_1=1$, $m_2\ge3$, $m_3=2$, $m_4=2$, $s=t_1$.  (The case $1,2,m_3\ge3,2$
yields identical results.)
\bigskip
\begin{center}
\small
\begin{tabular}{ll}
\multicolumn{1}{c}{Group} & \multicolumn{1}{c}{No. of \B{Q}- and $F$-classes} \\ \hline \\[-6pt]
$\C{Z}(G)\iso C_2\times C_{2^{m_2}}\times C_4\times C_4$ & $N_1=32m_2+8$ \\
               & $M_1=64m_2-56$ \\[4pt]
$\C{Z}(G)/\langle s\rangle\iso C_{2^{m_2}}\times C_4\times C_4$ & $\B{Q}\colon 16m_2+4$ \\
               & $F\colon 32m_2-28$ \\[4pt]
$\langle\C{Z}(G),x\rangle/\langle s\rangle\iso C_{2^{m_2+1}}\times C_4\times C_4$ & $\B{Q}\colon 16m_2+20$ ($N_2=16$) \\
               & $F\colon 32m_2+4$ ($M_2=32$)\\[4pt]
$\langle\C{Z}(G),y\rangle/\langle s\rangle\iso C_{2^{m_2}}\times C_8\times C_4$ & $\B{Q}\colon 32m_2-12$ ($N_3=16m_2-16$) \\
               & $F\colon 64m_2-60$ ($M_3=32m_2-32$) \\[4pt]
$\langle\C{Z}(G),xy\rangle/\langle s\rangle\iso C_{2^{m_2+1}}\times C_4\times C_4$ & $\B{Q}\colon 16m_2+20$ ($N_4=16$) \\
               & $F\colon 32m_2+4$ ($M_4=32$)\\[4pt]
$G/G'\iso C_{2^{m_2+1}}\times C_8\times C_4$ & $N_0=32m_2+20$ \\
               & $M_0=64m_2+4$ \\ \hline
\end{tabular}
\end{center}
\bigskip

Since $N=N_1+2N_2+N_3=48m_2+24$ and $M=M_1+2M_2+M_3=96m_2-24$, $\B{Q}G$
is the direct sum of $N_0=32m_2+20$ fields
and $N-N_0=16m_2+4$
quaternion algebras while $FG$ is the direct sum of $M_0=64m_2+4$ fields
and $M-M_0=32m_2-28$ quaternion algebras.

\bigskip\par\noindent (vi.c) $m_1=1$, $m_2=2$, $m_3=2$, $m_4\ge3$, $s=t_1$.
\bigskip
\begin{center}
\small
\begin{tabular}{ll}
\multicolumn{1}{c}{Group} & \multicolumn{1}{c}{No. of \B{Q}- and $F$-classes} \\ \hline \\[-6pt]
$\C{Z}(G)\iso C_2\times C_4\times C_4\times C_{2^{m_4}}$ & $N_1=32m_4+8$ \\
               & $M_1=64m_2-56$ \\[4pt]
$\C{Z}(G)/\langle s\rangle\iso C_4\times C_4\times C_{2^{m_4}}$ & $\B{Q}\colon 16m_4+4$ \\
               & $F\colon 32m_4-28$ \\[4pt]
$\langle\C{Z}(G),x\rangle/\langle s\rangle\iso C_8\times C_4\times C_{2^{m_4}}$ & $\B{Q}\colon 32m_4-12$ ($N_2=16m_4-16$) \\
               & $F\colon 64m_4-60$ ($M_2=32m_4-32$)\\[4pt]
$\langle\C{Z}(G),y\rangle/\langle s\rangle\iso C_4\times C_8\times C_{2^{m_4}}$ & $\B{Q}\colon 32m_4-12$ ($N_3=16m_4-16$) \\
               & $F\colon 64m_4-60$ ($M_3=32m_4-32$) \\[4pt]
$\langle\C{Z}(G),xy\rangle/\langle s\rangle\iso C_4\times C_8\times C_{2^{m_4}}$ & $\B{Q}\colon 32m_4-12$ ($N_4=16m_4-16$) \\
               & $F\colon 64m_4-60$ ($M_4=32m_4-32$)\\[4pt]
$G/G'\iso C_8\times C_8\times C_{2^{m_4}}$ & $N_0=64m_4-44$ \\
               & $M_0=128m_4-124$ \\ \hline
\end{tabular}
\end{center}
\bigskip

Since $N=N_1+3N_2=80m_4-40$ and $M=M_1+3M_2=160m_4-152$, $\B{Q}G$
is the direct sum of $N_0=64m_4-44$ fields
and $N-N_0=16m_4+4$
quaternion algebras while $FG$ is the direct sum of $M_0=128m_4-124$ fields
and $M-M_0=32m_4-28$ quaternion algebras.

\newpage
\par\noindent (vi.d) $m_1=1$, $m_2\ge3$, $m_3\ge3$, $m_4=2$, $s=t_1$.   Because
of the symmetry with respect to $x$ and $y$, we may assume that $m_2\ge m_3$.
Assume first, in fact, that $m_2>m_3$.

\begin{center}
\small
\begin{tabular}{ll}
\multicolumn{1}{c}{Group} & \multicolumn{1}{c}{No. of \B{Q}- and $F$-classes} \\ \hline \\[-6pt]
$\C{Z}(G)\iso C_2\times C_{2^{m_2}}\times C_{2^{m_3}}\times C_4$ & $N_1=2^{m_3+3}(3+m_2-m_3)-24$ \\
               & $M_1=2^{m_3+4}(3+m_2-m_3)-120$ \\[4pt]
$\C{Z}(G)/\langle s\rangle\iso C_{2^{m_2}}\times C_{2^{m_3}}\times C_4$ & $\B{Q}\colon 2^{m_3+2}(3+m_2-m_3)-12$ \\
               & $F\colon 2^{m_3+3}(3+m_2-m_3)-60$ \\[4pt]
$\langle\C{Z}(G),x\rangle/\langle s\rangle\iso C_{2^{m_2+1}}\times C_{2^{m_3}}\times C_4$ & $\B{Q}\colon 2^{m_3+2}(4+m_2-m_3)-12$ \\
                         & \qquad ($N_2=2^{m_3+2}$) \\
               & $F\colon 2^{m_3+3}(4+m_2-m_3)-60$ \\
                         & \qquad ($M_2=2^{m_3+3}$)\\[4pt]
$\langle\C{Z}(G),y\rangle/\langle s\rangle\iso C_{2^{m_2}}\times C_{2^{m_3+1}}\times C_4$ & $\B{Q}\colon 2^{m_3+3}(2+m_2-m_3)-12$ \\
                         & \qquad ($N_3=2^{m_3+2}(1+m_2-m_3)$) \\
               & $F\colon 2^{m_3+4}(2+m_2-m_3)-60$ \\
                         & \qquad ($M_3=2^{m_3+3}(1+m_2-m_3)$) \\[4pt]
$\langle\C{Z}(G),xy\rangle/\langle s\rangle\iso C_{2^{m_2+1}}\times C_{2^{m_3}}\times C_4$ & $\B{Q}\colon 2^{m_3+2}(4+m_2-m_3)-12$ \\
                         & \qquad ($N_4=2^{m_3+2}$) \\
               & $F\colon 2^{m_3+3}(4+m_2-m_3)-60$ \\
                         & \qquad($M_4=2^{m_3+3}$)\\[4pt]
$G/G'\iso C_{2^{m_2+1}}\times C_{2^{m_3+1}}\times C_4$ & $N_0=2^{m_3+3}(3+m_2-m_3)-12$ \\
               & $M_0=2^{m_3+4}(3+m_2-m_3)-60$ \\ \hline
\end{tabular}
\end{center}
\medskip

Since $N=N_1+2N_2+N_3=2^{m_3+2}(9+3m_2-3m_3)-24$ and $M=M_1+2M_2+M_3=2^{m_3+3}(9+3m_2-3m_2)-120$,
$\B{Q}G$ is the direct sum of $N_0=2^{m_3+3}(3+m_2-m_3)-12$ fields
and $N-N_0=2^{m_3+2}(3+m_2-m_3)-12$
quaternion algebras while $FG$ is the direct sum of $M_0=2^{m_3+4}(3+m_2-m_3)-60$ fields
and $M-M_0=2^{m_3+3}(3+m_2-m_3)-60$ quaternion algebras.

\medskip
Now assume $m_2=m_3$.
\begin{center}
\small
\begin{tabular}{ll}
\multicolumn{1}{c}{Group} & \multicolumn{1}{c}{No. of \B{Q}- and $F$-classes} \\ \hline \\[-6pt]
$\C{Z}(G)\iso C_2\times C_{2^{m_2}}\times C_{2^{m_2}}\times C_4$ & $N_1=2^{m_2+3}(3)-24$ \\
               & $M_1=2^{m_2+4}(3)-120$ \\[4pt]
$\C{Z}(G)/\langle s\rangle\iso C_{2^{m_2}}\times C_{2^{m_2}}\times C_4$ & $\B{Q}\colon 2^{m_2+2}(3)-12$ \\
               & $F\colon 2^{m_2+3}(3)-60$ \\[4pt]
$\langle\C{Z}(G),x\rangle/\langle s\rangle\iso C_{2^{m_2+1}}\times C_{2^{m_2}}\times C_4$ & $\B{Q}\colon 2^{m_2+2}(4)-12$ ($N_2=2^{m_2+2}$) \\
               & $F\colon 2^{m_2+3}(4)-60$ ($M_2=2^{m_2+3}$)\\[4pt]
$\langle\C{Z}(G),y\rangle/\langle s\rangle\iso C_{2^{m_2}}\times C_{2^{m_2+1}}\times C_4$ & $\B{Q}\colon 2^{m_2+4}-12$ ($N_3=2^{m_2+2}$) \\
               & $F\colon 2^{m_2+5}-60$ ($M_3=2^{m_2+3}$) \\[4pt]
$\langle\C{Z}(G),xy\rangle/\langle s\rangle\iso C_{2^{m_2+1}}\times C_{2^{m_2}}\times C_4$ & $\B{Q}\colon 2^{m_2+4}-12$ ($N_4=2^{m_2+2}$) \\
               & $F\colon 2^{m_2+5}-60$ ($M_4=2^{m_2+3}$)\\[4pt]
$G/G'\iso C_{2^{m_2+1}}\times C_{2^{m_2+1}}\times C_4$ & $N_0=2^{m_2+3}(3)-12$ \\
               & $M_0=2^{m_2+4}(3)-60$ \\ \hline
\end{tabular}
\end{center}
\medskip

Since $N=N_1+3N_2=2^{m_2+2}(9)-24$ and $M=M_1+3M_2=2^{m_2+3}(9)-120$, $\B{Q}G$
is the direct sum of $N_0=2^{m_2+3}(3)-12$ fields
and $N-N_0=2^{m_2+2}(3)-12$
quaternion algebras while $FG$ is the direct sum of $M_0=2^{m_2+4}(3)-60$ fields
and $M-M_0=2^{m_2+3}(3)-60$ quaternion algebras.

\bigskip\par\noindent (vi.e) $m_1=1$, $m_2\ge3$, $m_3=2$, $m_4\ge3$, $s=t_1$. (The case $1,2,m_3\ge3,m_4\ge3$
yields identical results.)

Here there two possibilities, the first being the case $m_2\ge m_4$.
\bigskip
\begin{center}
\small
\begin{tabular}{ll}
\multicolumn{1}{c}{Group} & \multicolumn{1}{c}{No. of \B{Q}- and $F$-classes} \\ \hline \\[-6pt]
$\C{Z}(G)\iso C_2\times C_{2^{m_2}}\times C_4\times C_{2^{m_4}}$ & $N_1=2^{m_4+3}(3+m_2-m_4)-24$ \\
               & $M_1=2^{m_4+4}(3+m_2-m_4)-120$ \\[4pt]
$\C{Z}(G)/\langle s\rangle\iso C_{2^{m_2}}\times C_4\times C_{2^{m_4}}$ & $\B{Q}\colon 2^{m_4+2}(3+m_2-m_4)-12$ \\
               & $F\colon 2^{m_4+3}(3+m_2-m_4)-60$ \\[4pt]
$\langle\C{Z}(G),x\rangle/\langle s\rangle\iso C_{2^{m_2+1}}\times C_4\times C_{2^{m_4}}$ & $\B{Q}\colon 2^{m_4+2}(4+m_2-m_4)-12$ \\
                     & \qquad ($N_2=2^{m_4+2}$) \\
               & $F\colon 2^{m_4+3}(4+m_2-m_4)-60$ \\
                     & \qquad ($M_2=2^{m_4+3}$)\\[4pt]
$\langle\C{Z}(G),y\rangle/\langle s\rangle\iso C_{2^{m_2}}\times C_8\times C_{2^{m_4}}$ & $\B{Q}\colon 2^{m_4+3}(3+m_2-m_4)-44$ \\
                     & \qquad ($N_3=2^{m_4+2}(3+m_2-m_4)-32$) \\
               & $F\colon 2^{m_4+4}(3+m_2-m_4)-124$ \\
                     & \qquad ($M_3=2^{m_4+3}(3+m_2-m_4)-64$) \\[4pt]
$\langle\C{Z}(G),xy\rangle/\langle s\rangle\iso C_{2^{m_2+1}}\times C_4\times C_{2^{m_4}}$ & $\B{Q}\colon 2^{m_4+2}(4+m_2-m_4)-12$ \\
                     & \qquad ($N_4=2^{m_4+2}$) \\
               & $F\colon 2^{m_4+3}(4+m_2-m_4)-60$ \\
                     & \qquad ($M_4=2^{m_4+3}$)\\[4pt]
$G/G'\iso C_{2^{m_2+1}}\times C_8\times C_{2^{m_4}}$ & $N_0=2^{m_4+3}(4+m_2-m_4)-44$ \\
               & $M_0=2^{m_4+4}(4+m_2-m_4)-124$ \\ \hline
\end{tabular}
\end{center}
\bigskip

Since $N=N_1+2N_2+N_3=2^{m_4+2}(11+3m_2-3m_4)-56$ and $M=M_1+2M_2+M_3=2^{m_4+3}(11+3m_2-3m_4)-184$,
$\B{Q}G$ is the direct sum of $N_0=2^{m_4+3}(4+m_2-m_4)-44$ fields
and $N-N_0=2^{m_4+2}(3+m_2-m_4)-12$
quaternion algebras while $FG$ is the direct sum of $M_0=2^{m_4+4}(4+m_2-m_4)-124$ fields
and $M-M_0=2^{m_4+3}(3+m_2-m_4)-60$ quaternion algebras.

\newpage
Now assume $m_4>m_2$.
\bigskip

\begin{center}
\small
\begin{tabular}{ll}
\multicolumn{1}{c}{Group} & \multicolumn{1}{c}{No. of \B{Q}- and $F$-classes} \\ \hline \\[-6pt]
$\C{Z}(G)\iso C_2\times C_{2^{m_2}}\times C_4\times C_{2^{m_4}}$ & $N_1=2^{m_2+3}(3+m_4-m_2)-24$ \\
               & $M_1=2^{m_2+4}(3+m_4-m_2)-120$ \\[4pt]
$\C{Z}(G)/\langle s\rangle\iso C_{2^{m_2}}\times C_4\times C_{2^{m_4}}$ & $\B{Q}\colon 2^{m_2+2}(3+m_4-m_2)-12$ \\
               & $F\colon 2^{m_2+3}(3+m_4-m_2)-60$ \\[4pt]
$\langle\C{Z}(G),x\rangle/\langle s\rangle$   & $\B{Q}\colon 2^{m_2+3}(2+m_4-m_2)-12$ \\
  \quad $\iso C_{2^{m_2+1}}\times C_4\times C_{2^{m_4}}$  & \qquad ($N_2=2^{m_2+2}(1+m_4-m_2)$) \\
               & $F\colon 2^{m_2+3}(2+m_4-m_2)-60$ \\
                     & \qquad ($M_2=2^{m_2+3}(1+m_4-m_2)$)\\[4pt]
$\langle\C{Z}(G),y\rangle/\langle s\rangle$   & $\B{Q}\colon 2^{m_2+3}(3+m_4-m_2)-44$ \\
  \quad $\iso C_{2^{m_2}}\times C_8\times C_{2^{m_4}}$  & \qquad ($N_3=2^{m_2+2}(3+m_4-m_2)-32$) \\
               & $F\colon 2^{m_2+4}(3+m_4-m_2)-124$ \\
                     & \qquad ($M_3=2^{m_2+3}(3+m_4-m_2)-64$) \\[4pt]
$\langle\C{Z}(G),xy\rangle/\langle s\rangle$    & $\B{Q}\colon 2^{m_2+3}(2+m_4-m_2)-12$ \\
 \quad $\iso C_{2^{m_2+1}}\times C_4\times C_{2^{m_4}}$   & \qquad ($N_4=2^{m_2+2}(1+m_4-m_2)$) \\
                & $F\colon 2^{m_2+4}(2+m_4-m_2)-60$ \\
                  & \qquad ($M_4=2^{m_2+3}(1+m_4-m_2)$)\\[4pt]
$G/G'\iso C_{2^{m_2+1}}\times C_8\times C_{2^{m_4}}$ & $N_0=2^{m_2+4}(2+m_4-m_2)-44$ \\
               & $M_0=2^{m_2+5}(2+m_4-m_2)-124$ \\ \hline
\end{tabular}
\end{center}
\bigskip

Since $N=N_1+2N_2+N_3=2^{m_2+2}(11+5m_4-5m_2)-56$ and $M=M_1+2M_2+M_3=2^{m_2+3}(11+5m_4-5m_2)-184$,
$\B{Q}G$ is the direct sum of $N_0=2^{m_2+4}(2+m_4-m_2)-44$ fields
and $N-N_0=2^{m_2+2}(3+m_4-m_2)-12$
quaternion algebras while $FG$ is the direct sum of $M_0=2^{m_2+5}(2+m_4-m_2)-124$ fields
and $M-M_0=2^{m_2+3}(3+m_4-m_2)-60$ quaternion algebras.

\bigskip\par\noindent (vi.f) $m_1=1$, $m_2\ge3$, $m_3\ge3$, $m_4\ge3$, $s=t_1$.
Because of the symmetry in $x$ and $y$, we may assume that $m_2\ge m_3$, but this still leaves three
possibilities.  Many of the numbers that arise involve the number of cyclic subgroups of the direct
product of $C_{2^a}\times C_{2^b}\times C_{2^c}$ which for convenience in what follows, and assuming $a\ge b\ge c$,
we denote $f(a,b,c)$.  Thus, from \lemref{lem:3cyclics}, we
have
\begin{equation*}
f(a,b,c)=\tfrac73(4^c-1)+3(2^c)(2^b-2^c)+(a-b)2^{b+c}+1.
\end{equation*}

\newpage
We investigate the first case, $m_2\ge m_3\ge m_4$, in two steps,
supposing first that $m_2>m_3\ge m_4$.
\bigskip
\begin{center}
\footnotesize
\begin{tabular}{ll}
\multicolumn{1}{c}{Group} & \multicolumn{1}{c}{No. of \B{Q}- and $F$-classes} \\ \hline \\[-6pt]
$\C{Z}(G)$ & $N_1=2f(m_2,m_3,m_4)$ \\
                \quad $\iso C_2\times C_{2^{m_2}}\times C_{2^{m_3}}\times C_{2^{m_4}}$  & $M_1=2N_1-72$ \\[8pt]
$\C{Z}(G)/\langle s\rangle$ & $\B{Q}\colon f(m_2,m_3,m_4)$ \\
  \quad $\iso C_{2^{m_2}}\times C_{2^{m_3}}\times C_{2^{m_4}}$  & $F\colon 2f(m_2,m_3,m_4)-36$ \\[8pt]
$\langle\C{Z}(G),x\rangle/\langle s\rangle$ & $\B{Q}\colon f(m_2+1,m_3,m_4)$ \\
  \quad $\iso C_{2^{m_2+1}}\times C_{2^{m_3}}\times C_{2^{m_4}}$   & \quad ($N_2=f(m_2+1,m_3,m_4)-f(m_2,m_3,m_4)$) \\
               & $F\colon 2f(m_2+1,m_3,m_4)-36$ \\
                     & \quad ($M_2=2N_2$)\\[8pt]
$\langle\C{Z}(G),y\rangle/\langle s\rangle$ & $\B{Q}\colon f(m_2,m_3+1,m_4)$ \\
  \quad $\iso C_{2^{m_2}}\times C_{2^{m_3+1}}\times C_{2^{m_4}}$  & \quad ($N_3=f(m_2,m_3+1,m_4)-f(m_2,m_3,m_4)$) \\
               & $F\colon 2f(m_2,m_3+1,m_4)-36$ \\
                     & \quad ($M_3=2N_3$) \\[8pt]
$\langle\C{Z}(G),xy\rangle/\langle s\rangle$ & $\B{Q}\colon f(m_2+1,m_3,m_4)$ \\
  \quad $\iso C_{2^{m_2+1}}\times C_{2^{m_3}}\times C_{2^{m_4}}$  & \quad ($N_4=N_2$) \\
                & $F\colon 2f(m_2+1,m_3,m_4)-36$ \\
                     & \quad ($M_4=M_2$)\\[8pt]
$G/G'$ & $N_0=f(m_2+1,m_3+1,m_4)$ \\
  \quad $\iso C_{2^{m_2+1}}\times C_{2^{m_3+1}}\times C_{2^{m_4}}$ & $M_0=2N_0-36$ \\ \hline
\end{tabular}
\end{center}
\bigskip

Since
$N=N_1+2N_2+N_3=2f(m_2+1,m_3,m_4)+f(m_2,m_3+1,m_4)-f(m_2,m_3,m_4)=2^{m_3+m_4}(9+3m_2-3m_3)
-\tfrac43(4^{m_4}+2)$
and $M=M_1+2M_2+M_3=2N-72$, $\B{Q}G$ is the direct sum of
\begin{equation*}
N_0=f(m_2+1,m_3+1,m_4)=2^{m_3+m_4}(6+2m_2-2m_3)-\tfrac23(4^{m_4}+2)
\end{equation*}
fields and
\begin{equation*}
N-N_0=2^{m_3+m_4}(3+m_2-m_3)-\tfrac23(4^{m_4}+2)
\end{equation*}
quaternion algebras while $FG$ is the direct sum of
$M_0=2N_0-36$ fields and $M-M_0=2(N-N_0)-36$ quaternion
algebras.

\newpage
Now suppose $m_2=m_3\ge m_4$.
\bigskip
\begin{center}
\footnotesize
\begin{tabular}{ll}
\multicolumn{1}{c}{Group} & \multicolumn{1}{c}{No. of \B{Q}- and $F$-classes} \\ \hline \\[-6pt]
$\C{Z}(G)$ & $N_1=2f(m_2,m_2,m_4)$ \\
  \quad $\iso C_2\times C_{2^{m_2}}\times C_{2^{m_2}}\times C_{2^{m_4}}$  & $M_1=2N_1-72$ \\[8pt]
$\C{Z}(G)/\langle s\rangle$ & $\B{Q}\colon f(m_2,m_2,m_4)$ \\
  \quad $\iso C_{2^{m_2}}\times C_{2^{m_2}}\times C_{2^{m_4}}$  & $F\colon 2f(m_2,m_2,m_4)-36$ \\[8pt]
$\langle\C{Z}(G),x\rangle/\langle s\rangle$  & $\B{Q}\colon f(m_2+1,m_2,m_4)$ \\
  \quad $\iso C_{2^{m_2+1}}\times C_{2^{m_2}}\times C_{2^{m_4}}$   & \quad ($N_2=f(m_2+1,m_2,m_4)-f(m_2,m_2,m_4)$) \\
               & $F\colon 2f(m_2+1,m_2,m_4)-36$ \\
                     & \quad ($M_2=2N_2$)\\[8pt]
$\langle\C{Z}(G),y\rangle/\langle s\rangle$ & $\B{Q}\colon f(m_2+1,m_2,m_4)$ \\
  \quad $\iso C_{2^{m_2}}\times C_{2^{m_2+1}}\times C_{2^{m_4}}$  & \quad ($N_3=N_2$) \\
               & $F\colon 2f(m_2+1,m_2,m_4)-36$ \\
                     & \quad ($M_3=2N_2$) \\[8pt]
$\langle\C{Z}(G),xy\rangle/\langle s\rangle$  & $\B{Q}\colon f(m_2+1,m_2,m_4)$ \\
  \quad $\iso C_{2^{m_2+1}}\times C_{2^{m_2}}\times C_{2^{m_4}}$  & \quad ($N_4=N_2$) \\
                & $F\colon 2f(m_2+1,m_2,m_4)-36$ \\
                     & \quad ($M_4=M_2$)\\[8pt]
$G/G'$ & $N_0=f(m_2+1,m_2+1,m_4)$ \\
  \quad $\iso C_{2^{m_2+1}}\times C_{2^{m_2+1}}\times C_{2^{m_4}}$  & $M_0=2N_0-36$ \\ \hline
\end{tabular}
\end{center}
\bigskip

Since $N=N_1+3N_2=3f(m_2+1,m_2,m_4)-f(m_2,m_2,m_4)=9(2^{m_2+m_4})-\tfrac43(4^{m_4}+2)$
and $M=M_1+3M_2=2N-72$, $\B{Q}G$ is the direct sum of
\begin{equation*}
N_0=6(2^{m_2+m_4})-\tfrac23(4^{m_4}+2)
\end{equation*}
fields and
\begin{align*}
N-N_0 &= 3f(m_2+1,m_2,m_4)-f(m_2,m_2,m_4)-f(m_2+1,m_2+1,m_4) \\
  &=3(2^{m_2+m_4})-\tfrac23(4^{m_4}+2)
\end{align*}
quaternion algebras, while $FG$ is the direct sum of
$M_0=2N_0-36$ fields and $M-M_0=2(N-N_0)-36$ quaternion
algebras.

\newpage
The case $m_4\ge m_2\ge m_3$ we again split into several subcases, considering the possibility
$m_4>m_2>m_3$ first.
\bigskip
\begin{center}
\footnotesize
\begin{tabular}{ll}
\multicolumn{1}{c}{Group} & \multicolumn{1}{c}{No. of \B{Q}- and $F$-classes} \\ \hline \\[-6pt]
$\C{Z}(G)$ & $N_1=2f(m_4,m_2,m_3)$ \\
  \quad $\iso C_2\times C_{2^{m_2}}\times C_{2^{m_3}}\times C_{2^{m_4}}$   & $M_1=2N_1-72$ \\[8pt]
$\C{Z}(G)/\langle s\rangle$ & $\B{Q}\colon f(m_4,m_2,m_3)$ \\
 \quad $\iso C_{2^{m_2}}\times C_{2^{m_3}}\times C_{2^{m_4}}$  & $F\colon 2f(m_4,m_2,m_3)-36$ \\[8pt]
$\langle\C{Z}(G),x\rangle/\langle s\rangle$ & $\B{Q}\colon f(m_4,m_2+1,m_3)$ \\
  \quad $\iso C_{2^{m_2+1}}\times C_{2^{m_3}}\times C_{2^{m_4}}$   & \quad ($N_2=f(m_4,m_2+1,m_3)-f(m_4,m_2,m_3)$) \\
               & $F\colon 2f(m_4,m_2+1,m_3)-36$ \\
                     & \quad ($M_2=2N_2$)\\[8pt]
$\langle\C{Z}(G),y\rangle/\langle s\rangle$ & $\B{Q}\colon f(m_4,m_2,m_3+1)$ \\
  \quad $\iso C_{2^{m_2}}\times C_{2^{m_3+1}}\times C_{2^{m_4}}$    & \quad ($N_3=f(m_4,m_2,m_3+1)-f(m_4,m_2,m_3)$) \\
               & $F\colon 2f(m_4,m_3+1,m_2)-36$ \\
                     & \quad ($M_3=2N_3$) \\[8pt]
$\langle\C{Z}(G),xy\rangle/\langle s\rangle$ & $\B{Q}\colon f(m_4,m_2+1,m_3)$ \\
  \quad $\iso C_{2^{m_2+1}}\times C_{2^{m_3}}\times C_{2^{m_4}}$   & \quad ($N_4=N_2$) \\
                & $F\colon 2f(m_4,m_2+1,m_3)-36$ \\
                  & \quad ($M_4=M_2$)\\[8pt]
$G/G'$ & $N_0=f(m_4,m_2+1,m_3+1)$ \\
  \quad $\iso C_{2^{m_2+1}}\times C_{2^{m_3+1}}\times C_{2^{m_4}}$   & $M_0=2N_0-36$ \\ \hline
\end{tabular}
\end{center}
\bigskip

Since $N=N_1+2N_2+N_3=2f(m_4,m_2+1,m_3)+f(m_4,m_2,m_3+1)-f(m_4,m_2,m_3)
=2^{m_2+m_3}(11+5m_4-5m_2)-\tfrac23[5(4^{m_3})+4]$ and $M=M_1+2M_2+M_3=2N-72$,
$\B{Q}G$ is the direct sum of
\begin{equation*}
N_0=2^{m_2+m_3}(8+4m_4-4m_2)-\tfrac43[2(4^{m_3})+1]
\end{equation*}
fields and
\begin{align*}
N-N_0  &= 2f(m_4,m_2+1,m_3)+f(m_4,m_2,m_3+1)-f(m_4,m_2,m_3) \\
   & \qquad -f(m_4,m_2+1,m_3+1) = 2^{m_2+m_3}(3+m_4-m_2)-\tfrac23(4^{m_3}+2)
\end{align*}
quaternion algebras, while $FG$ is the direct sum of $M_0=2N_0-36$ fields
and $M-M_0=2(N-N_0)-36$ quaternion algebras.

\newpage
Now suppose $m_4>m_2=m_3$.
\bigskip
\begin{center}
\footnotesize
\begin{tabular}{ll}
\multicolumn{1}{c}{Group} & \multicolumn{1}{c}{No. of \B{Q}- and $F$-classes} \\ \hline \\[-6pt]
$\C{Z}(G)$  & $N_1=2f(m_4,m_2,m_2)$ \\
   \quad $\iso C_2\times C_{2^{m_2}}\times C_{2^{m_2}}\times C_{2^{m_4}}$  & $M_1=2N_1-72$ \\[8pt]
$\C{Z}(G)/\langle s\rangle$  & $\B{Q}\colon f(m_4,m_2,m_2)$ \\
  \quad $\iso C_{2^{m_2}}\times C_{2^{m_2}}\times C_{2^{m_4}}$  & $F\colon 2f(m_4,m_2,m_2)-36$ \\[8pt]
$\langle\C{Z}(G),x\rangle/\langle s\rangle$ & $\B{Q}\colon f(m_4,m_2+1,m_2)$ \\
  \quad $\iso C_{2^{m_2+1}}\times C_{2^{m_2}}\times C_{2^{m_4}}$   & \quad ($N_2=f(m_4,m_2+1,m_2)-f(m_4,m_2,m_2)$) \\
               & $F\colon 2f(m_4,m_2+1,m_2)-36$ \\
                     & \quad ($M_2=2N_2$)\\[8pt]
$\langle\C{Z}(G),y\rangle/\langle s\rangle$ & $\B{Q}\colon f(m_4,m_2+1,m_2)$ \\
   \quad $\iso C_{2^{m_2}}\times C_{2^{m_2+1}}\times C_{2^{m_4}}$  & \quad ($N_3=N_2$) \\
               & $F\colon 2f(m_4,m_2+1,m_2)-36$ \\
                     & \quad ($M_3=M_2$) \\[8pt]
$\langle\C{Z}(G),xy\rangle/\langle s\rangle$  & $\B{Q}\colon f(m_4,m_2+1,m_2)$ \\
  \quad $\iso C_{2^{m_2+1}}\times C_{2^{m_2}}\times C_{2^{m_4}}$   & \quad ($N_4=N_2$) \\
                & $F\colon 2f(m_4,m_2+1,m_2)-36$ \\
                  & \quad ($M_4=M_2$)\\[8pt]
$G/G'$ & $N_0=f(m_4,m_2+1,m_2+1)$ \\
   \quad $\iso C_{2^{m_2+1}}\times C_{2^{m_2+1}}\times C_{2^{m_4}}$  & $M_0=2N_0-36$ \\ \hline
\end{tabular}
\end{center}
\bigskip

Since $N=N_1+3N_2=3f(m_4,m_2+1,m_2)-f(m_4,m_2,m_2)
=4^{m_2}(5m_4-5m_2)+\tfrac13[23(4^{m_2})-8]$
and $M=M_1+3M_2=2N-72$,
$\B{Q}G$ is the direct sum of
\begin{equation*}
N_0=4^{m_2}(4m_4-4m_2)+\tfrac43(4^{m_2+1}-1)
\end{equation*}
fields and
\begin{align*}
N-N_0  &= 3f(m_4,m_2+1,m_2)-f(m_4,m_2,m_2)-f(m_4,m_2+1,m_2+1) \\
  &=4^{m_2}(m_4-m_2)+\tfrac13[7(4^{m_2})-4]
\end{align*}
quaternion algebras while $FG$ is the direct sum of $M_0=2N_0-36$ fields
and $M-M_0=2(N-N_0)-36$ quaternion algebras.

\newpage
Suppose $m_4=m_2>m_3$.
\bigskip
\begin{center}
\footnotesize
\begin{tabular}{ll}
\multicolumn{1}{c}{Group} & \multicolumn{1}{c}{No. of \B{Q}- and $F$-classes} \\ \hline \\[-6pt]
$\C{Z}(G)$ & $N_1=2f(m_2,m_2,m_3)$ \\
8   \quad $\iso C_2\times C_{2^{m_2}}\times C_{2^{m_3}}\times C_{2^{m_2}}$  & $M_1=2N_1-72$ \\[8pt]
$\C{Z}(G)/\langle s\rangle$ & $\B{Q}\colon f(m_2,m_2,m_3)$ \\
  \quad $\iso C_{2^{m_2}}\times C_{2^{m_3}}\times C_{2^{m_2}}$   & $F\colon 2f(m_2,m_2,m_3)-36$ \\[8pt]
$\langle\C{Z}(G),x\rangle/\langle s\rangle$  & $\B{Q}\colon f(m_2+1,m_2,m_3)$ \\
  \quad $\iso C_{2^{m_2+1}}\times C_{2^{m_3}}\times C_{2^{m_2}}$  & \quad ($N_2=f(m_2+1,m_2,m_3)-f(m_2,m_2,m_3)$) \\
               & $F\colon 2f(m_2+1,m_2,m_3)-36$ \\
                     & \quad ($M_2=2N_2$)\\[8pt]
$\langle\C{Z}(G),y\rangle/\langle s\rangle$ & $\B{Q}\colon f(m_2,m_2,m_3+1)$ \\
  \quad $\iso C_{2^{m_2}}\times C_{2^{m_3+1}}\times C_{2^{m_2}}$  & \quad ($N_3=f(m_2,m_2,m_3+1)-f(m_2,m_2,m_3)$) \\
               & $F\colon 2f(m_2,m_2,m_3+1)-36$ \\
                     & \quad ($M_3=2N_3$) \\[8pt]
$\langle\C{Z}(G),xy\rangle/\langle s\rangle$  & $\B{Q}\colon f(m_2+1,m_2,m_3)$ \\
  \quad $\iso C_{2^{m_2+1}}\times C_{2^{m_3}}\times C_{2^{m_2}}$  & \quad ($N_4=N_2$) \\
                & $F\colon 2f(m_2+1,m_2,m_3)-36$ \\
                  & \quad ($M_4=M_2$)\\[8pt]
$G/G'$  & $N_0=f(m_2+1,m_2,m_3+1)$ \\
 \quad $\iso C_{2^{m_2+1}}\times C_{2^{m_3+1}}\times C_{2^{m_2}}$   & $M_0=2N_0-36$ \\ \hline
\end{tabular}
\end{center}
\bigskip

Since $N=N_1+2N_2+N_3=2f(m_2+1,m_2,m_3)+f(m_2,m_2,m_3+1)-f(m_2,m_2,m_3)
=11(2^{m_2+m_3})-\tfrac23[5(4^{m_3})+4]$ and $M=M_1+2M_2+M_3=2N-72$,
$\B{Q}G$ is the direct sum of
\begin{equation*}
N_0=8(2^{m_2+m_3})-\tfrac43[2(4^{m_3})+1]
\end{equation*}
fields and
\begin{align*}
N-N_0  &= 2f(m_2+1,m_2,m_3)+f(m_2,m_2,m_3+1)-f(m_2,m_2,m_3)  \\
  & \qquad -f(m_2+1,m_2,m_3+1)=3(2^{m_2+m_3})-\tfrac23(4^{m_3}+2)
\end{align*}
quaternion algebras while $FG$ is the direct sum of $M_0=2N_0-36$ fields
and $M-M_0=2(N-N_0)-36$ quaternion algebras.

\newpage
Suppose $m_4=m_2=m_3$.
\bigskip
\begin{center}
\footnotesize
\begin{tabular}{ll}
\multicolumn{1}{c}{Group} & \multicolumn{1}{c}{No. of \B{Q}- and $F$-classes} \\ \hline \\[-6pt]
$\C{Z}(G)$ & $N_1=2f(m_2,m_2,m_2)$ \\
  \quad $\iso C_2\times C_{2^{m_2}}\times C_{2^{m_2}}\times C_{2^{m_2}}$   & $M_1=2N_1-72$ \\[8pt]
$\C{Z}(G)/\langle s\rangle$ & $\B{Q}\colon f(m_2,m_2,m_2)$ \\
 \quad $\iso C_{2^{m_2}}\times C_{2^{m_2}}\times C_{2^{m_2}}$   & $F\colon 2f(m_2,m_2,m_2)-36$ \\[8pt]
$\langle\C{Z}(G),x\rangle/\langle s\rangle$  & $\B{Q}\colon f(m_2+1,m_2,m_2)$ \\
  \quad $\iso C_{2^{m_2+1}}\times C_{2^{m_2}}\times C_{2^{m_2}}$   & \quad ($N_2=f(m_2+1,m_2,m_2)-f(m_2,m_2,m_2)$) \\
               & $F\colon 2f(m_2+1,m_2,m_2)-36$ \\
                     & \quad ($M_2=2N_2$)\\[8pt]
$\langle\C{Z}(G),y\rangle/\langle s\rangle$ & $\B{Q}\colon f(m_2+1,m_2,m_2)$ \\
  \quad $\iso C_{2^{m_2}}\times C_{2^{m_2+1}}\times C_{2^{m_2}}$   & \quad ($N_3=N_2$) \\
               & $F\colon 2f(m_2+1,m_2,m_2)-36$ \\
                     & \quad ($M_3=M_2$) \\[8pt]
$\langle\C{Z}(G),xy\rangle/\langle s\rangle$ & $\B{Q}\colon f(m_2+1,m_2,m_2)$ \\
   \quad $\iso C_{2^{m_2+1}}\times C_{2^{m_2}}\times C_{2^{m_2}}$ & \quad ($N_4=N_2$) \\
                & $F\colon 2f(m_2+1,m_2,m_2)-36$ \\
                  & \quad ($M_4=M_2$)\\[8pt]
$G/G'$ \\
  \quad $\iso C_{2^{m_2+1}}\times C_{2^{m_2+1}}\times C_{2^{m_2}}$ & $N_0=f(m_2+1,m_2+1,m_2)$ \\
               & $M_0=2N_0-36$ \\ \hline
\end{tabular}
\end{center}
\bigskip

Since $N=N_1+3N_2=3f(m_2+1,m_2,m_2)-f(m_2,m_2,m_2)=\tfrac13[23(4^{m_2})-8]$ and
$M=M_1+3M_2=2N-36$,
$\B{Q}G$ is the direct sum of
\begin{equation*}
N_0=\tfrac43(4^{m_2+1}-1)
\end{equation*}
fields and
\begin{align*}
N-N_0  &= 3f(m_2+1,m_2,m_2)-f(m_2,m_2,m_2)-f(m_2+1,m_2+1,m_2) \\
   &= \tfrac13[7(4^{m_2})-4]
\end{align*}
quaternion algebras, while $FG$ is the direct sum of $M_0=2N_0-36$ fields
and $M-M_0=2(N-N_0)-36$ quaternion algebras.

\newpage
We investigate the final case, $m_2\ge m_4\ge m_3$, first supposing $m_2\ge m_4>m_3$.
\bigskip
\begin{center}
\footnotesize
\begin{tabular}{ll}
\multicolumn{1}{c}{Group} & \multicolumn{1}{c}{No. of \B{Q}- and $F$-classes} \\ \hline \\[-6pt]
$\C{Z}(G)$ & $N_1=2f(m_2,m_4,m_3)$ \\
   \quad $\iso C_2\times C_{2^{m_2}}\times C_{2^{m_3}}\times C_{2^{m_4}}$  & $M_1=2N_1-72$ \\[8pt]
$\C{Z}(G)/\langle s\rangle$ & $\B{Q}\colon f(m_2,m_4,m_3)$ \\
  \quad $\iso C_{2^{m_2}}\times C_{2^{m_3}}\times C_{2^{m_4}}$   & $F\colon 2f(m_2,m_4,m_3)-36$ \\[8pt]
$\langle\C{Z}(G),x\rangle/\langle s\rangle$ & $\B{Q}\colon f(m_2+1,m_4,m_3)$ \\
  \quad $\iso C_{2^{m_2+1}}\times C_{2^{m_3}}\times C_{2^{m_4}}$    & \quad ($N_2=f(m_2+1,m_4,m_3)-f(m_2,m_4,m_3)$) \\
               & $F\colon 2f(m_2+1,m_4,m_3)-36$ \\
                     & \quad ($M_2=2N_2$)\\[8pt]
$\langle\C{Z}(G),y\rangle/\langle s\rangle$ b & $\B{Q}\colon f(m_2,m_4,m_3+1)$ \\
  \quad $\iso C_{2^{m_2}}\times C_{2^{m_3+1}}\times C_{2^{m_4}}$   & \quad ($N_3=f(m_2,m_4,m_3+1)-f(m_2,m_4,m_3)$) \\
               & $F\colon 2f(m_2,m_4,m_3+1)-36$ \\
                     & \quad ($M_3=2N_3$) \\[8pt]
$\langle\C{Z}(G),xy\rangle/\langle s\rangle$ & $\B{Q}\colon f(m_2+1,m_4,m_3)$ \\
  \quad $\iso C_{2^{m_2+1}}\times C_{2^{m_3}}\times C_{2^{m_4}}$   & \quad ($N_4=N_2$) \\
                & $F\colon 2f(m_2+1,m_4,m_3)-36$ \\
                  & \quad ($M_4=M_2$)\\[8pt]
$G/G'$ \\
 \quad $\iso C_{2^{m_2+1}}\times C_{2^{m_3+1}}\times C_{2^{m_4}}$ & $N_0=f(m_2+1,m_4,m_3+1)$ \\
               & $M_0=2N_0-36$ \\ \hline
\end{tabular}
\end{center}
\bigskip

Since $N=N_1+2N_2+N_3=2f(m_2+1,m_4,m_3)+f(m_2,m_4,m_3+1)-f(m_2,m_4,m_3)
=2^{m_3+m_4}(11+3m_2-3m_4)-\tfrac23[5(4^{m_3})+4]$ and $M=M_1+2M_2+M_3=2N-72$,
$\B{Q}G$ is the direct sum of
\begin{equation*}
N_0=2^{m_3+m_4}(8+2m_2-2m_4)-\tfrac43[2(4^{m_3})+1]
\end{equation*}
fields and
\begin{align*}
N-N_0  &= 2f(m_2+1,m_4,m_3)+f(m_2,m_4,m_3+1)-f(m_2,m_4,m_3) \\
   & \qquad -f(m_2+1,m_4,m_3+1) \\
   & =2^{m_3+m_4}(3+m_2-m_4)-\tfrac23(4^{m_3}+2)
\end{align*}
quaternion algebras, while $FG$ is the direct sum of $M_0=2N_0-36$ fields
and $M-M_0=2(N-N_0)-36$ quaternion algebras.

\newpage
Now suppose $m_2>m_4=m_3$.
\bigskip
\begin{center}
\footnotesize
\begin{tabular}{ll}
\multicolumn{1}{c}{Group} & \multicolumn{1}{c}{No. of \B{Q}- and $F$-classes} \\ \hline \\[-6pt]
$\C{Z}(G)$  & $N_1=2f(m_2,m_3,m_3)$ \\
   \quad $\iso C_2\times C_{2^{m_2}}\times C_{2^{m_3}}\times C_{2^{m_3}}$ & $M_1=2N_1-72$ \\[8pt]
$\C{Z}(G)/\langle s\rangle \iso C_{2^{m_2}}\times C_{2^{m_3}}\times C_{2^{m_3}}$ & $\B{Q}\colon f(m_2,m_3,m_3)$ \\
               & $F\colon 2f(m_2,m_3,m_3)-36$ \\[8pt]
$\langle\C{Z}(G),x\rangle/\langle s\rangle$ & $\B{Q}\colon f(m_2+1,m_3,m_3)$ \\
  \quad $\iso C_{2^{m_2+1}}\times C_{2^{m_3}}\times C_{2^{m_3}}$    & \quad ($N_2=f(m_2+1,m_3,m_3)-f(m_2,m_3,m_3)$) \\
               & $F\colon 2f(m_2+1,m_3,m_3)-36$ \\
                     & \quad ($M_2=2N_2$)\\[8pt]
$\langle\C{Z}(G),y\rangle/\langle s\rangle$ & $\B{Q}\colon f(m_2,m_3+1,m_3)$ \\
   \quad $\iso C_{2^{m_2}}\times C_{2^{m_3+1}}\times C_{2^{m_3}}$  & \quad ($N_3=f(m_2,m_3+1,m_3)-f(m_2,m_3,m_3)$) \\
               & $F\colon 2f(m_2,m_3+1,m_3)-36$ \\
                     & \quad ($M_3=2N_3$) \\[8pt]
$\langle\C{Z}(G),xy\rangle/\langle s\rangle$ & $\B{Q}\colon f(m_2+1,m_3,m_3)$ \\
  \quad $\iso C_{2^{m_2+1}}\times C_{2^{m_3}}\times C_{2^{m_3}}$   & \quad ($N_4=N_2$) \\
                & $F\colon 2f(m_2+1,m_3,m_3)-36$ \\
                  & \quad ($M_4=M_2$)\\[8pt]
$G/G'$ & $N_0=f(m_2+1,m_3+1,m_3)$ \\
   \quad $\iso C_{2^{m_2+1}}\times C_{2^{m_3+1}}\times C_{2^{m_3}}$  & $M_0=2N_0-36$ \\ \hline
\end{tabular}
\end{center}
\bigskip

Since $N=N_1+2N_2+N_3=2f(m_2+1,m_3,m_3)+f(m_2,m_3+1,m_3)-f(m_2,m_3,m_3)
=4^{m_3}(3m_2-3m_3)+\tfrac13[23(4^{m_3})-8]$ and $M=M_1+2M_2+M_3=2N-72$,
$\B{Q}G$ is the direct sum of
\begin{equation*}
N_0=4^{m_3}(2m_2-2m_3)+\tfrac43(4^{m_3+1}-1)
\end{equation*}
fields and
\begin{align*}
N-N_0  &= 2f(m_2+1,m_3,m_3)+f(m_2,m_3+1,m_3)-f(m_2,m_3,m_3) \\
   & \qquad -f(m_2+1,m_3+1,m_3)=4^{m_3}(m_2-m_3)+\tfrac13[7(4^{m_3})-4]
\end{align*}
quaternion algebras while $FG$ is the direct sum of $M_0=2N_0-36$ fields
and $M-M_0=2(N-N_0)-36$ quaternion algebras.

\medskip\par\noindent (vii) $m_1=2$, $m_2=m_3=m_4=1$, $s=t_1^2$.

\begin{center}
\small
\begin{tabular}{ll}
\multicolumn{1}{c}{Group} & \multicolumn{1}{c}{No. of $K$-classes} \\ \hline \\[-6pt]
$\C{Z}(G)\iso C_4\times C_2\times C_2\times C_2$ & $N_1=24$  \\
$\C{Z}(G)/\langle s\rangle\iso C_2\times C_2\times C_2\times C_2$ & $16$  \\
$\langle\C{Z}(G),x\rangle/\langle s\rangle\iso C_2\times C_4\times C_2\times C_2$ & $24$ ($N_2=8$) \\
$\langle\C{Z}(G),y\rangle/\langle s\rangle\iso C_2\times C_2\times C_4\times C_2$ & $24$ ($N_3=8$)  \\
$\langle\C{Z}(G),xy\rangle/\langle s\rangle\iso C_2\times C_4\times C_2\times C_2$ & $24$ ($N_4=8$) \\
$G/G'\iso C_2\times C_4\times C_4\times C_2$ & $N_0=40$ \\ \hline
\end{tabular}
\end{center}
\medskip

Since $N=N_1+3N_2=48$, $KG$ is the direct sum of $N_0=40$ fields and $N-N_0=8$
quaternion algebras.

\bigskip\par\noindent (viii.a) $m_1=2$, $m_2=m_3=1$, $m_4=2$, $s=t_1^2$.
\begin{center}
\small
\begin{tabular}{ll}
\multicolumn{1}{c}{Group} & \multicolumn{1}{c}{No. of $K$-classes} \\ \hline \\[-6pt]
$\C{Z}(G)\iso C_4\times C_2\times C_2\times C_4$ & $N_1=40$  \\
$\C{Z}(G)/\langle s\rangle\iso C_2\times C_2\times C_2\times C_4$ & $24$  \\
$\langle\C{Z}(G),x\rangle/\langle s\rangle\iso C_2\times C_4\times C_2\times C_4$ & $40$ ($N_2=16$)  \\
$\langle\C{Z}(G),y\rangle/\langle s\rangle\iso C_2\times C_2\times C_4\times C_4$ & $40$ ($N_3=16$)  \\
$\langle\C{Z}(G),xy\rangle/\langle s\rangle\iso C_2\times C_2\times C_4\times C_4$ & $40$ ($N_4=16$)  \\
$G/G'\iso C_2\times C_4\times C_4\times C_4$ & $N_0=72$ \\ \hline
\end{tabular}
\end{center}
\medskip

Since $N=N_1+3N_2=88$, $KG$ is the direct sum of $N_0=72$ fields and $N-N_0=16$
quaternion algebras.

\bigskip\par\noindent (viii.b) $m_1=2$, $m_2=m_3=1$, $m_4\ge3$, $s=t_1^2$.
\begin{center}
\small
\begin{tabular}{lll}
\multicolumn{1}{c}{Group} & \multicolumn{1}{c}{No. of \B{Q}-classes} & \multicolumn{1}{c}{No. of $F$-classes} \\ \hline \\[-6pt]
$\C{Z}(G)$ \\
 \quad $\iso C_4\times C_2\times C_2\times C_{2^{m_4}}$ & $N_1=16m_4+8$ & $M_1=32m_4-24$ \\
$\C{Z}(G)/\langle s\rangle$ \\
 \quad $\iso C_2\times C_2\times C_2\times C_{2^{m_4}}$ & $8m_4+8$ & $16m_4-8$ \\
$\langle\C{Z}(G),x\rangle/\langle s\rangle$ \\
 \quad $\iso C_2\times C_4\times C_2\times C_{2^{m_4}}$ & $16m_4+8$ ($N_2=8m_4$)& $32m_4-24$ ($M_2=16m_4-16$)\\
$\langle\C{Z}(G),y\rangle/\langle s\rangle$ \\
 \quad $\iso C_2\times C_2\times C_4\times C_{2^{m_4}}$ & $16m_4+8$ ($N_3=8m_4$) & $32m_4-24$ ($M_3=16m_4-16$) \\
$\langle\C{Z}(G),xy\rangle/\langle s\rangle$ \\
 \quad $\iso C_2\times C_2\times C_4\times C_{2^{m_4}}$ & $16m_4+8$ ($N_4=8m_4$) & $32m_4-24$ ($M_4=16m_4-16$)\\
$G/G'$ \\
 \quad $\iso C_2\times C_4\times C_4\times C_{2^{m_4}}$ & $N_0=32m_4+8$ & $M_0=64m_4-56$ \\ \hline
\end{tabular}
\end{center}
\medskip

Since $N=N_1+3N_2=40m_4+8$ and $M=M_1+3M_2=80m_4-72$, $\B{Q}G$ is
the direct sum of $N_0=32m_4+8$ fields and $N-N_0=8m_4$ quaternion algebras while
$FG$ is the direct sum of $M_0=64m_4-56$ fields and $M-M_0=16m_4-16$ quaternion algebras.

\bigskip\par\noindent (ix.a) $m_1=2$, $m_2=1$, $m_3=2$, $m_4=1$, $s=t_1^2$.  (The conclusions in
the case $2,2,1,1$ are the same.)
\begin{center}
\small
\begin{tabular}{lll}
\multicolumn{1}{c}{Group} & \multicolumn{1}{c}{No. of \B{Q}-classes} & \multicolumn{1}{c}{No. of $F$-classes} \\ \hline \\[-6pt]
$\C{Z}(G)\iso C_4\times C_2\times C_4\times C_2$ & $N_1=40$ & $M_1=40$ \\
$\C{Z}(G)/\langle s\rangle\iso C_2\times C_2\times C_4\times C_2$ & $24$ & $24$ \\
$\langle\C{Z}(G),x\rangle/\langle s\rangle\iso C_2\times C_4\times C_4\times C_2$ & $40$ ($N_2=16$)& $40$ ($M_2=16$)\\
$\langle\C{Z}(G),y\rangle/\langle s\rangle\iso C_2\times C_2\times C_8\times C_2$ & $32$ ($N_3=8$) & $40$ ($M_3=16$) \\
$\langle\C{Z}(G),xy\rangle/\langle s\rangle\iso C_2\times C_2\times C_8\times C_2$ & $32$ ($N_4=8$) & $40$ ($M_4=16$)\\
$G/G'\iso C_2\times C_4\times C_8\times C_2$ & $N_0=56$ & $M_0=72$ \\ \hline
\end{tabular}
\end{center}
\medskip

Since $N=N_1+N_2+2N_3=72$ and $M=M_1+3M_2=88$, $\B{Q}G$ is the
direct sum of $N_0=56$ fields and $N-N_0=16$ quaternion
algebras while $FG$ is the direct sum of $M_0=72$ fields and
$M-M_0=16$ quaternion algebras.

\medskip\par (ix.b)\noindent $m_1=2$, $m_2=1$, $m_3\ge3$, $m_4=1$, $s=t_1^2$.  (The conclusions
in the case $2,m_2\ge3,1,1$ are the same.)
\bigskip
\begin{center}
\small
\begin{tabular}{ll}
\multicolumn{1}{c}{Group} & \multicolumn{1}{c}{No. of \B{Q}- and $F$-classes} \\ \hline \\[-6pt]
$\C{Z}(G)\iso C_4\times C_2\times C_{2^{m_3}}\times C_2$ & $N_1=16m_3+8$ \\
                                                         & $M_1=32m_3-24$ \\[4pt]
$\C{Z}(G)/\langle s\rangle\iso C_2\times C_2\times C_{2^{m_3}}\times C_2$ & \B{Q}: $8m_3+8$ \\
                                                              & $F$: $16m_3-8$ \\[4pt]
$\langle\C{Z}(G),x\rangle/\langle s\rangle\iso C_2\times C_4\times C_{2^{m_3}}\times C_2$ & \B{Q}: $16m_3+8$ ($N_2=8m_3$) \\
                   & $F$: $32m_3-24$ ($M_2=16m_3-16$)\\[4pt]
$\langle\C{Z}(G),y\rangle/\langle s\rangle\iso C_2\times C_2\times C_{2^{m_3+1}}\times C_2$ & \B{Q}: $8m_3+16$ ($N_3=8$) \\
                       & $F$: $16m_3+8$ ($M_3=16$) \\[4pt]
$\langle\C{Z}(G),xy\rangle/\langle s\rangle\iso C_2\times C_2\times C_{2^{m_3+1}}\times C_2$ & \B{Q}: $8m_3+16$ ($N_4=8$) \\
               & $F$: $16m_3+8$ ($M_4=16$)\\[4pt]
$G/G'\iso C_2\times C_4\times C_{2^{m_3+1}}\times C_2$ & $N_0=16m_3+24$ \\
                                             & $M_0=32m_3+8$ \\ \hline
\end{tabular}
\end{center}
\bigskip

Since $N=N_1+N_2+2N_3=24m_3+24$ and $M=M_1+M_2+2M_3=48m_3-8$,
$\B{Q}G$ is the direct sum of $N_0=16m_3+24$ fields and
$N-N_0=8m_3$ quaternion algebras while $FG$ is the direct sum
of $M_0=32m_3+8$ fields and $M-M_0=16m_3-16$ quaternion
algebras.

\bigskip\par\noindent (x.a) $m_1=2$, $m_2=1$, $m_3=m_4=2$, $s=t_1^2$. (The conclusions in
the case $2,2,1,2$ are the same.)
\bigskip
\begin{center}
\small
\begin{tabular}{ccc}
Group & No. of \B{Q}-classes & No. of $F$-classes \\ \hline \\[-6pt]
$\C{Z}(G)\iso C_4\times C_2\times C_4\times C_4$ & $N_1=72$ & $M_1=72$ \\
$\C{Z}(G)/\langle s\rangle\iso C_2\times C_2\times C_4\times C_4$ & $40$ & $40$ \\
$\langle\C{Z}(G),x\rangle/\langle s\rangle\iso C_2\times C_4\times C_4\times C_4$ & $72$ ($N_2=32$)& $72$ ($M_2=32$)\\
$\langle\C{Z}(G),y\rangle/\langle s\rangle\iso C_2\times C_2\times C_8\times C_4$ & $56$ ($N_3=16$) & $72$ ($M_3=32$) \\
$\langle\C{Z}(G),xy\rangle/\langle s\rangle\iso C_2\times C_2\times C_8\times C_4$ & $56$ ($N_4=16$) & $72$ ($M_4=32$)\\
$G/G'\iso C_2\times C_4\times C_8\times C_4$ & $N_0=104$ & $M_0=136$ \\ \hline
\end{tabular}
\end{center}
\bigskip

Since $N=N_1+N_2+2N_3=136$ and $M=M_1+3M_2=168$,
$\B{Q}G$ is the direct sum of $N_0=104$ fields and $N-N_0=32$
quaternion algebras while $FG$ is the direct sum of $M_0=136$ fields
and $M-M_0=32$ quaternion algebras.

\newpage
\smallskip\par\noindent (x.b) $m_1=2$, $m_2=1$, $m_3\ge 3$, $m_4=2$, $s=t_1^2$. (The conclusions
in the case $2,m_2\ge3,1,2$ are the same.)
\bigskip
\begin{center}
\small
\begin{tabular}{ll}
\multicolumn{1}{c}{Group} & \multicolumn{1}{c}{No. of \B{Q}- and $F$-classes} \\ \hline \\[-6pt]
$\C{Z}(G)\iso C_4\times C_2\times C_{2^{m_3}}\times C_4$ & \B{Q}: $N_1=32m_3+8$ \\
                          & $F$: $M_1=64m_3-56$ \\[4pt]
$\C{Z}(G)/\langle s\rangle\iso C_2\times C_2\times C_{2^{m_3}}\times C_4$ & \B{Q}: $16m_3+8$ \\
                  & $F$: $32m_3-24$ \\[4pt]
$\langle\C{Z}(G),x\rangle/\langle s\rangle\iso C_2\times C_4\times C_{2^{m_3}}\times C_4$ & $\B{Q}\colon 32m_3+8$ ($N_2=16m_3$) \\
              & $F\colon 64m_3-56$ ($M_2=32m_3-32$)\\[4pt]
$\langle\C{Z}(G),y\rangle/\langle s\rangle\iso C_2\times C_2\times C_{2^{m_3+1}}\times C_4$ & $\B{Q}\colon 16m_3+24$ ($N_3=16$) \\
             & $F\colon 32m_3+8$ ($M_3=32$) \\[4pt]
$\langle\C{Z}(G),xy\rangle/\langle s\rangle\iso C_2\times C_2\times C_{2^{m_3+1}}\times C_4$ & $\B{Q}\colon 16m_3+24$ ($N_4=16$) \\
              & $F\colon 32m_3+8$ ($M_4=32$)\\[4pt]
$G/G'\iso C_2\times C_4\times C_{2^{m_3+1}}\times C_4$ & $N_0=32m_3+40$ \\
                                             & $M_0=64m_3+8$ \\ \hline
\end{tabular}
\end{center}
\bigskip

Since $N=N_1+N_2+2N_3=48m_3+40$ and $M=M_1+M_2+2M_3=96m_3-24$,
$\B{Q}G$ is the direct sum of $N_0=32m_3+40$ fields and $N-N_0=16m_3$
quaternion algebras while $FG$ is the direct sum of $M_0=64m_3+8$ fields
and $M-M_0=32m_3-32$ quaternion algebras.

\bigskip\par\noindent (x.c) $m_1=2$, $m_2=1$, $m_3=2$, $m_4\ge3$, $s=t_1^2$. (The conclusions
in the case $2,2,1,m_4\ge3$ are the same.)
\begin{center}
\small
\begin{tabular}{ll}
\multicolumn{1}{c}{Group} & \multicolumn{1}{c}{No. of \B{Q}- and $F$-classes} \\ \hline \\[-6pt]
$\C{Z}(G)\iso C_4\times C_2\times C_4\times C_{2^{m_4}}$ & $N_1=32m_4+8$ \\
                                & $M_1=64m_4-56$ \\[4pt]
$\C{Z}(G)/\langle s\rangle\iso C_2\times C_2\times C_4\times C_{2^{m_4}}$ & $\B{Q}\colon 16m_4+8$ \\
                                 & $F\colon 32m_4-24$ \\[4pt]
$\langle\C{Z}(G),x\rangle/\langle s\rangle\iso C_2\times C_4\times C_4\times C_{2^{m_4}}$ & $\B{Q}\colon 32m_4+8$ ($N_2=16m_4$) \\
              & $F\colon 64m_4-56$ ($M_2=32m_4-32$)\\[4pt]
$\langle\C{Z}(G),y\rangle/\langle s\rangle\iso C_2\times C_2\times C_8\times C_{2^{m_4}}$ & $\B{Q}\colon 32m_4-8$ ($N_3=16m_4-16$) \\
                 & $F\colon 64m_4-56$ ($M_3=32m_4-32$) \\[4pt]
$\langle\C{Z}(G),xy\rangle/\langle s\rangle\iso C_2\times C_2\times C_8\times C_{2^{m_4}}$ & $\B{Q}\colon 32m_4-8$ ($N_4=16m_4-16$) \\
                 & $F\colon 64m_4-56$ ($M_4=32m_4-32$)\\[4pt]
$G/G'\iso C_2\times C_4\times C_8\times C_{2^{m_4}}$ & $N_0=64m_4-24$ \\
                                           & $M_0=128m_4-120$ \\ \hline
\end{tabular}
\end{center}
\medskip

Since $N=N_1+N_2+2N_3=80m_4-24$ and
$M=M_1+M_2+2M_3=160m_4-152$, $\B{Q}G$ is the direct sum of
$N_0=64m_4-24$ fields and $N-N_0=16m_4$ quaternion algebras
while $FG$ is the direct sum of $M_0=128m_4-120$ fields and
$M-M_0=32m_4-32$ quaternion algebras.

\newpage
\smallskip\par\noindent (x.d) $m_1=2$, $m_2=1$, $m_3\ge3$, $m_4\ge3$, $s=t_1^2$. (The conclusions
in the case $2,m_2\ge3,1,m_4\ge3$ are the same.)
\medskip

Assume first that $m_3=m_4$.
\begin{center}
\small
\begin{tabular}{lll}
\multicolumn{1}{c}{Group} & \multicolumn{1}{c}{No. of \B{Q}-classes} & \multicolumn{1}{c}{No. of $F$-classes} \\ \hline \\[-6pt]
$\C{Z}(G)$  & $N_1=2^{m_3+3}(3)-24$    & $M_1=2^{m_3+4}(3)-120$ \\
 \quad $\iso C_4\times C_2\times C_{2^{m_3}}\times C_{2^{m_3}}$   \\[4pt]
$\C{Z}(G)/\langle s\rangle$ & $2^{m_3+2}(3)-8$  & $2^{m_3+3}(3)-56$ \\
  \quad $\iso C_2\times C_2\times C_{2^{m_3}}\times C_{2^{m_3}}$     \\[4pt]
$\langle\C{Z}(G),x\rangle/\langle s\rangle$ & $2^{m_3+3}(3)-24$ & $2^{m_3+4}(3)-120$ \\
 \qquad $\iso C_2\times C_4\times C_{2^{m_3}}\times C_{2^{m_3}}$ & ($N_2=2^{m_3+2}(3)-16$)  & ($M_2=2^{m_3+3}(3)-64$)\\[4pt]
$\langle\C{Z}(G),y\rangle/\langle s\rangle$ & $2^{m_3+2}(4)-8$ & $2^{m_4+3}(4)-56$ \\
 \qquad $\iso C_2\times C_2\times C_{2^{m_3+1}}\times C_{2^{m_3}}$  &  ($N_3=2^{m_3+2}$)  & ($M_3=2^{m_3+3}$) \\[4pt]
$\langle\C{Z}(G),xy\rangle/\langle s\rangle$ & $2^{m_3+2}(4)-8$  & $2^{m_4+3}(4)-56$ \\
 \qquad  $\iso C_2\times C_2\times C_{2^{m_3+1}}\times C_{2^{m_3}}$ &  ($N_4=2^{m_3+2}$) & ($M_4=2^{m_4+2}$)\\[4pt]
$G/G'$ & $N_0=2^{m_3+3}(4)-24$  & $M_0=2^{m_3+4}(4)-120$ \\
 \quad $\iso C_2\times C_4\times C_{2^{m_3+1}}\times C_{2^{m_3}}$  \\ \hline
\end{tabular}
\end{center}
\medskip

Since $N=N_1+N_2+2N_3=2^{m_3+2}(11)-40$ and $M=M_1+M_2+2M_3=2^{m_3+3}(11)-184$,
$\B{Q}G$ is the direct sum of $N_0=2^{m_3+5}-24$ fields and $N-N_0=2^{m_3+2}(3)-16$
quaternion algebras while $FG$ is the direct sum of $M_0=2^{m_3+6}-120$ fields
and $M-M_0=2^{m_3+3}(3)-64$ quaternion algebras.

\medskip

Now assume $m_3>m_4$.
\begin{center}
\small
\begin{tabular}{ll}
\multicolumn{1}{c}{Group} & \multicolumn{1}{c}{No. of \B{Q}- and $F$-classes} \\ \hline \\[-6pt]
$\C{Z}(G)\iso C_4\times C_2\times C_{2^{m_3}}\times C_{2^{m_4}}$ & $N_1=2^{m_4+3}(3+m_3-m_4)-24$ \\
                            & $M_1=2^{m_4+4}(3+m_3-m_4)-120$ \\[4pt]
$\C{Z}(G)/\langle s\rangle\iso C_2\times C_2\times C_{2^{m_3}}\times C_{2^{m_4}}$ & $\B{Q}\colon 2^{m_4+2}(3+m_3-m_4)-8$ \\
                  & $F\colon 2^{m_4+3}(3+m_3-m_4)-56$ \\[4pt]
$\langle\C{Z}(G),x\rangle/\langle s\rangle$  & $\B{Q}\colon 2^{m_4+3}(3+m_3-m_4)-24$ \\
  \quad $\iso C_2\times C_4\times C_{2^{m_3}}\times C_{2^{m_4}}$   & \qquad ($N_2=2^{m_4+2}(3+m_3-m_4)-16$) \\
             & $F\colon 2^{m_4+4}(3+m_3-m_4)-120$ \\
                       & \qquad ($M_2=2^{m_4+3}(3+m_3-m_4)-64$)\\[4pt]
$\langle\C{Z}(G),y\rangle/\langle s\rangle$    & $\B{Q}\colon 2^{m_4+2}(4+m_3-m_4)-8$ \\
 \quad $\iso C_2\times C_2\times C_{2^{m_3+1}}\times C_{2^{m_4}}$  & \qquad ($N_3=2^{m_4+2}$) \\
                  & $F\colon 2^{m_4+3}(4+m_3-m_4)-56$ \\
                             & \qquad ($M_3=2^{m_4+3}$) \\[4pt]
$\langle\C{Z}(G),xy\rangle/\langle s\rangle$ & $\B{Q}\colon 2^{m_4+2}(4+m_3-m_4)-8$ \\
  \quad $\iso C_2\times C_2\times C_{2^{m_3+1}}\times C_{2^{m_4}}$ & \qquad ($N_4=2^{m_4+2}$) \\
                  & $F\colon 2^{m_4+3}(4+m_3-m_4)-56$ \\
                             & \qquad ($M_4=2^{m_4+3}$)\\[4pt]
$G/G'\iso C_2\times C_4\times C_{2^{m_3+1}}\times C_{2^{m_4}}$ & $N_0=2^{m_4+3}(4+m_3-m_4)-24$ \\
                                                     & $M_0=2^{m_4+4}(4+m_3-m_4)-120$ \\ \hline
\end{tabular}
\end{center}
\medskip

Since $N=N_1+N_2+2N_3=2^{m_4+2}(11+3m_3-3m_4)-40$ and $M=M_1+M_2+2M_3=2^{m_4+3}(11+3m_3-3m_4)-184$,
$\B{Q}G$ is the direct sum of $N_0=2^{m_4+3}(4+m_3-m_4)-24$ fields and $N-N_0=2^{m_4+2}(3+m_3-m_4)-16$
quaternion algebras while $FG$ is the direct sum of $M_0=2^{m_4+4}(4+m_3-m_4)-120$ fields
and $M-M_0=2^{m_4+3}(3+m_3-m_4)-64$ quaternion algebras.

\bigskip
Finally, assume that $m_3<m_4$.

\begin{center}
\small
\begin{tabular}{ll}
\multicolumn{1}{c}{Group} & \multicolumn{1}{c}{No. of \B{Q}- and $F$-classes} \\ \hline \\[-6pt]
$\C{Z}(G)\iso C_4\times C_2\times C_{2^{m_3}}\times C_{2^{m_4}}$ & $N_1=2^{m_3+3}(3+m_4-m_3)-24$ \\
                          & $M_1=2^{m_3+4}(3+m_4-m_3)-120$ \\[4pt]
$\C{Z}(G)/\langle s\rangle\iso C_2\times C_2\times C_{2^{m_3}}\times C_{2^{m_4}}$ & $\B{Q}\colon 2^{m_3+2}(3+m_4-m_3)-8$ \\
                          & $F\colon 2^{m_3+3}(3+m_4-m_3)-56$ \\[4pt]
$\langle\C{Z}(G),x\rangle/\langle s\rangle$   & $\B{Q}\colon 2^{m_3+3}(3+m_4-m_3)-24$ \\
   \quad $\iso C_2\times C_4\times C_{2^{m_3}}\times C_{2^{m_4}}$ & \qquad ($N_2=2^{m_3+2}(3+m_4-m_3)-16$) \\
                          & $F\colon 2^{m_3+4}(3+m_4-m_3)-120$ \\
                                  & \qquad ($M_2=2^{m_3+3}(3+m_4-m_3)-64$)\\[4pt]
$\langle\C{Z}(G),y\rangle/\langle s\rangle$   & $\B{Q}\colon 2^{m_3+3}(2+m_4-m_3)-8$ \\
  \quad $\iso C_2\times C_2\times C_{2^{m_3+1}}\times C_{2^{m_4}}$ & \qquad ($N_3=2^{m_3+2}(1+m_4-m_3)$) \\
                          & $F\colon 2^{m_3+4}(2+m_4-m_3)-56$ \\
                                  & \qquad ($M_3=2^{m_3+3}(1+m_4-m_3)$) \\[4pt]
$\langle\C{Z}(G),xy\rangle/\langle s\rangle$ & $\B{Q}\colon 2^{m_3+3}(2+m_4-m_3)-8$ \\
 \quad $\iso C_2\times C_2\times C_{2^{m_3+1}}\times C_{2^{m_4}}$ & \qquad ($N_4=2^{m_3+2}(1+m_4-m_3)$) \\
                          & $F\colon 2^{m_3+4}(2+m_4-m_3)-56$ \\
                                  & \qquad ($M_4=2^{m_3+3}(1+m_4-m_3)$)\\[4pt]
$G/G'\iso C_2\times C_4\times C_{2^{m_3+1}}\times C_{2^{m_4}}$ & $N_0=2^{m_3+4}(2+m_4-m_3)-24$ \\
                                                     & $M_0=2^{m_3+5}(2+m_4-m_3)-120$ \\ \hline
\end{tabular}
\end{center}
\medskip

Since $N=N_1+N_2+2N_3=2^{m_3+2}(11+5m_4-5m_3)-40$ and $M=M_1+M_2+2M_3=2^{m_3+3}(11+5m_4-5m_3)-184$,
$\B{Q}G$ is the direct sum of $N_0=2^{m_3+4}(2+m_4-m_3)-24$ fields and $N-N_0=2^{m_3+2}(3+m_4-m_3)-64$
quaternion algebras while $FG$ is the direct sum of $M_0=2^{m_3+5}(2+m_4-m_3)-120$ fields
and $M-M_0=2^{m_3+3}(3+m_4-m_3)-64$ quaternion algebras.

\bigskip\par\noindent (xi.a) $m_1=2$, $m_2=m_3=2$, $m_4=1$, $s=t_1^2$.
\begin{center}
\small
\begin{tabular}{ccc}
Group & No. of \B{Q}-classes & No. of $F$-classes \\ \hline \\[-6pt]
$\C{Z}(G)\iso C_4\times C_4\times C_4\times C_2$ & $N_1=72$ & $M_1=72$ \\
$\C{Z}(G)/\langle s\rangle\iso C_2\times C_4\times C_4\times C_2$ & $40$ & $40$ \\
$\langle\C{Z}(G),x\rangle/\langle s\rangle\iso C_2\times C_8\times C_4\times C_2$ & $56$ ($N_2=16$)& $72$ ($M_2=32$)\\
$\langle\C{Z}(G),y\rangle/\langle s\rangle\iso C_2\times C_4\times C_8\times C_2$ & $56$ ($N_3=16$) & $72$ ($M_3=32$) \\
$\langle\C{Z}(G),xy\rangle/\langle s\rangle\iso C_2\times C_4\times C_8\times C_2$ & $56$ ($N_4=16$) & $72$ ($M_4=32$)\\
$G/G'\iso C_2\times C_8\times C_8\times C_2$ & $N_0=88$ & $M_0=136$ \\ \hline
\end{tabular}
\end{center}
\medskip

Since $N=N_1+3N_2=120$ and $M=M_1+3M_2=168$, $\B{Q}G$ is the direct sum of $N_0=88$ fields and $N-N_0=32$
quaternion algebras while $FG$ is the direct sum of $M_0=136$ fields
and $M-M_0=32$ quaternion algebras.

\newpage
\par\noindent (xi.b) $m_1=2$, $m_2\ge3$, $m_3=2$, $m_4=1$, $s=t_1^2$. (The conclusions
in the case $2,2,m_3\ge3,1$ are the same.)
\begin{center}
\small
\begin{tabular}{ll}
\multicolumn{1}{c}{Group} & \multicolumn{1}{c}{No. of \B{Q}- and $F$-classes} \\ \hline \\[-6pt]
$\C{Z}(G)\iso C_4\times C_{2^{m_2}}\times C_4\times C_2$ & $N_1=32m_2+8$ \\
               & $M_1=64m_2-56$ \\[4pt]
$\C{Z}(G)/\langle s\rangle\iso C_2\times C_{2^{m_2}}\times C_4\times C_2$ & $\B{Q}\colon 16m_2+8$ \\
               & $F\colon 32m_2-24$ \\[4pt]
$\langle\C{Z}(G),x\rangle/\langle s\rangle\iso C_2\times C_{2^{m_2+1}}\times C_4\times C_2$ & $\B{Q}\colon 16m_2+24$ ($N_2=16$) \\
               & $F\colon 32m_2+8$ ($M_2=32$)\\[4pt]
$\langle\C{Z}(G),y\rangle/\langle s\rangle\iso C_2\times C_{2^{m_2}}\times C_8\times C_2$ & $\B{Q}\colon 32m_2-8$ ($N_3=16m_2-16$) \\
               & $F\colon 64m_2-56$ ($M_3=32m_2-32$) \\[4pt]
$\langle\C{Z}(G),xy\rangle/\langle s\rangle\iso C_2\times C_{2^{m_2+1}}\times C_4\times C_2$ & $\B{Q}\colon 16m_2+24$ ($N_4=16$) \\
               & $F\colon 32m_2+8$ ($M_4=32$)\\[4pt]
$G/G'\iso C_2\times C_{2^{m_2+1}}\times C_8\times C_2$ & $N_0=32m_2+24$ \\
               & $M_0=64m_2+8$ \\ \hline
\end{tabular}
\end{center}
\medskip

Since $N=N_1+2N_2+N_3=48m_2+24$ and $M=M_1+2M_2+M_3=96m_2-24$,
$\B{Q}G$ is the direct sum of $N_0=32m_2+24$ fields and
$N-N_0=16m_2$ quaternion algebras while $FG$ is the direct sum
of $M_0=64m_2+8$ fields and $M-M_0=32m_2-32$ quaternion
algebras.

\bigskip\par\noindent (xi.c) $m_1=2$, $m_2\ge3$, $m_3\ge3$, $m_4=1$, $s=t_1^2$.
By symmetry in $x$ and $y$, we may assume that $m_2\ge m_3$.  We deal with the case $m_2>m_3$ first.
\begin{center}
\small
\begin{tabular}{ll}
\multicolumn{1}{c}{Group} & \multicolumn{1}{c}{No. of \B{Q}- and $F$-classes} \\ \hline \\[-6pt]
$\C{Z}(G)\iso C_4\times C_{2^{m_2}}\times C_{2^{m_3}}\times C_2$ & $N_1=2^{m_3+3}(3+m_2-m_3)-24$ \\
                                            & $M_1=2^{m_3+4}(3+m_2-m_3)-120$ \\[4pt]
$\C{Z}(G)/\langle s\rangle\iso C_2\times C_{2^{m_2}}\times C_{2^{m_3}}\times C_2$ & $\B{Q}\colon 2^{m_3+2}(3+m_2-m_3)-8$ \\
                                            & $F\colon 2^{m_3+3}(3+m_2-m_3)-56$ \\[4pt]
$\langle\C{Z}(G),x\rangle/\langle s\rangle\iso C_2\times C_{2^{m_2+1}}\times C_{2^{m_3}}\times C_2$ & $\B{Q}\colon 2^{m_3+2}(4+m_2-m_3)-8$ \\
                                                       &\qquad ($N_2=2^{m_3+2}$) \\
                                            & $F\colon 2^{m_3+3}(4+m_2-m_3)-56$ \\
                                                       & \qquad ($M_2=2^{m_3+3}$)\\[4pt]
$\langle\C{Z}(G),y\rangle/\langle s\rangle\iso C_2\times C_{2^{m_2}}\times C_{2^{m_3+1}}\times C_2$ & $\B{Q}\colon 2^{m_3+3}(2+m_2-m_3)-8$ \\
                                                       & \qquad ($N_3=2^{m_3+2}(1+m_2-m_3)$) \\
                                            & $F\colon 2^{m_3+4}(2+m_2-m_3)-56$ \\
                                                       & \qquad ($M_3=2^{m_3+3}(1+m_2-m_3)$) \\[4pt]
$\langle\C{Z}(G),xy\rangle/\langle s\rangle\iso C_2\times C_{2^{m_2+1}}\times C_{2^{m_3}}\times C_2$ & $\B{Q}\colon 2^{m_3+2}(4+m_2-m_3)-8$ \\
                                                       & \qquad ($N_4=2^{m_3+2}$) \\
                                            & $F\colon 2^{m_3+3}(4+m_2-m_3)-56$ \\
                                                       & \qquad ($M_4=2^{m_3+3}$)\\[4pt]
$G/G'\iso C_2\times C_{2^{m_2+1}}\times C_{2^{m_3+1}}\times C_2$ & $N_0=2^{m_3+3}(3+m_2-m_3)-8$ \\
                                            & $M_0=2^{m_3+4}(3+m_2-m_3)-56$ \\ \hline
\end{tabular}
\end{center}
\medskip

Since $N=N_1+2N_2+N_3=2^{m_3+2}(9+3m_2-3m_3)-24$ and
$M=M_1+2M_2+M_3=2^{m_3+3}(9+3m_2-3m_3)-120$, $\B{Q}G$ is the
direct sum of $N_0=2^{m_3+3}(3+m_2-m_3)-8$ fields and
$N-N_0=2^{m_3+2}(3+m_2-m_3)-16$ quaternion algebras while $FG$
is the direct sum of $M_0=2^{m_3+4}(3+m_2-m_3)-56$ fields and
$M-M_0=2^{m_3+3}(3+m_2-m_3)-64$ quaternion algebras.

\medskip
When $m_2=m_3$, here's what happens.
\bigskip
\begin{center}
\small
\begin{tabular}{ll}
\multicolumn{1}{c}{Group} & \multicolumn{1}{c}{No. of \B{Q}- and $F$-classes} \\ \hline \\[-6pt]
$\C{Z}(G)\iso C_4\times C_{2^{m_2}}\times C_{2^{m_2}}\times C_2$ & $N_1=2^{m_2+3}(3)-24$ \\
                                                                 & $M_1=2^{m_2+4}(3)-120$ \\[4pt]
$\C{Z}(G)/\langle s\rangle\iso C_2\times C_{2^{m_2}}\times C_{2^{m_2}}\times C_2$ & $\B{Q}\colon 2^{m_2+2}(3)-8$ \\
                            & $F\colon 2^{m_2+3}(3)-56$ \\[4pt]
$\langle\C{Z}(G),x\rangle/\langle s\rangle\iso C_2\times C_{2^{m_2+1}}\times C_{2^{m_2}}\times C_2$ & $\B{Q}\colon 2^{m_2+2}(4)-8$ ($N_2=2^{m_2+2}$) \\
                        & $F\colon 2^{m_2+3}(4)-56$ ($M_2=2^{m_2+3}$)\\[4pt]
$\langle\C{Z}(G),y\rangle/\langle s\rangle\iso C_2\times C_{2^{m_2}}\times C_{2^{m_2+1}}\times C_2$ & $\B{Q}\colon 2^{m_2+2}(4)-4$ ($N_3=2^{m_2+2}$) \\
                        & $F\colon 2^{m_2+3}(4)-56$ ($M_3=2^{m_2+3}$) \\[4pt]
$\langle\C{Z}(G),xy\rangle/\langle s\rangle\iso C_2\times C_{2^{m_2}}\times C_{2^{m_2+1}}\times C_2$ & $\B{Q}\colon 2^{m_2+3}(4)-4$ ($N_4=2^{m_2+2}$) \\
                        & $F\colon 2^{m_2+3}(4)-56$ ($M_4=2^{m_2+3}$)\\[4pt]
$G/G'\iso C_2\times C_{2^{m_2+1}}\times C_{2^{m_2+1}}\times C_2$ & $N_0=2^{m_2+3}(3)-8$ \\
                        & $M_0=2^{m_2+4}(3)-56$ \\ \hline
\end{tabular}
\end{center}
\bigskip

Since $N=N_1+3N_2=2^{m_2+2}(9)-24$ and
$M=M_1+3M_2=2^{m_2+3}(9)-120$, $\B{Q}G$ is the direct sum of
$N_0=2^{m_2+3}(3)-8$ fields and $N-N_0=2^{m_2+2}(3)-16$
quaternion algebras while $FG$ is the direct sum of
$M_0=2^{m_2+4}(3)-56$ fields and $M-M_0=2^{m_2+3}(3)-64$
quaternion algebras.

\bigskip\par\noindent (xii.a) $m_1=2$, $m_2=m_3=m_4=2$, $s=t_1^2$.
\bigskip
\begin{center}
\small
\begin{tabular}{ccc}
Group & No. of \B{Q}-classes & No. of $F$-classes \\ \hline \\[-6pt]
$\C{Z}(G)\iso C_4\times C_4\times C_4\times C_4$ & $N_1=136$ & $M_1=136$ \\
$\C{Z}(G)/\langle s\rangle\iso C_2\times C_4\times C_4\times C_4$ & $72$ & $72$ \\
$\langle\C{Z}(G),x\rangle/\langle s\rangle\iso C_2\times C_8\times C_4\times C_4$ & $104$ ($N_2=32$)& $136$ ($M_2=64$)\\
$\langle\C{Z}(G),y\rangle/\langle s\rangle\iso C_2\times C_4\times C_8\times C_4$ & $104$ ($N_3=32$) & $136$ ($M_3=64$) \\
$\langle\C{Z}(G),xy\rangle/\langle s\rangle\iso C_2\times C_4\times C_8\times C_4$ & $104$ ($N_4=32$) & $136$ ($M_4=64$)\\
$G/G'\iso C_2\times C_8\times C_8\times C_4$ & $N_0=168$ & $M_0=264$ \\ \hline
\end{tabular}
\end{center}
\medskip

Since $N=N_1+3N_2=232$ and $M=M_1+3M_2=328$, $\B{Q}G$ is the direct sum of $N_0=168$ fields and $N-N_0=64$
quaternion algebras while $FG$ is the direct sum of $M_0=264$ fields
and $M-M_0=64$ quaternion algebras.

\newpage
\par\noindent (xii.b) $m_1=2$, $m_2\ge3$, $m_3=2$, $m_4=2$, $s=t_1^2$.  (The case $2,2,m_3\ge3,2$
will give identical results.)
\bigskip
\begin{center}
\small
\begin{tabular}{ll}
\multicolumn{1}{c}{Group} & \multicolumn{1}{c}{No. of \B{Q}- and $F$-classes} \\ \hline \\[-6pt]
$\C{Z}(G)\iso C_4\times C_{2^{m_2}}\times C_4\times C_4$ & $N_1=64m_2+8$ \\
               & $M_1=128m_2-120$ \\[4pt]
$\C{Z}(G)/\langle s\rangle\iso C_2\times C_{2^{m_2}}\times C_4\times C_4$ & $\B{Q}\colon 32m_2+8$ \\
               & $F\colon 64m_2-56$ \\[4pt]
$\langle\C{Z}(G),x\rangle/\langle s\rangle\iso C_2\times C_{2^{m_2+1}}\times C_4\times C_4$ & $\B{Q}\colon 32m_2+40$ ($N_2=32$) \\
               & $F\colon 64m_2+8$ ($M_2=64$)\\[4pt]
$\langle\C{Z}(G),y\rangle/\langle s\rangle\iso C_2\times C_{2^{m_2}}\times C_8\times C_4$ & $\B{Q}\colon 64m_2-24$ ($N_3=32m_2-32$) \\
               & $F\colon 128m_2-120$ ($M_3=64m_2-64$) \\[4pt]
$\langle\C{Z}(G),xy\rangle/\langle s\rangle\iso C_2\times C_{2^{m_2+1}}\times C_4\times C_4$ & $\B{Q}\colon 32m_2+40$ ($N_4=32$) \\
               & $F\colon 64m_2+8$ ($M_4=64$)\\[4pt]
$G/G'\iso C_2\times C_{2^{m_2+1}}\times C_8\times C_4$ & $N_0=64m_2+40$ \\
               & $M_0=128m_2+8$ \\ \hline
\end{tabular}
\end{center}
\bigskip

Since $N=N_1+2N_2+N_3=96m_2+40$ and $M=M_1+2M_2+M_3=192m_2-56$, $\B{Q}G$
is the direct sum of $N_0=64m_2+40$ fields
and $N-N_0=32m_2$
quaternion algebras while $FG$ is the direct sum of $M_0=128m_2+8$ fields
and $M-M_0=64m_2-64$ quaternion algebras.

\bigskip\par\noindent (xii.c) $m_1=2$, $m_2=2$, $m_3=2$, $m_4\ge3$, $s=t_1^2$.
\bigskip
\begin{center}
\small
\begin{tabular}{ll}
\multicolumn{1}{c}{Group} & \multicolumn{1}{c}{No. of \B{Q}- and $F$-classes} \\ \hline \\[-6pt]
$\C{Z}(G)\iso C_4\times C_4\times C_4\times C_{2^{m_4}}$ & $N_1=64m_4+8$ \\
               & $M_1=128m_2-120$ \\[4pt]
$\C{Z}(G)/\langle s\rangle\iso C_2\times C_4\times C_4\times C_{2^{m_4}}$ & $\B{Q}\colon 32m_4+8$ \\
               & $F\colon 64m_4-56$ \\[4pt]
$\langle\C{Z}(G),x\rangle/\langle s\rangle\iso C_2\times C_8\times C_4\times C_{2^{m_4}}$ & $\B{Q}\colon 64m_4-24$ ($N_2=32m_4-32$) \\
               & $F\colon 128m_4-120$ ($M_2=64m_4-64$)\\[4pt]
$\langle\C{Z}(G),y\rangle/\langle s\rangle\iso C_2\times C_4\times C_8\times C_{2^{m_4}}$ & $\B{Q}\colon 64m_4-24$ ($N_3=32m_4-32$) \\
               & $F\colon 128m_4-120$ ($M_3=64m_4-64$) \\[4pt]
$\langle\C{Z}(G),xy\rangle/\langle s\rangle\iso C_2\times C_4\times C_8\times C_{2^{m_4}}$ & $\B{Q}\colon 64m_4-24$ ($N_4=32m_4-32$) \\
               & $F\colon 128m_4-120$ ($M_4=64m_4-64$)\\[4pt]
$G/G'\iso C_2\times C_8\times C_8\times C_{2^{m_4}}$ & $N_0=128m_4-88$ \\
               & $M_0=256m_4-248$ \\ \hline
\end{tabular}
\end{center}
\bigskip

Since $N=N_1+3N_2=160m_4-88$ and $M=M_1+3M_2=320m_4-312$, $\B{Q}G$
is the direct sum of $N_0=128m_4-88$ fields
and $N-N_0=192m_4$
quaternion algebras while $FG$ is the direct sum of $M_0=256m_4-248$ fields
and $M-M_0=64m_4-64$ quaternion algebras.

\newpage
\par\noindent (xii.d) $m_1=2$, $m_2\ge3$, $m_3\ge3$, $m_4=2$, $s=t_1$.   Because
of the symmetry with respect to $x$ and $y$, we may assume that $m_2\ge m_3$.
Assume first, in fact, that $m_2>m_3$.
\bigskip

\begin{center}
\small
\begin{tabular}{ll}
\multicolumn{1}{c}{Group} & \multicolumn{1}{c}{No. of \B{Q}- and $F$-classes} \\ \hline \\[-6pt]
$\C{Z}(G)\iso C_4\times C_{2^{m_2}}\times C_{2^{m_3}}\times C_4$ & $N_1=2^{m_3+4}(3+m_2-m_3)-56$ \\
               & $M_1=2^{m_3+5}(3+m_2-m_3)-248$ \\[4pt]
$\C{Z}(G)/\langle s\rangle\iso C_2\times C_{2^{m_2}}\times C_{2^{m_3}}\times C_4$ & $\B{Q}\colon 2^{m_3+3}(3+m_2-m_3)-24$ \\
               & $F\colon 2^{m_3+4}(3+m_2-m_3)-120$ \\[4pt]
$\langle\C{Z}(G),x\rangle/\langle s\rangle\iso C_2\times C_{2^{m_2+1}}\times C_{2^{m_3}}\times C_4$ & $\B{Q}\colon 2^{m_3+3}(4+m_2-m_3)-24$ \\
                         & \qquad ($N_2=2^{m_3+3}$) \\
               & $F\colon 2^{m_3+4}(4+m_2-m_3)-120$ \\
                         & \qquad ($M_2=2^{m_3+4}$)\\[4pt]
$\langle\C{Z}(G),y\rangle/\langle s\rangle\iso C_2\times C_{2^{m_2}}\times C_{2^{m_3+1}}\times C_4$ & $\B{Q}\colon 2^{m_3+4}(2+m_2-m_3)-24$ \\
                         & \qquad ($N_3=2^{m_3+3}(1+m_2-m_3)$) \\
               & $F\colon 2^{m_3+5}(2+m_2-m_3)-120$ \\
                         & \qquad ($M_3=2^{m_3+4}(1+m_2-m_3)$) \\[4pt]
$\langle\C{Z}(G),xy\rangle/\langle s\rangle\iso C_2\times C_{2^{m_2+1}}\times C_{2^{m_3}}\times C_4$ & $\B{Q}\colon 2^{m_3+3}(4+m_2-m_3)-24$ \\
                         & \qquad ($N_4=2^{m_3+3}$) \\
               & $F\colon 2^{m_3+4}(4+m_2-m_3)-120$ \\
                         & \qquad($M_4=2^{m_3+4}$)\\[4pt]
$G/G'\iso C_2\times C_{2^{m_2+1}}\times C_{2^{m_3+1}}\times C_4$ & $N_0=2^{m_3+4}(3+m_2-m_3)-24$ \\
               & $M_0=2^{m_3+5}(3+m_2-m_3)-120$ \\ \hline
\end{tabular}
\end{center}
\bigskip

Since $N=N_1+2N_2+N_3=2^{m_3+3}(9+3m_2-3m_3)-56$ and $M=M_1+2M_2+M_3=2^{m_3+4}(9+3m_2-3m_2)-248$,
$\B{Q}G$ is the direct sum of $N_0=2^{m_3+4}(3+m_2-m_3)-24$ fields
and $N-N_0=2^{m_3+3}(3+m_2-m_3)-32$
quaternion algebras while $FG$ is the direct sum of $M_0=2^{m_3+5}(3+m_2-m_3)-120$ fields
and $M-M_0=2^{m_3+4}(3+m_2-m_3)-128$ quaternion algebras.

\newpage
Now assume $m_2=m_3$.
\bigskip
\begin{center}
\small
\begin{tabular}{ll}
\multicolumn{1}{c}{Group} & \multicolumn{1}{c}{No. of \B{Q}- and $F$-classes} \\ \hline \\[-6pt]
$\C{Z}(G)\iso C_4\times C_{2^{m_2}}\times C_{2^{m_2}}\times C_4$ & $N_1=2^{m_2+4}(3)-56$ \\
               & $M_1=2^{m_2+5}(3)-248$ \\[4pt]
$\C{Z}(G)/\langle s\rangle\iso C_2\times C_{2^{m_2}}\times C_{2^{m_2}}\times C_4$ & $\B{Q}\colon 2^{m_2+3}(3)-24$ \\
               & $F\colon 2^{m_2+4}(3)-120$ \\[4pt]
$\langle\C{Z}(G),x\rangle/\langle s\rangle\iso C_2\times C_{2^{m_2+1}}\times C_{2^{m_2}}\times C_4$ & $\B{Q}\colon 2^{m_2+3}(4)-24$ ($N_2=2^{m_2+3}$) \\
               & $F\colon 2^{m_2+4}(4)-120$ ($M_2=2^{m_2+4}$)\\[4pt]
$\langle\C{Z}(G),y\rangle/\langle s\rangle\iso C_2\times C_{2^{m_2}}\times C_{2^{m_2+1}}\times C_4$ & $\B{Q}\colon 2^{m_2+5}-24$ ($N_3=2^{m_2+3}$) \\
               & $F\colon 2^{m_2+6}-120$ ($M_3=2^{m_2+4}$) \\[4pt]
$\langle\C{Z}(G),xy\rangle/\langle s\rangle\iso C_2\times C_{2^{m_2+1}}\times C_{2^{m_2}}\times C_4$ & $\B{Q}\colon 2^{m_2+5}-24$ ($N_4=2^{m_2+3}$) \\
               & $F\colon 2^{m_2+6}-120$ ($M_4=2^{m_2+4}$)\\[4pt]
$G/G'\iso C_2\times C_{2^{m_2+1}}\times C_{2^{m_2+1}}\times C_4$ & $N_0=2^{m_2+4}(3)-24$ \\
               & $M_0=2^{m_2+5}(3)-120$ \\ \hline
\end{tabular}
\end{center}
\bigskip

Since $N=N_1+3N_2=2^{m_2+3}(9)-56$ and $M=M_1+3M_2=2^{m_2+4}(9)-248$, $\B{Q}G$
is the direct sum of $N_0=2^{m_2+4}(3)-24$ fields
and $N-N_0=2^{m_2+2}(3)-32$
quaternion algebras while $FG$ is the direct sum of $M_0=2^{m_2+5}(3)-120$ fields
and $M-M_0=2^{m_2+4}(3)-128$ quaternion algebras.

\bigskip\par\noindent (xii.e) $m_1=2$, $m_2\ge3$, $m_3=2$, $m_4\ge3$, $s=t_1$. (Conclusions
in the case $2,2,m_3\ge3,m_4\ge3$ are the same.)

Here there are two possibilities, the first being the case $m_2\ge m_4$.

\begin{center}
\small
\begin{tabular}{ll}
\multicolumn{1}{c}{Group} & \multicolumn{1}{c}{No. of \B{Q}- and $F$-classes} \\ \hline \\[-6pt]
$\C{Z}(G)\iso C_4\times C_{2^{m_2}}\times C_4\times C_{2^{m_4}}$ & $N_1=2^{m_4+4}(3+m_2-m_4)-56$ \\
               & $M_1=2^{m_4+5}(3+m_2-m_4)-248$ \\[4pt]
$\C{Z}(G)/\langle s\rangle\iso C_2\times C_{2^{m_2}}\times C_4\times C_{2^{m_4}}$ & $\B{Q}\colon 2^{m_4+3}(3+m_2-m_4)-24$ \\
               & $F\colon 2^{m_4+4}(3+m_2-m_4)-120$ \\[4pt]
$\langle\C{Z}(G),x\rangle/\langle s\rangle\iso C_2\times C_{2^{m_2+1}}\times C_4\times C_{2^{m_4}}$ & $\B{Q}\colon 2^{m_4+3}(4+m_2-m_4)-24$ \\
                     & \qquad ($N_2=2^{m_4+3}$) \\
               & $F\colon 2^{m_4+4}(4+m_2-m_4)-120$ \\
                     & \qquad ($M_2=2^{m_4+4}$)\\[4pt]
$\langle\C{Z}(G),y\rangle/\langle s\rangle\iso C_2\times C_{2^{m_2}}\times C_8\times C_{2^{m_4}}$ & $\B{Q}\colon 2^{m_4+4}(3+m_2-m_4)-88$ \\
                     & \qquad ($N_3=2^{m_4+3}(3+m_2-m_4)-64$) \\
               & $F\colon 2^{m_4+5}(3+m_2-m_4)-248$ \\
                     & \qquad ($M_3=2^{m_4+4}(3+m_2-m_4)+56$) \\[4pt]
$\langle\C{Z}(G),xy\rangle/\langle s\rangle\iso C_2\times C_{2^{m_2+1}}\times C_4\times C_{2^{m_4}}$ & $\B{Q}\colon 2^{m_4+3}(4+m_2-m_4)-24$ \\
                     & \qquad ($N_4=2^{m_4+3}$) \\
               & $F\colon 2^{m_4+4}(4+m_2-m_4)-120$ \\
                     & \qquad ($M_4=2^{m_4+4}$)\\[4pt]
$G/G'\iso C_2\times C_{2^{m_2+1}}\times C_8\times C_{2^{m_4}}$ & $N_0=2^{m_4+4}(4+m_2-m_4)-88$ \\
               & $M_0=2^{m_4+5}(4+m_2-m_4)-248$ \\ \hline
\end{tabular}
\end{center}
\medskip

Since $N=N_1+2N_2+N_3=2^{m_4+3}(11+3m_2-3m_4)-120$ and $M=M_1+2M_2+M_3=2^{m_4+4}(11+3m_2-3m_4)-192$,
$\B{Q}G$ is the direct sum of $N_0=2^{m_4+4}(4+m_2-m_4)-88$ fields
and $N-N_0=2^{m_4+3}(3+m_2-m_4)-32$
quaternion algebras while $FG$ is the direct sum of $M_0=2^{m_4+5}(4+m_2-m_4)-248$ fields
and $M-M_0=2^{m_4+4}(3+m_2-m_4)-104$ quaternion algebras.

\enlargethispage{2\baselineskip}
\medskip
Now assume $m_4>m_2$.

\begin{center}
\small
%
\end{center}
\medskip

Since $N=N_1+2N_2+N_3=2^{m_2+3}(11+5m_4-5m_2)-184$ and $M=M_1+2M_2+M_3=2^{m_2+4}(11+5m_4-5m_2)-376$,
$\B{Q}G$ is the direct sum of $N_0=2^{m_2+5}(2+m_4-m_2)-88$ fields
and $N-N_0=2^{m_2+3}(3+m_4-m_2)-96$
quaternion algebras while $FG$ is the direct sum of $M_0=2^{m_2+6}(2+m_4-m_2)-248$ fields
and $M-M_0=2^{m_2+4}(3+m_4-m_2)-128$ quaternion algebras.

\bigskip\par\noindent (xii.f) $m_1=1$, $m_2\ge3$, $m_3\ge3$, $m_4\ge3$, $s=t_1$.
Because of the symmetry in $x$ and $y$, we may assume that $m_2\ge m_3$, but this still leaves three
possibilities.  As before, we let $f(a,b,c)=\tfrac73(4^c-1)+3(2^c)(2^b-2^c)+(a-b)2^{b+c}+1$
denote the number of cyclic subgroups of the direct
product $C_{2^a}\times C_{2^b}\times C_{2^c}$, with $a\ge b\ge c$.  We also let
\begin{multline*}
g(a,b,c,d)  \\
=\tfrac{15}{7}(8^d-1)+\tfrac73(2^d)(4^c-4^d)+3(2^{c+d})(2^b-2^c)+(a-b)2^{b+c+d}+1
\end{multline*}
denote the number of cyclic subgroups in $C_{2^a}\times C_{2^b}\times C_{2^c}\times C_{2^d}$,
$a\ge b\ge c\ge d$.

\newpage
We investigate the first case, $m_2\ge m_3\ge m_4$, in two steps,
supposing first that $m_2>m_3\ge m_4$.
\bigskip
\begin{center}
\footnotesize
\begin{tabular}{ll}
\multicolumn{1}{c}{Group} & \multicolumn{1}{c}{No. of \B{Q}- and $F$-classes} \\ \hline \\[-6pt]
$\C{Z}(G)$ & $N_1=g(m_2,m_3,m_4,2)$ \\
  \quad $\iso C_4\times C_{2^{m_2}}\times C_{2^{m_3}}\times C_{2^{m_4}}$  & $M_1=2N_1-136$ \\[8pt]
$\C{Z}(G)/\langle s\rangle$ & $\B{Q}\colon 2f(m_2,m_3,m_4)$ \\
  \quad $\iso C_2\times C_{2^{m_2}}\times C_{2^{m_3}}\times C_{2^{m_4}}$  & $F\colon 4f(m_2,m_3,m_4)-72$ \\[8pt]
$\langle\C{Z}(G),x\rangle/\langle s\rangle$  & $\B{Q}\colon 2f(m_2+1,m_3,m_4)$ \\
  \quad $\iso C_2\times C_{2^{m_2+1}}\times C_{2^{m_3}}\times C_{2^{m_4}}$  & \quad ($N_2=2f(m_2+1,m_3,m_4)-2f(m_2,m_3,m_4)$) \\
               & $F\colon 4f(m_2+1,m_3,m_4)-72$ \\
                     & \quad ($M_2=2N_2$)\\[8pt]
$\langle\C{Z}(G),y\rangle/\langle s\rangle$  & $\B{Q}\colon 2f(m_2,m_3+1,m_4)$ \\
  \quad $\iso C_2\times C_{2^{m_2}}\times C_{2^{m_3+1}}\times C_{2^{m_4}}$  & \quad ($N_3=2f(m_2,m_3+1,m_4)-2f(m_2,m_3,m_4)$) \\
               & $F\colon 4f(m_2,m_3+1,m_4)-72$ \\
                     & \quad ($M_3=2N_3$) \\[8pt]
$\langle\C{Z}(G),xy\rangle/\langle s\rangle$ & $\B{Q}\colon 2f(m_2+1,m_3,m_4)$ \\
  \quad $\iso C_2\times C_{2^{m_2+1}}\times C_{2^{m_3}}\times C_{2^{m_4}}$ & \quad ($N_4=N_2$) \\
                & $F\colon 4f(m_2+1,m_3,m_4)-72$ \\
                     & \quad ($M_4=M_2$)\\[8pt]
$G/G'$ & $N_0=2f(m_2+1,m_3+1,m_4)$ \\
  \quad $\iso C_2\times C_{2^{m_2+1}}\times C_{2^{m_3+1}}\times C_{2^{m_4}}$                & $M_0=2N_0-72$ \\ \hline
\end{tabular}
\end{center}
\bigskip

Since $N=N_1+2N_2+N_3=g(m_2,m_3,m_4,2)+4f(m_2+1,m_3,m_4)+2f(m_2,m_3+1,m_4)-6f(m_2,m_3,m_4)
=2^{m_3+m_4}(18+6m_2-6m_3)-\tfrac83(4^{m_4}+5)$ and $M=M_1+2M_2+M_3=2N-136$,
$\B{Q}G$ is the direct sum of
\begin{equation*}
N_0=2^{m_3+m_4}(12+4m_2-4m_3)-\tfrac43(4^{m_4}+2)
\end{equation*}
fields and
\begin{equation*}
N-N_0=2^{m_3+m_4}(6+2m_2-2m_3)-\tfrac43(4^{m_4}+8)
\end{equation*}
quaternion algebras while $FG$ is the direct sum of $M_0=2N_0-72$ fields
and $M-M_0=2(N-N_0)-64$ quaternion algebras.

\newpage
Now suppose $m_2=m_3\ge m_4$.
\bigskip
\begin{center}
\footnotesize
\begin{tabular}{ll}
\multicolumn{1}{c}{Group} & \multicolumn{1}{c}{No. of \B{Q}- and $F$-classes} \\ \hline \\[-6pt]
$\C{Z}(G)$ & $N_1=g(m_2,m_2,m_4,2)$ \\
  \quad $\iso C_4\times C_{2^{m_2}}\times C_{2^{m_2}}\times C_{2^{m_4}}$  & $M_1=2N_1-136$ \\[8pt]
$\C{Z}(G)/\langle s\rangle$  & $\B{Q}\colon 2f(m_2,m_2,m_4)$ \\
  \quad $\iso C_2\times C_{2^{m_2}}\times C_{2^{m_2}}\times C_{2^{m_4}}$   & $F\colon 4f(m_2,m_2,m_4)-72$ \\[8pt]
$\langle\C{Z}(G),x\rangle/\langle s\rangle$ & $\B{Q}\colon 2f(m_2+1,m_2,m_4)$ \\
  \quad $\iso C_2\times C_{2^{m_2+1}}\times C_{2^{m_2}}\times C_{2^{m_4}}$   & \quad ($N_2=2f(m_2+1,m_2,m_4)-2f(m_2,m_2,m_4)$) \\
               & $F\colon 4f(m_2+1,m_2,m_4)-72$ \\
                     & \quad ($M_2=2N_2$)\\[8pt]
$\langle\C{Z}(G),y\rangle/\langle s\rangle$  & $\B{Q}\colon 2f(m_2+1,m_2,m_4)$ \\
  \quad $\iso C_2\times C_{2^{m_2}}\times C_{2^{m_2+1}}\times C_{2^{m_4}}$  & \quad ($N_3=N_2$) \\
               & $F\colon 4f(m_2+1,m_2,m_4)-72$ \\
                     & \quad ($M_3=2N_2$) \\[8pt]
$\langle\C{Z}(G),xy\rangle/\langle s\rangle$  & $\B{Q}\colon 2f(m_2+1,m_2,m_4)$ \\
  \quad $\iso C_2\times C_{2^{m_2+1}}\times C_{2^{m_2}}\times C_{2^{m_4}}$  & \quad ($N_4=N_2$) \\
                & $F\colon 4f(m_2+1,m_2,m_4)-72$ \\
                     & \quad ($M_4=M_2$)\\[8pt]
$G/G'$  & $N_0=2f(m_2+1,m_2+1,m_4)$ \\
  \quad $\iso C_2\times C_{2^{m_2+1}}\times C_{2^{m_2+1}}\times C_{2^{m_4}}$ & $M_0=2N_0-72$ \\ \hline
\end{tabular}
\end{center}
\bigskip

Here, we have $N=N_1+3N_2=g(m_2,m_2,m_4,2)+6f(m_2+1,m_2,m_4)-6f(m_2,m_2,m_4)
=18(2^{m_2+m_4})-\tfrac83(4^{m_4}+5)$ and $M=M_1+2M_2+M_3=2N-136$, so
$\B{Q}G$ is the direct sum of
\begin{equation*}
N_0=12(2^{m_2+m_4})-\tfrac43(4^{m_4}+2)
\end{equation*}
fields and
\begin{equation*}
N-N_0=6(2^{m_2+m_4})-\tfrac43(4^{m_4}+8)
\end{equation*}
quaternion algebras, while $FG$ is the direct sum of $M_0=2N_0-72$ fields
and $M-M_0=2(N-N_0)-64$ quaternion algebras.

\newpage
The case $m_4\ge m_2\ge m_3$ we again split into several subcases, considering the possibility
$m_4>m_2>m_3$ first.
\bigskip
\begin{center}
\footnotesize
\begin{tabular}{ll}
\multicolumn{1}{c}{Group} & \multicolumn{1}{c}{No. of \B{Q}- and $F$-classes} \\ \hline \\[-6pt]
$\C{Z}(G)$ & $N_1=g(m_4,m_2,m_3,2)$ \\
  \quad $\iso C_4\times C_{2^{m_2}}\times C_{2^{m_3}}\times C_{2^{m_4}}$  & $M_1=2N_1-136$ \\[8pt]
$\C{Z}(G)/\langle s\rangle$ & $\B{Q}\colon 2f(m_4,m_2,m_3)$ \\
 \quad $\iso C_2\times C_{2^{m_2}}\times C_{2^{m_3}}\times C_{2^{m_4}}$   & $F\colon 4f(m_4,m_2,m_3)-72$ \\[8pt]
$\langle\C{Z}(G),x\rangle/\langle s\rangle$ & $\B{Q}\colon 2f(m_4,m_2+1,m_3)$ \\
  \quad $\iso C_2\times C_{2^{m_2+1}}\times C_{2^{m_3}}\times C_{2^{m_4}}$  & \quad ($N_2=2f(m_4,m_2+1,m_3)-2f(m_4,m_2,m_3)$) \\
               & $F\colon 4f(m_4,m_2+1,m_3)-72$ \\
                     & \quad ($M_2=2N_2$)\\[8pt]
$\langle\C{Z}(G),y\rangle/\langle s\rangle$ & $\B{Q}\colon 2f(m_4,m_2,m_3+1)$ \\
  \quad $\iso C_2\times C_{2^{m_2}}\times C_{2^{m_3+1}}\times C_{2^{m_4}}$  & \quad ($N_3=2f(m_4,m_2,m_3+1)-2f(m_4,m_2,m_3)$) \\
               & $F\colon 4f(m_4,m_3+1,m_2)-72$ \\
                     & \quad ($M_3=2N_3$) \\[8pt]
$\langle\C{Z}(G),xy\rangle/\langle s\rangle$ & $\B{Q}\colon 2f(m_4,m_2+1,m_3)$ \\
  \quad $\iso C_2\times C_{2^{m_2+1}}\times C_{2^{m_3}}\times C_{2^{m_4}}$    & \quad ($N_4=N_2$) \\
                & $F\colon 4f(m_4,m_2+1,m_3)-72$ \\
                  & \quad ($M_4=M_2$)\\[8pt]
$G/G'$ & $N_0=2f(m_4,m_2+1,m_3+1)$ \\
  \quad $\iso C_2\times C_{2^{m_2+1}}\times C_{2^{m_3+1}}\times C_{2^{m_4}}$  & $M_0=2N_0-72$ \\ \hline
\end{tabular}
\end{center}
\bigskip

Since $N=N_1+2N_2+N_3=g(m_4,m_2,m_3,2)+4f(m_4,m_2+1,m_3)+2f(m_4,m_2,m_3+1)-6f(m_4,m_2,m_3)
=2^{m_2+m_3}(22+10m_4-10m_2)-\tfrac{20}{3}(4^{m_3}+2)$ and $M=M_1+2M_2+M_3=2N-136$,
$\B{Q}G$ is the direct sum of
\begin{equation*}
N_0=2^{m_2+m_3}(16+8m_4-8m_2)-\tfrac83[2(4^{m_3})+1]
\end{equation*}
fields and
\begin{equation*}
N-N_0=2^{m_2+m_3}(6+2m_4-2m_2)-\tfrac43(4^{m_3}+8)
\end{equation*}
quaternion algebras, while $FG$ is the direct sum of $M_0=2N_0-72$ fields
and $M-M_0=2(N-N_0)-64$ quaternion algebras.

\newpage
Now suppose $m_4>m_2=m_3$.
\bigskip
\begin{center}
\footnotesize
\begin{tabular}{ll}
\multicolumn{1}{c}{Group} & \multicolumn{1}{c}{No. of \B{Q}- and $F$-classes} \\ \hline \\[-6pt]
$\C{Z}(G)$ & $N_1=g(m_4,m_2,m_2,2)$ \\
   \quad $\iso C_4\times C_{2^{m_2}}\times C_{2^{m_2}}\times C_{2^{m_4}}$   & $M_1=2N_1-136$ \\[8pt]
$\C{Z}(G)/\langle s\rangle$  & $\B{Q}\colon 2f(m_4,m_2,m_2)$ \\
  \quad $\iso C_2\times C_{2^{m_2}}\times C_{2^{m_2}}\times C_{2^{m_4}}$    & $F\colon 4f(m_4,m_2,m_2)-72$ \\[8pt]
$\langle\C{Z}(G),x\rangle/\langle s\rangle$ & $\B{Q}\colon 2f(m_4,m_2+1,m_2)$ \\
  \quad $\iso C_2\times C_{2^{m_2+1}}\times C_{2^{m_2}}\times C_{2^{m_4}}$   & \quad ($N_2=2f(m_4,m_2+1,m_2)-2f(m_4,m_2,m_2)$) \\
               & $F\colon 4f(m_4,m_2+1,m_2)-72$ \\
                     & \quad ($M_2=2N_2$)\\[8pt]
$\langle\C{Z}(G),y\rangle/\langle s\rangle$ & $\B{Q}\colon 2f(m_4,m_2+1,m_2)$ \\
   \quad $\iso C_2\times C_{2^{m_2}}\times C_{2^{m_2+1}}\times C_{2^{m_4}}$  & \quad ($N_3=N_2$) \\
               & $F\colon 4f(m_4,m_2+1,m_2)-72$ \\
                     & \quad ($M_3=M_2$) \\[8pt]
$\langle\C{Z}(G),xy\rangle/\langle s\rangle$  & $\B{Q}\colon 2f(m_4,m_2+1,m_2)$ \\
  \quad $\iso C_2\times C_{2^{m_2+1}}\times C_{2^{m_2}}\times C_{2^{m_4}}$ & \quad ($N_4=N_2$) \\
                & $F\colon 4f(m_4,m_2+1,m_2)-72$ \\
                  & \quad ($M_4=M_2$)\\[8pt]
$G/G'$ & $N_0=2f(m_4,m_2+1,m_2+1)$ \\
   \quad $\iso C_2\times C_{2^{m_2+1}}\times C_{2^{m_2+1}}\times C_{2^{m_4}}$ & $M_0=2N_0-72$ \\ \hline
\end{tabular}
\end{center}
\bigskip

In this situation, we have  $N=N_1+3N_2=g(m_4,m_2,m_2,2)+6f(m_4,m_2+1,m_2)
-6f(m_4,m_2,m_2)=4^{m_2}(10m_4-10m_2)+\tfrac23[23(4^{m_2})-20]$ and $M=M_1+2M_2+M_3=2N-136$,
$\B{Q}G$ is the direct sum of
\begin{equation*}
N_0 = 2f(m_4,m_2+1,m_2+1)=4^{m_2}(8m_4-8m_2)+\tfrac43[8(4^{m2})-2]
\end{equation*}
fields and
\begin{equation*}
N-N_0=4^{m_2}(2m_4-2m_2)+\tfrac23[7(4^{m_2})-16]
\end{equation*}
quaternion algebras while $FG$ is the direct sum of $M_0=2N_0-72$ fields
and $M-M_0=2(N-N_0)-64$ quaternion algebras.

\newpage
Suppose $m_4=m_2>m_3$.
\bigskip
\begin{center}
\footnotesize
\begin{tabular}{ll}
\multicolumn{1}{c}{Group} & \multicolumn{1}{c}{No. of \B{Q}- and $F$-classes} \\ \hline \\[-6pt]
$\C{Z}(G)$ & $N_1=g(m_2,m_2,m_3,2)$ \\
   \quad $\iso C_4\times C_{2^{m_2}}\times C_{2^{m_3}}\times C_{2^{m_2}}$  & $M_1=2N_1-136$ \\[8pt]
$\C{Z}(G)/\langle s\rangle$ & $\B{Q}\colon 2f(m_2,m_2,m_3)$ \\
  \quad $\iso C_2\times C_{2^{m_2}}\times C_{2^{m_3}}\times C_{2^{m_2}}$  & $F\colon 4f(m_2,m_2,m_3)-72$ \\[8pt]
$\langle\C{Z}(G),x\rangle/\langle s\rangle$ & $\B{Q}\colon 2f(m_2+1,m_2,m_3)$ \\
  \quad $\iso C_2\times C_{2^{m_2+1}}\times C_{2^{m_3}}\times C_{2^{m_2}}$ & \quad ($N_2=2f(m_2+1,m_2,m_3)-2f(m_2,m_2,m_3)$) \\
               & $F\colon 4f(m_2+1,m_2,m_3)-72$ \\
                     & \quad ($M_2=2N_2$)\\[8pt]
$\langle\C{Z}(G),y\rangle/\langle s\rangle$ & $\B{Q}\colon 2f(m_2,m_2,m_3+1)$ \\
  \quad $\iso C_2\times C_{2^{m_2}}\times C_{2^{m_3+1}}\times C_{2^{m_2}}$  & \quad ($N_3=2f(m_2,m_2,m_3+1)-2f(m_2,m_2,m_3)$) \\
               & $F\colon 4f(m_2,m_2,m_3+1)-72$ \\
                     & \quad ($M_3=2N_3$) \\[8pt]
$\langle\C{Z}(G),xy\rangle/\langle s\rangle$ & $\B{Q}\colon 2f(m_2+1,m_2,m_3)$ \\
  \quad $\iso C_2\times C_{2^{m_2+1}}\times C_{2^{m_3}}\times C_{2^{m_2}}$  & \quad ($N_4=N_2$) \\
                & $F\colon 4f(m_2+1,m_2,m_3)-72$ \\
                  & \quad ($M_4=M_2$)\\[8pt]
$G/G'$ & $N_0=2f(m_2+1,m_2,m_3+1)$ \\
 \quad $\iso C_2\times C_{2^{m_2+1}}\times C_{2^{m_3+1}}\times C_{2^{m_2}}$  & $M_0=2N_0-72$ \\ \hline
\end{tabular}
\end{center}
\bigskip

Since $N=N_1+2N_2+N_3=g(m_2,m_2,m_3,2)+4f(m_2+1,m_2,m_3)+2f(m_2,m_2,m_3+1)-6f(m_2,m_2,m_3)
=22(2^{m_2+m_3})-\tfrac{20}{3}(4^{m_3}+2)$
and $M=M_1+2M_2+M_3=2N-136$, $\B{Q}G$ is the direct sum of
\begin{equation*}
N_0=2f(m_2+1,m_2,m_3+1)=16(2^{m_2+m_3})-\tfrac83[2(4^{m_3})+1]
\end{equation*}
fields and
\begin{equation*}
N-N_0=6(2^{m_2+m_3})-\tfrac43(4^{m_3}+8)
\end{equation*}
quaternion algebras while $FG$ is the direct sum of $M_0=2N_0-72$ fields
and $M-M_0=2(N-N_0)-64$ quaternion algebras.

\newpage
Suppose $m_4=m_2=m_3$.
\bigskip
\begin{center}
\footnotesize
\begin{tabular}{ll}
\multicolumn{1}{c}{Group} & \multicolumn{1}{c}{No. of \B{Q}- and $F$-classes} \\ \hline \\[-6pt]
$\C{Z}(G)$  & $N_1=g(m_2,m_2,m_2,2)$ \\ 
  \quad $\iso C_4\times C_{2^{m_2}}\times C_{2^{m_2}}\times C_{2^{m_2}}$ & $M_1=2N_1-136$ \\[8pt]
$\C{Z}(G)/\langle s\rangle$ & $\B{Q}\colon 2f(m_2,m_2,m_2)$ \\
 \quad $\iso C_2\times C_{2^{m_2}}\times C_{2^{m_2}}\times C_{2^{m_2}}$   & $F\colon 4f(m_2,m_2,m_2)-72$ \\[8pt]
$\langle\C{Z}(G),x\rangle/\langle s\rangle$ & $\B{Q}\colon 2f(m_2+1,m_2,m_2)$ \\
  \quad $\iso C_2\times C_{2^{m_2+1}}\times C_{2^{m_2}}\times C_{2^{m_2}}$ & \quad ($N_2=2f(m_2+1,m_2,m_2)-2f(m_2,m_2,m_2)$) \\
               & $F\colon 4f(m_2+1,m_2,m_2)-72$ \\
                     & \quad ($M_2=2N_2$)\\[8pt]
$\langle\C{Z}(G),y\rangle/\langle s\rangle$ & $\B{Q}\colon 2f(m_2+1,m_2,m_2)$ \\
  \quad $\iso C_{2^{m_2}}\times C_{2^{m_2+1}}\times C_{2^{m_2}}$   & \quad ($N_3=N_2$) \\
               & $F\colon 4f(m_2+1,m_2,m_2)-72$ \\
                     & \quad ($M_3=M_2$) \\[8pt]
$\langle\C{Z}(G),xy\rangle/\langle s\rangle$  & $\B{Q}\colon 2f(m_2+1,m_2,m_2)$ \\
   \quad $\iso C_2\times C_{2^{m_2+1}}\times C_{2^{m_2}}\times C_{2^{m_2}}$   & \quad ($N_4=N_2$) \\
                & $F\colon 4f(m_2+1,m_2,m_2)-72$ \\
                  & \quad ($M_4=M_2$)\\[8pt]
$G/G'$ & $N_0=2f(m_2+1,m_2+1,m_2)$ \\
  \quad $\iso C_2\times C_{2^{m_2+1}}\times C_{2^{m_2+1}}\times C_{2^{m_2}}$  & $M_0=2N_0-72$ \\ \hline
\end{tabular}
\end{center}
\bigskip

Here, we have $N=N_1+3N_2=g(m_2,m_2,m_2,2)+6f(m_2+1,m_2,m_2)-6f(m_2,m_2,m_2)
=\tfrac23[23(4^{m_2})-20$ and $M=M_1+2M_2+M_3=2N-136$,
$\B{Q}G$ is the direct sum of
\begin{equation*}
N_0=2f(m_2+1,m_2+1,m_2)=\tfrac83(4^{m_2+1}-1)
\end{equation*}
fields and
\begin{equation*}
N-N_0=\tfrac23[7(4^{m_2})-16]
\end{equation*}
quaternion algebras, while $FG$ is the direct sum of $M_0=2N_0-72$ fields
and $M-M_0=2(N-N_0)-64$ quaternion algebras.

\newpage
We investigate the final case, $m_2\ge m_4\ge m_3$, first supposing $m_2\ge m_4>m_3$.
\bigskip
\begin{center}
\footnotesize
\begin{tabular}{ll}
\multicolumn{1}{c}{Group} & \multicolumn{1}{c}{No. of \B{Q}- and $F$-classes} \\ \hline \\[-6pt]
$\C{Z}(G)$ & $N_1=g(m_2,m_4,m_3,2)$ \\
   \quad $\iso C_4\times C_{2^{m_2}}\times C_{2^{m_3}}\times C_{2^{m_4}}$  & $M_1=2N_1-136$ \\[8pt]
$\C{Z}(G)/\langle s\rangle$ & $\B{Q}\colon 2f(m_2,m_4,m_3)$ \\
  \quad $\iso C_2\times C_{2^{m_2}}\times C_{2^{m_3}}\times C_{2^{m_4}}$   & $F\colon 4f(m_2,m_4,m_3)-72$ \\[8pt]
$\langle\C{Z}(G),x\rangle/\langle s\rangle$ & $\B{Q}\colon 2f(m_2+1,m_4,m_3)$ \\
  \quad $\iso C_2\times C_{2^{m_2+1}}\times C_{2^{m_3}}\times C_{2^{m_4}}$   & \quad ($N_2=2f(m_2+1,m_4,m_3)-2f(m_2,m_4,m_3)$) \\
               & $F\colon 4f(m_2+1,m_4,m_3)-72$ \\
                     & \quad ($M_2=2N_2$)\\[8pt]
$\langle\C{Z}(G),y\rangle/\langle s\rangle$ b & $\B{Q}\colon 2f(m_2,m_4,m_3+1)$ \\
  \quad $\iso C_2\times C_{2^{m_2}}\times C_{2^{m_3+1}}\times C_{2^{m_4}}$ & \quad ($N_3=2f(m_2,m_4,m_3+1)-2f(m_2,m_4,m_3)$) \\
               & $F\colon 4f(m_2,m_4,m_3+1)-72$ \\
                     & \quad ($M_3=2N_3$) \\[8pt]
$\langle\C{Z}(G),xy\rangle/\langle s\rangle$  & $\B{Q}\colon 2f(m_2+1,m_4,m_3)$ \\
  \quad $\iso C_2\times C_{2^{m_2+1}}\times C_{2^{m_3}}\times C_{2^{m_4}}$    & \quad ($N_4=N_2$) \\
                & $F\colon 4f(m_2+1,m_4,m_3)-72$ \\
                  & \quad ($M_4=M_2$)\\[8pt]
$G/G'$ & $N_0=2f(m_2+1,m_4,m_3+1)$ \\
 \quad $\iso C_2\times C_{2^{m_2+1}}\times C_{2^{m_3+1}}\times C_{2^{m_4}}$ & $M_0=2N_0-72$ \\ \hline
\end{tabular}
\end{center}
\bigskip

Since $N=N_1+2N_2+N_3=g(m_2,m_4,m_3,2)+4f(m_2+1,m_4,m_3+1)+2f(m_2,m_4,m_3+1)
-6f(m_2,m_4,m_3)=2^{m_3+m_4}(22+6m_2-6m_4)-\tfrac{20}{3}(4^{m_3}+2)$ and $M=M_1+2M_2+M_3=2N-136$,
$\B{Q}G$ is the direct sum of
\begin{equation*}
N_0=2f(m_2+1,m_4,m_3+1)=2^{m_3+m_4}(16+4m_2-4m_4)+\tfrac83[2(4^{m_3})+1]
\end{equation*}
fields and
\begin{equation*}
N-N_0=2^{m_3+m_4}(14+4m_2-4m_4)-\tfrac{16}{3}(4^{m_3}+2)
\end{equation*}
quaternion algebras, while $FG$ is the direct sum of $M_0=2N_0-72$ fields
and $M-M_0=2(N-N_0)-64$ quaternion algebras.

\newpage
We conclude this long investigation of small cases with the
supposition $m_2>m_4=m_3$.
\begin{center}
\footnotesize
\begin{tabular}{ll}
\multicolumn{1}{c}{Group} & \multicolumn{1}{c}{No. of \B{Q}- and $F$-classes} \\ \hline \\[-6pt]
$\C{Z}(G)$ & $N_1=g(m_2,m_3,m_3,2)$ \\
   \quad $\iso C_4\times C_{2^{m_2}}\times C_{2^{m_3}}\times C_{2^{m_3}}$  & $M_1=2N_1-136$ \\[8pt]
$\C{Z}(G)/\langle s\rangle \iso C_2\times C_{2^{m_2}}\times C_{2^{m_3}}\times C_{2^{m_3}}$ & $\B{Q}\colon 2f(m_2,m_3,m_3)$ \\
               & $F\colon 4f(m_2,m_3,m_3)-72$ \\[8pt]
$\langle\C{Z}(G),x\rangle/\langle s\rangle$ & $\B{Q}\colon 2f(m_2+1,m_3,m_3)$ \\
  \quad $\iso C_2\times C_{2^{m_2+1}}\times C_{2^{m_3}}\times C_{2^{m_3}}$  & \quad ($N_2=2f(m_2+1,m_3,m_3)-2f(m_2,m_3,m_3)$) \\
               & $F\colon 4f(m_2+1,m_3,m_3)-72$ \\
                     & \quad ($M_2=2N_2$)\\[8pt]
$\langle\C{Z}(G),y\rangle/\langle s\rangle$ & $\B{Q}\colon 2f(m_2,m_3+1,m_3)$ \\
   \quad $\iso C_2\times C_{2^{m_2}}\times C_{2^{m_3+1}}\times C_{2^{m_3}}$  & \quad ($N_3=2f(m_2,m_3+1,m_3)-2f(m_2,m_3,m_3)$) \\
               & $F\colon 4f(m_2,m_3+1,m_3)-72$ \\
                     & \quad ($M_3=2N_3$) \\[8pt]
$\langle\C{Z}(G),xy\rangle/\langle s\rangle$  & $\B{Q}\colon 2f(m_2+1,m_3,m_3)$ \\
  \quad $\iso C_2\times C_{2^{m_2+1}}\times C_{2^{m_3}}\times C_{2^{m_3}}$  & \quad ($N_4=N_2$) \\
                & $F\colon 4f(m_2+1,m_3,m_3)-72$ \\
                  & \quad ($M_4=M_2$)\\[8pt]
$G/G'$  & $N_0=2f(m_2+1,m_3+1,m_3)$ \\
   \quad $\iso C_2\times C_{2^{m_2+1}}\times C_{2^{m_3+1}}\times C_{2^{m_3}}$ & $M_0=2N_0-72$ \\ \hline
\end{tabular}
\end{center}
\medskip

Since $N=N_1+2N_2+N_3=g(m_2,m_3,m_3,2)+4f(m_2+1,m_3+1,m_3)+2f(m_2,m_3+1,m_3)
-6f(m_2,m_3,m_3)=4^{m_3}(6m_2-6m_3)+\tfrac23[23(4^{m_3}-20)$ and $M=M_1+2M_2+M_3=2N-136$,
$\B{Q}G$ is the direct sum of
\begin{equation*}
N_0=2f(m_2+1,m_3+1,m_3)=4^{m_3}(4m_2-4m_3)+\tfrac83(4^{m_3+1}-1)
\end{equation*}
fields and
\begin{equation*}
N-N_0=4^{m_3}(2m_2-2m_3)+\tfrac23[7(4^{m_3})-16]
\end{equation*}
quaternion algebras while $FG$ is the direct sum of $M_0=2N_0-72$ fields
and $M-M_0=2(N-N_0)-64$ quaternion algebras.

\bigskip

Finally we move to the general situation where all $m_i\ge3$.
Recall that $G=G_0\times\langle t_4\rangle$ with $G_0$ a group
from class $\C{D}_5$. \lemref{lem:4cyclics} allows to us to determine the number
$N_1$ of \B{Q}-classes of $\C{Z}(G)\iso C_{2^{m_1}}\times C_{2^{m_2}}\times C_{2^{m_3}}
\times C_{2^{m_4}}$, while Lemmas \ref{lem:xclasses} and \ref{lem:4cyclics}
give us the numbers $N_2$, $N_3$ and $N_4$  of \B{Q}-classes involving $x$, $y$, and $xy$,
respectively.  Thus $\B{Q}G$ has $N=N_1+N_2+N_3+N_4$ simple components.
Now $G/G'=\langle \overline{t_1}\rangle\times\langle\overline{x}\rangle
\times\langle\overline{y}\rangle\times\langle\overline{t_4}\rangle
\iso C_{2^{m_1-1}}\times C_{2^{m_2+1}}\times C_{2^{m_3+1}}\times C_{2^{m_4}}$.
\lemref{lem:4cyclics} provides the formula for $N_0$, the number of cyclic subgroups
of this group,
which is also the number of fields in the commutative part of $\B{Q}G$.  We conclude that
$\B{Q}G$ is the direct sum of $N_0$ fields and $N-N_0$ quaternion
algebras.

Let $q\equiv3\pmod8$ be a positive integer and
let $F$ be the finite field of order $q$.  The cyclic subgroups of
the centre of $G$ that are generated by elements of order at most $4$
are contained in $C_4\times C_4\times C_4\times C_4$
and \lemref{lem:4cyclics} tells us there are $136$ of these.
By \lemref{lem1}, we conclude that
$136$ of the \B{Q}-classes of $\C{Z}(G)$ are also
$F$-classes while the remaining $N_1-136$ \B{Q}-classes split into two $F$-classes.
This gives $M_1=136+2(N_1-136)=2N_1-136$ $F$-classes.

Similar reasoning shows us that the numbers
of $F$-classes involving $x$, $y$ and $xy$ are, respectively, $M_2=2N_2$,
$M_3=2N_3$ and $M_4=2N_4$.
In all, $G$ contains
$M=M_1+M_2+M_3+M_4=2N-136$ $F$-classes, so  $FG$
is the direct sum of $2N-136$ simple components.
Of the $N_0$ \B{Q}-classes in $G/G'$, $136$ are also $F$-classes
while the remaining $N_0-136$ \B{Q}-classes
each split into two $F$-classes.  It follows that the commutative part of $FG$
is the direct sum of $M_0=136+2(N_0-136)=2N_0-136$ fields.

\begin{thm} Let $G_0$ be a group of type $\C{D}_5$ and let $G=G_0\times\langle t_4\rangle$
with a centre the direct product of four cyclic groups of orders
$2^{m_1}$, $2^{m_2}$, $2^{m_3}$, $2^{m_4}$,
respectively, with $m_1\ge3$, $m_2\ge3$, $m_3\ge3$, $m_4\ge3$.   Denote by $N_1$, $N_2$,
$N_3$,  $N_4$ the number of cyclic subgroups of \C{Z}(G) and the numbers
of \B{Q}-classes involving $x$, $y$ and $xy$, respectively.  (Note that these numbers
can be determined by \lemref{lem:4cyclics} and \lemref{lem:xclasses}.)
Let $N=N_1+N_2+N_3+N_4$ and let $N_0$ be the number of cyclic subgroups of $G/G'$.
Then $\B{Q}G$ is the direct sum of $N_0$ fields and $N-N_0$ quaternion algebras.
Let $K$ a finite field of order $q$.  Then there are least
$M=2N-136$ simple components in the Wedderburn decomposition of $KG$.
This minimal number is realized if $q\equiv3\pmod8$, in which case
$KG$ is the direct sum of $M_0=2N_0-136$  fields and $M-M_0=2(N-N_0)$
quaternion algebras, each necessarily a $2\times 2$ matrix ring.
\end{thm}


\providecommand{\bysame}{\leavevmode\hbox to3em{\hrulefill}\thinspace}
\providecommand{\MR}{\relax\ifhmode\unskip\space\fi MR }
\providecommand{\MRhref}[2]{%
  \href{http://www.ams.org/mathscinet-getitem?mr=#1}{#2}
}
\providecommand{\href}[2]{#2}

\end{document}